\documentclass[3p]{elsarticle}

\makeatletter
\def\ps@pprintTitle{%
    \let\@oddhead\@empty
    \let\@evenhead\@empty
    \let\@evenfoot\@oddfoot
    }
\makeatother
\usepackage{tabularx}
\usepackage{amsmath}
\usepackage{amssymb}
\usepackage{amsthm}
\usepackage{mathrsfs}
\usepackage{bm}
\usepackage{graphicx}
\usepackage[final]{hyperref}
\usepackage[ruled,vlined]{algorithm2e}
\usepackage{enumitem}
\usepackage{array}
\usepackage{multirow}
\usepackage{makecell}
\usepackage[title]{appendix}
\usepackage{empheq}
\usepackage{cleveref}
\usepackage{cases}
\usepackage{textcomp}
\usepackage{gensymb}
\usepackage{subcaption}
\hypersetup{
	colorlinks=true,       
	linkcolor=blue,        
	citecolor=blue,        
	filecolor=magenta,     
	urlcolor=blue         
}
\usepackage{float}
\usepackage{multirow}

\usepackage{regexpatch}

\usepackage{tcolorbox} 
\usepackage{xcolor} 
\newcommand{\vect}[1]{\boldsymbol{\mathbf{#1}}}
 
\newcommand{\R}{\mathbb{R}}

\newcommand{\norm}[1]{\left\lVert #1 \right\rVert}

\newcommand{\grad}{\nabla}

\DeclareMathOperator*{\esssup}{ess\,sup}

\newtheorem{theorem}{Theorem}

\newtheorem{corollary}[theorem]{Corollary}
\newtheorem{remark}{Remark}

\newcommand{\noise}{\boldsymbol{N}}

\newcommand{\bdmc}[1]{\boldsymbol{\mathcal{#1}}}
\newcommand{\as}{\text{a.s.}}
\newcommand{\expect}[2]{\mathbb{E}_{#1}\left[#2 \right]}

\newcommand{\dd}{\,\textup{d}}

\newcommand{\bzero}{\bm{0}}
\newcommand{\bcolon}{\bm{:}}

\newcommand{\be}{\bm{e}}

\newcommand{\bn}{\bm{n}}

\newcommand{\bt}{\bm{t}}
\newcommand{\bu}{\bm{u}}
\newcommand{\bv}{\bm{v}}
\newcommand{\bw}{\bm{w}}
\newcommand{\bx}{\bm{x}}
\newcommand{\by}{\bm{y}}

\newcommand{\bC}{\bm{C}}

\newcommand{\bF}{\bm{F}}

\newcommand{\bI}{\bm{I}}

\newcommand{\bN}{\bm{N}}

\newcommand{\bS}{\bm{S}}

\newcommand{\bX}{\bm{X}}
\newcommand{\bY}{\bm{Y}}

\newcommand{\bchi}{{\bm{\chi}}}

\newcommand{\calM}{{\mathcal{M}}}

\newcommand{\operator}{\widetilde{\mathcal{F}}}

\newcommand{\dualDot}[2]{\langle {#1}, {#2}\rangle}

\newcommand{\post}[1]{\nu_{M|\bY}\left(#1|\by^*\right)}
\newcommand{\postt}[1]{\widetilde{\nu_{M|\bY}}\left(#1|\by^*\right)}
\newcommand*\tcircle[1]{%
  \raisebox{-0.5pt}{%
    \textcircled{\fontsize{7pt}{0}\fontfamily{phv}\selectfont #1}%
  }%
}

\begin{document}
\begin{frontmatter}
\title{Residual-Based Error Correction for Neural Operator Accelerated Infinite-Dimensional Bayesian Inverse Problems}
\author{Lianghao Cao\corref{cor1}}
\ead{lianghao@oden.utexas.edu}
\author{Thomas O'Leary-Roseberry}
\ead{tom.olearyroseberry@utexas.edu}
\author{Prashant K. Jha}
\ead{prashant.jha@austin.utexas.edu}
\author{J. Tinsley Oden}
\ead{oden@oden.utexas.edu}
\author{Omar Ghattas}
\ead{omar@oden.utexas.edu}

\address{Oden Institute for Computational Sciences and Engineering, The University of Texas at Austin, 201 E. 24th Street, C0200, Austin, TX 78712, United States of America.
}

\cortext[cor1]{Corresponding authors}

\begin{abstract}
We explore using neural operators, or neural network representations of nonlinear maps between function spaces, to accelerate infinite-dimensional Bayesian inverse problems (BIPs) with models governed by nonlinear parametric partial differential equations (PDEs). Neural operators have gained significant attention in recent years for their ability to approximate the parameter-to-solution maps defined by PDEs using as training data solutions of PDEs at a limited number of parameter samples. The computational cost of BIPs can be drastically reduced if a large number of PDE solves required for posterior characterization are replaced with evaluations of trained neural operators. However, reducing error in the resulting BIP solutions via reducing the approximation error of the neural operators in training can be challenging and unreliable. We provide an \textit{a priori} error bound result that implies certain BIPs can be ill-conditioned to the approximation error of neural operators, thus leading to inaccessible accuracy requirements in training. To reliably deploy neural operators in BIPs, we consider a strategy for enhancing the performance of neural operators, which is to correct the prediction of a trained neural operator by solving a linear variational problem based on the PDE residual. We show that a trained neural operator with error correction can achieve a quadratic reduction of its approximation error, all while retaining substantial computational speedups of posterior sampling when models are governed by highly nonlinear PDEs. The strategy is applied to two numerical examples of BIPs based on a nonlinear reaction--diffusion problem and deformation of hyperelastic materials. We demonstrate that posterior representations of the two BIPs produced using trained neural operators are greatly and consistently enhanced by error correction.
\end{abstract}


\end{frontmatter}

\tableofcontents

\section{Introduction}\label{sec:intro}
Many mathematical models of physical systems are governed by parametric partial differential equations (PDEs), where the \textit{states} of the systems are described by spatially and/or temporally--varying functions of PDE solutions, such as the evolution of temperature fields modeled by the heat equation and material deformation modeled by the nonlinear elasticity equation. The parameters, such as thermal conductivity and Young's modulus, of these models, often characterize properties of the physical systems and cannot be directly determined; one has to solve \textit{inverse problems} for that purpose, where the parameters are inferred from discrete and noisy observations of the states. To account for uncertainties in observations and our prior knowledge of the parameters, represented by \textit{prior probability distributions}, in the solutions of inverse problems, they are often formulated via Bayes' rule, or as \textit{Bayesian inverse problems}, for which solutions are probability distributions of the parameters conditioned on the observations, or \textit{posterior probability distributions}. In some scenarios, our prior knowledge of the parameters requires them to be treated as functions, leading to \textit{infinite-dimensional Bayesian inverse problems}. These scenarios arise, for example, when the parameters are possibly spatially varying with uncertain spatial structures. Bayesian inverse problems are fundamental to constructing predictive models~\cite{Jaynes2003, Biegler2010, Oden2017, Ghattas2021}, and the need for inferring parameters as functions can be found in many areas of engineering, sciences, and medicine~\cite{wang2005using, Isaac2015, Zhu2016,  Alghamdi2020, Chen2021, liang2022bayesian}.

For models governed by large-scale highly nonlinear parametric PDEs, numerical simulations are computationally expensive as it involves solving high-dimensional linear systems in an iterative manner many times to obtain solutions with desired accuracy~\cite{Kelley1995}. In these cases, solving infinite-dimensional Bayesian inverse problems can be intractable, as numerically approximating infinite-dimensional posterior distributions with complex structures requires an untenable number of numerical solutions at different parameters, i.e., these problems suffer from the \textit{curse of dimensionality}. Many mathematical and numerical techniques are developed to mitigate the computational burden of these problems. Examples of these techniques are (i) advanced sampling methods exploiting the intrinsic low-dimensionality~\cite{cui2016dimension, constantine2016accelerating} or derivatives~\cite{stuart2004conditional, martin2012stochastic, bui2014solving} of posterior distributions, (ii) direct posterior construction and statistical computation via Laplace approximation~\cite{Bui2013,schillings2020on}, deterministic quadrature~\cite{schillings2013sparse, gantner2016computational}, or transport maps~\cite{parno2016multiscale, chen2019projected, zech2022sparse, Wang2021}, and (iii) surrogate modeling using polynomial approximation~\cite{marzourk2007stochastic,marzouk2009dimensionality} or model order reduction~\cite{galbally2010nonlinear, lieberman2010parameter,cui2015data} combined with multilevel or multifidelity methods~\cite{Dodwell2015, Teckentrup2015, Peherstorfer2018}.

Neural operators, or neural network representations of nonlinear maps between function spaces, have gained significant interest in recent years for their ability to represent the parameter-to-state maps defined by nonlinear parametric PDEs, and approximate these maps using a limited number of PDE solutions at samples of different parameters \cite{BhattacharyaHosseiniKovachkiEtAl2020,FrescaManzoni2022,KovachkiLiLiuEtAl2021,LiKovachkiAzizzadenesheliEtAl2020a,LiKovachkiAzizzadenesheliEtAl2020b,LuJinKarniadakis2019,OLearyRoseberryVillaChenEtAl2022,OLearyRoseberryDuChaudhuriEtAl2021,RaissiPerdikarisKarniadakis2019,WangWangPerdikaris2021,YuLuMengEtAl2022,Hesthaven2018}. Notable neural operators include POD-NN~\cite{Hesthaven2018}, DeepONet~\cite{LuJinKarniadakis2019}, Fourier neural operator~\cite{LiZhengKovachkiEtAl2021}, and derivative-informed reduced basis neural networks~\cite{OLearyRoseberryVillaChenEtAl2022}. The problem of approximating nonlinear maps is often referred to as the \textit{operator learning problem}, and numerically solving the operator learning problem by optimizing the neural network weights is referred to as \textit{training}. Neural operators are fast-to-evaluate and offer an alternative to the existing surrogate modeling techniques for accelerating the posterior characterization of infinite-dimensional Bayesian inverse problems by replacing the nonlinear PDE solves with evaluations of trained neural operators. We explore this alternative surrogate modeling approach using neural operators in this work.

The direct deployment of trained neural operators as surrogates of the nonlinear PDE-based model transfers most of the computational cost from posterior characterization to the offline generation of training samples and neural network training. Moreover, in contrast to some of the surrogate modeling approaches that approximate the parameter-to-observation or parameter-to-likelihood maps~\cite{marzourk2007stochastic, marko2021parallel}, neural operators approximate the parameter-to-state map, or \textit{learn the physical laws}. As a result, they can be used as surrogates for a class of different Bayesian inverse problems with models governed by the same PDEs but with different types of observations and noise models, thus further amortize the cost of surrogate construction.

While the drastic reduction of computational cost is advantageous, the accuracy of trained neural operators as well as the accuracy of the resulting posterior characterization produced by them needs to be examined. In theory, there are universal approximation results, such as those for DeepONet~\cite{LuJinKarniadakis2019}, Fourier neural operators~\cite{KovachkiLiLiuEtAl2021}, and reduced basis architectures \cite{BhattacharyaHosseiniKovachkiEtAl2020,OLearyRoseberryDuChaudhuriEtAl2021}, that imply the existence of neural operators that approximate a given nonlinear map between function spaces within certain classes arbitrarily well. In practice, however, constructing and training neural operators to satisfy a given accuracy can be challenging and unreliable. One often observes an empirical accuracy ceiling -- enriching training data and enhancing the representation power of network operators via increasing the inner-layer dimensions or the depth of neural networks, as often suggested by universal approximation theories, do not guarantee improved performance. In fact, in certain cases, increasing training data or depth of networks can lead to degraded performance. These behaviors are contrary to some other approximation methods, such as the finite element method with hp-refinement and surrogate modeling using polynomial approximation or model order reduction, for which theoretical results are well-connected to numerical implementation for controlling and reducing approximation errors~\cite{Babuska1994, demkowicz2006computing, ernst2012on, chen2017reduced}. The unreliability of neural operator performance improvement via training is a result of several confounding reasons that are discussed in this work. It is demonstrated via empirical studies in recent work by de Hoop et.\thinspace al~\cite{deHoop2022cost}, where neural operator performance, measured by their cost--accuracy trade-off, for approximating the parameter-to-state maps of various nonlinear parametric PDEs are provided.

The approximation error of a trained neural operator in the operator learning problem propagates to the error in the solutions of Bayesian inverse problems when the trained neural operator is employed as a surrogate. We demonstrate, through deriving an \textit{a priori} bound, that the approximation error of a trained neural operator controls the error in the posterior distributions defined using the trained neural operator. Additionally, the bounding constant shows that Bayesian inverse problems can be ill-conditioned to the approximation error of neural operators in many scenarios, such as when the prior is uninformative, data is high-dimensional, noise corruption is small, or the models are inadequate. Our theoretical result suggests that for many challenging Bayesian inverse problems, posing accuracy requirements on their solutions may lead to significantly tighter accuracy requirements in the training of neural operators that are practically inaccessible due to the limitation of neural operator training.

In this work, we consider a strategy for reliably deploying a trained neural operator as a surrogate in, but not limited to, infinite-dimensional Bayesian inverse problems. This strategy is inspired by a recent work by Jha and Oden~\cite{jha2022goal} on extending the goal-oriented \textit{a posteriori} error estimation techniques~\cite{oden2001goal, oden2002estimation, prudhomme1999goal, prudhomme2003computable, becker2001optimal, giles2002adjoint, pierce2000adjoint} to accelerate Bayesian calibration of high-fidelity models with a calibrated low-fidelity model. Instead of directly using the prediction of the trained neural operator at a given parameter for likelihood evaluation, we first solve a linear \textit{error correction} problem based on the PDE residual evaluated at the neural operator prediction and then use the obtained solution for likelihood evaluation. We show that solving this error-correction problem is equivalent to generating one Newton iteration under some mild conditions, and a trained neural operator with error correction can achieve global, i.e., over the prior distribution, quadratic error reduction when the approximation error of the trained neural operator is relatively small. We expect that the significant accuracy improvement of a trained neural operator from the error correction leads to vital accuracy improvement of the posterior characterization for challenging Bayesian inverse problems. The improvement in the accuracy of posterior characterization is achieved while retaining substantial computational speedups proportional to the expected number of iterative linear solves within a nonlinear PDE solve at parameters sampled from the posterior distribution,

To showcase the utility of the proposed strategy, two numerical examples are provided. In the first example, we consider the inference of an uncertain coefficient field in an equilibrium nonlinear reaction--diffusion problem with a cubic reaction term from discrete observations of the state. The second example concerns the inference of Young's modulus, as a spatially varying field, of a hyperelastic material from discrete observations of its displacement in response to an external force. For both examples, trained neural operators, despite reaching their empirical accuracy ceilings, fail to recover all distinctive features of the posterior predictive means, whereas the error-corrected neural operators are consistently successful in such tasks.

\subsection{Related works}

We next discuss some of the related works on error correction in surrogate modeling approaches for Bayesian inverse problems. To the best of our knowledge, the existing works mainly focus on building data-driven models of the approximation error of surrogate parameter-to-observation maps. The sampling-based techniques for error correction presented in these works are different from the residual-based approach proposed in this work. The term model error correction sometime refers to numerical methods for representing model inadequacy, which is beyond the scope of this work.

In the context of model order reduction, Arridge et.\thinspace al~\cite{arridge2006approximation} proposed an offline sampling approach for constructing a normal approximation for the joint probability distribution of the error in surrogate-predicted observations and the parameter over the prior distribution. The probability distribution of the error conditioned on the parameter can then be directly used for correcting likelihood evaluations defined using an additive Gaussian noise model. This approach simplifies the conditional dependence of the error on the parameter, leading to unreliable performance as pointed out by Manzoni et.\thinspace al~\cite{manzoni2016accurate}, who proposed two alternative error models: one based on radial basis interpolation and the other on linear regression models. Cui et.\thinspace al~\cite{cui2019posteriori} presented two methods for adaptively constructing error models during posterior sampling using delayed-acceptance Metropolis--Hastings: one is similar to that of Arridge et.\thinspace al but with posterior samples, and the other is a zeroth order error correction using the error evaluated at the current Markov chain position. 

Additionally, correcting errors in neural network surrogates is explored by Yan and Zhou~\cite{yan2019adaptive} for large-scale Bayesian inverse problems. They propose a strategy based on a predictor--corrector scheme using two neural networks. The predictor is a deep neural network surrogate of the parameter-to-observable map constructed offline. The corrector is a shallow neural network that takes the prediction of the surrogate as input and produces a corrected prediction. The corrector is trained using a few model simulations produced during posterior characterization.

\subsection{Layout of the paper}
The layout of the paper is as follows. In Section 2, infinite-dimensional Bayesian inverse problems and their numerical solutions are introduced in an abstract Hilbert space setting. In Section 3, the operator learning problem associated with neural operator approximation of nonlinear mappings in function spaces is introduced. The sources and reduction of approximation errors in neural network training are discussed. A result on \textit{a priori} bound of the error in the posterior distributions of the Bayesian inverse problem using the operator learning error is provided and interpreted. In Section 4, we introduce the residual-based error correction problem and discuss its conditional equivalency to a Newton-step problem. Then the error-corrected neural operator is proposed, and computational cost analysis for its use as a surrogate for posterior sampling is provided. Connections of the error-correction problem to goal-oriented \textit{a posteriori} error estimation techniques are also taken up in the same section. In Section 5, the physical, mathematical, and numerical settings for the two numerical examples of infinite-dimensional Bayesian inverse problems are provided. The empirical accuracy of neural operators and error-corrected neural operators at different sizes of training data is presented. Posterior mean estimates generated by the model, trained neural operators, and neural operators with error correction are visualized and examined to understand the accuracy of posterior sampling. The results of empirical and asymptotic cost analysis for the posterior sampling are also showcased. The concluding remarks are given in Section 6.
\section{Preliminaries}\label{sec:prelim}
In this section, we introduce infinite-dimensional Bayesian inverse problems in an abstract Hilbert space setting. We refer to~\cite{Stuart2010, Bui2013, Petra2014} and references therein for a more detailed analysis and numerical implementation of infinite-dimensional Bayesian inverse problems. For general treatments of Bayesian inference problems, see~\cite{gelman2014bayesian, robert2004monte}. For a reference on the theory of probability in infinite-dimensional Hilbert spaces, see~\cite{Prato2006}.

\subsection{Models governed by parametric partial differential equations}

Consider a mathematical model that predicts the state $u\in\mathcal{U}$ of a physical system given a parameter $m\in\mathcal{M}$. We assume that the model is governed by partial differential equations (PDEs), and $\mathcal{U}$ and $\mathcal{M}$ are infinite-dimensional separable real Hilbert spaces endowed with inner products $(\cdot,\cdot)_{\mathcal{U}}$ and $(\cdot,\cdot)_{\mathcal{M}}$, respectively. The state space $\mathcal{U}$ is a Sobolev space defined over a bounded, open, and sufficiently regular spatial domain $\Omega_u\subset\R^{3}$. It is either consists of functions with ranges in a vector space of dimension $d_s\leq 3$, such as $H^1(\Omega_u;\R^{d_s})$, or time-evolving functions, such as $L^2(0, T;H^1(\Omega_u; \R^{d_s}))$ with $T>0$. The former is appropriate for boundary value problems (BVPs), while the latter is appropriate for initial and boundary value problems (IBVPs). We assume $\mathcal{M}$ consists of spatially-varying scalar-valued functions defined over a set $\Omega_m\subset \overline{\Omega_u}$. The parameter $m$ may appear in boundary conditions, initial conditions, forcing terms, or coefficients of the PDEs.

We specify the model as an abstract nonlinear variational problem as follows. Let $\mathcal{U}_0\subseteq\mathcal{U}$ be a closed subspace that satisfies the homogenized strongly-enforced boundary and initial conditions of the PDEs. Let the solution set $\mathcal{V}_u\subseteq\mathcal{U}$ be an affine space of $\mathcal{U}_0$ that satisfies the strongly-enforced boundary conditions and initial conditions that possibly depend on $m$. The abstract nonlinear variational problem can be written as,
\begin{equation}\label{eq:residual}
    \text{Given } m\in\mathcal{M},\text{ find } u\in \mathcal{V}_u\text{ such that }\quad \mathcal{R}(u,m) = 0\quad \in \mathcal{U}_0^*\,,
\end{equation}
where $\mathcal{R} = \mathcal{R}(u, m)$ is a residual operator associated with the variational form, and $\mathcal{U}_0^*$ is the dual space of the space of test functions $\mathcal{U}_0$.
We assume that the residual operator is possibly nonlinear with respect to both the parameter and state, and the nonlinear variational problem has a unique solution for any $m\in\mathcal{\mathcal{M}}$. As a result, we can define a solution operator $\mathcal{F}:\mathcal{M}\to\mathcal{V}_u$, or the \textit{forward operator}, of the model, i.e.,
\begin{equation}
    \mathcal{R}(\mathcal{F}(m), m) \equiv 0\quad\forall m\in\mathcal{M}\,.
\end{equation}

\subsection{Infinite-dimensional Bayesian inverse problems}
Let $\by\in\R^{n_y}$ denote a set of discrete and noisy observations of the physical system described by the state $u\in\mathcal{U}$. We assume that the state $u$ and observations $\by$ are connected via a possibly nonlinear observation operator $\bdmc{B}:\mathcal{U}\to\R^{n_y}$ and a linear additive noise model\footnotemark\footnotetext{Alternatives to the linear additive noise model, such as a multiplicative noise model or a mixture of both, do not affect the error correction approach introduced in this work. See, e.g.,~\cite{Dunlop2019} for investigations of alternative noise models.},
\begin{equation}
    \by = \bdmc{B}(u) + \boldsymbol{n}\,,
\end{equation}
where $\boldsymbol{n}$ is an unknown noise vector that corrupts the observed data. We assume it is a realization of a random vector $\boldsymbol{N}$ with a probability distribution of $\nu_{\noise}$ and density of $\pi_{\noise}$.

Under the Bayesian framework of inverse problems, the model parameter is considered epistemically uncertain and a $\mathcal{M}$-valued random function denoted by $M$. Its probability distribution $\nu_M$, called \textit{the prior distribution}, incorporates our prior knowledge of the parameter. This leads to the following data model,
\begin{equation}\label{eq:data_model}
    \bY = (\bdmc{B}\circ\mathcal{F})(M) + \boldsymbol{N}\,,\quad M\sim\nu_M,\quad \noise\sim\nu_{\noise}\,,
\end{equation}
where the data is considered a random vector $\bY$ due to the influence of the measurement uncertainty and parameter uncertainty, represented by $\noise$ and $M$, respectively. In the context of the Bayesian inverse problem, the forward operator $\mathcal{F}$ is also referred to as the \textit{parameter-to-state map}.

Given a particular set of observed data $\by^*$, the goal of the Bayesian inverse problem is to construct or sample from the distribution of the parameter conditioned on the observed data $\by^*$, or the \textit{posterior distribution}, denoted by $\nu_{M|\bY}(\cdot|\by^*)$. The posterior and prior are related via the \textit{likelihood function}, $\mathcal{L}(\cdot;\by^*):\mathcal{M}\to\R_+$, according to \textit{Bayes' rule},
\begin{equation}\label{eq:bayes_rule}
    \frac{d\nu_{M|\bY}(\cdot|\by^*)}{d\nu_{M}}(m) = \frac{1}{Z(\by^*)}\underbrace{\pi_{\noise}\left(\by^*-(\bdmc{B}\circ\mathcal{F})(m)\right)}_{\eqqcolon\mathcal{L}(m;\boldsymbol{y}^*)}\quad\as\,,
\end{equation}
where $ d\nu_{M|\bY}(\cdot|\by^*)/d\nu_{M}$ is the Radon–Nikodym derivative of the posterior distribution with respect to the prior distribution, and $Z(\by^*) \coloneqq \expect{M\sim \nu_M}{\mathcal{L}(M;\by^*)}$ is the marginal likelihood or model evidence. The likelihood function evaluated at $m\in\mathcal{M}$ returns the probability of observing $\by^*$ at $m$ under the assumptions of the data model in~\eqref{eq:data_model}. Each evaluation of the likelihood function requires solving the model, i.e., evaluating the forward operator. The resulting posterior distribution encodes the additional knowledge of the parameter based on the information of the physical system contained in the observed data.

The prior is often defined as a Gaussian measure $\nu_M\coloneqq\mathcal{N}(m_{\text{pr}}, \mathcal{C}_{\text{pr}})$, where $m_{\text{pr}}\in\mathcal{M}$ is the mean and $\mathcal{C}_{\text{pr}}:\mathcal{M}\times\mathcal{M}\to\R$ is a covariance operator. The covariance operator can be defined using an inverse elliptic operator with a Robin boundary condition, which can be expressed in the strong form as
\begin{equation}\label{eq:gaussian_prior}
    \mathcal{C}_{\text{pr}} = \begin{cases}
    (-\alpha\nabla\cdot \boldsymbol{\Theta}\nabla + \beta)^{-d} &\quad \text{ in }\Omega\,, \\
    \boldsymbol{\Theta}n\cdot\grad + \gamma &\quad \text{ on }\partial\Omega\,,
    \end{cases}
\end{equation}
where $n$ is the outward normal vector, and the negative exponent $d$ is chosen sufficiently large to ensure bounded pointwise variance of the covariance operator and the well-posedness of the Bayesian inverse problem. The hyperparameters of the prior, namely, $\alpha, \beta, d, \boldsymbol{\Theta}, \gamma$, control its properties. For $\Omega_m$ with a dimension larger than one, i.e., $M$ is a Gaussian random field, the orthonormal matrix $\boldsymbol{\Theta}$ controls the anisotropy, and if $\boldsymbol{\Theta}$ is the identity matrix, one has isotropic random fields. The parameters $\alpha$ and $\beta$ together control the pointwise variance and correlation length of the random function. Finally, the value of $\gamma\propto\sqrt{\alpha\beta}$ is often chosen to minimize boundary artifacts.

\begin{remark}
Here we make a distinction between the physical parameter, denoted as $p\in\mathcal{M}$, and the model parameters $m\in\mathcal{M}$ that is ignored in the description above. While the physical parameter is often the target for inversion, they may have additional constraints so that the solution operator defined by the map between the physical parameter and model solution is only well-defined in a subset $\mathcal{D}\subset\mathcal{M}$. If the uncertain physical parameter $P\sim\nu_P = \mathcal{N}(m_p,\mathcal{C}_p)$ follows a Gaussian distribution and is used as the prior distribution in~\eqref{eq:bayes_rule}, then the condition $\nu_P(\mathcal{M}\setminus\mathcal{D})=0$ for the well-posedness of Bayesian inverse problems might not be satisfied. If this is the case, such constraints can be enforced by introducing a deterministic coupling of the physical parameter with the model parameter using a smooth function $\psi$ such that $\psi({\mathcal{M}})\subseteq\mathcal{D}$ and
\begin{equation}
    P = \psi(M)\,,\quad M\sim\mathcal{N}(m_{\text{pr}}, \mathcal{C}_{\text{pr}})\,.
\end{equation}
For example, if $P$ is strictly positive, such as the thermal conductivity field and Young's modulus in Section 5, then $\psi(\cdot) := \exp(\cdot)$ is typically used. Under this assumption, we can formulate Bayesian inverse problems based on a Gaussian prior distribution that is well-understood and relatively easy to implement. On the other hand, it is often the case that one has prior knowledge on $P$ but not $M$, thus this assumption requires designing suitable $\psi$, $m_{\text{pr}}$, and $\mathcal{C}_{\text{pr}}$ to reflect our uncertainty in $P$.
\end{remark}

\subsection{Numerical solutions of Bayesian inverse problems}

We consider a finite-dimensional approximation of the abstract Bayesian inverse problem introduced above. We start from a finite-dimensional approximation of the function spaces $\mathcal{U}$ and $\mathcal{M}$, denoted by $\mathcal{U}^h\subset \mathcal{U}$ and $\mathcal{M}^h\subset\mathcal{M}$, respectively. The finite-dimensional approximation of $\mathcal{M}$ is accomplished by a Galerkin approximation using a set of basis functions $\{\psi_j\in\mathcal{M}\}_{j=1}^{d_m}$. For the state $u$ defined over $\Omega_u$, e.g., $H^1(\Omega_u;\R^{d_s})$, we assume a similar approximation using a set of basis functions $\{\phi_j\in\mathcal{U}\}_{j=1}^{d_g}$ with typically a consistent discretization. For a time-evolving Sobolev space, additionally, a time discretization $0 = t_1< t_2< \dots < t_{d_t}=T$ is required and a discrete-time integration rule needs to be specified. We use $d_u = d_gd_s$ for BVPs and $d_u = d_gd_td_s$ for IBVPs to represent the total number of degrees-of-freedom of $\mathcal{U}^h$.

The residual $\mathcal{R}(u, m)$ can be estimated using numerical techniques such as the finite element method, with particular choices of the basis, test space discretization, as well as functional evaluations. The forward operator evaluations can then be numerically computed using iterative methods for nonlinear equations, such as fixed point iteration. We denote the forward operator associated with a computer simulation of the model as $\mathcal{F}^h:\mathcal{M}^h\to\mathcal{U}^h$.

The nonlinear Bayesian inverse problem can be solved numerically with a combination of the approximated forward operator $\mathcal{F}^h$, the prior $\nu_M^h$ represented in the finite-dimensional bases, i.e., its samples are in $\calM^h$, and a method for sampling from the posterior distribution $\nu^h_{M|\boldsymbol{Y}}(|\boldsymbol{y}^*)$ represented in the finite-dimensional bases. For generating numerical results later in Section 5, we restrict ourselves to sampling methods that require the evaluation of the likelihood function but not the gradient and Hessian estimates of the posterior that often yield increased computational efficiency\footnotemark. Here we consider a version of the Markov chain Monte Carlo (MCMC) method called the \textit{preconditioned Crank--Nicolson} (pCN) algorithm~\cite{Cotter2013}, a dimension-independent method that is applicable when the likelihood function is of the form:
\footnotetext{We acknowledge that this is a strong restriction for infinite-dimensional Bayesian inverse problems. The restriction is due to the limitation of neural operators in a general setting. In most cases, the loss function of neural operator training is uninformed of the gradient of forward operator, leading to poor performances in gradient estimations. See~\cite{OLearyRoseberryChenVillaEtAl2022} for how to incorporate high-dimensional derivative information in neural operator training.}
\begin{equation}\label{eq:pCN_likelihood}
    \mathcal{L}(m;\boldsymbol{y}^*) \propto \exp(-\Phi(m))\,,
\end{equation}
where $\Phi(\cdot):\mathcal{M}\to\R_+$ is referred to as the \textit{potential}. One of the examples of such a likelihood function is that of normally distributed noise with $\nu_{\boldsymbol{N}} = \mathcal{N}(\vect{0}, \boldsymbol{C}_{\noise})$ in which case we have
\begin{equation}\label{eq:gaussian_noise}
    \Phi(m) = \frac{1}{2}\norm{\boldsymbol{y}^* - (\bdmc{B}\circ\mathcal{F})(m)}^2_{\boldsymbol{C}_{\noise}^{-1}}\,,
\end{equation}
where $\norm{\cdot}_{\boldsymbol{C}_{\noise}^{-1}}$ denotes the discrete $l^2$-norm weighted by the inverse of the noise covariance matrix $\boldsymbol{C}_{\noise}^{-1}$. The pCN algorithm is described in Algorithm~\ref{al:mcmc}.

\begin{algorithm}
\SetAlgoLined
\KwResult{A Markov chain $\{m_j\in\mathcal{M}\}_{j=1}^{n_\text{chain}}$ with an invariant measure of $\nu_{M|\boldsymbol{Y}}(\cdot|\boldsymbol{y}^*)$ defined by~\eqref{eq:bayes_rule}.}
\SetKwInOut{Input}{Input}
\Input{(1) an initial guess $m_0\in \nu_{M}$, (2) a mixing parameter $\beta_{\text{pCN}}\in (0,1)$}
$k = 0$\;
$\Phi_0 = \Phi(m_0;\boldsymbol{y}^*)$\;
\While{$k<n_{\text{chain}}$}{
Sample from prior, $\hat{m}\sim \nu_{M}$\;
Generate the proposal parameter, $m_{\text{p}} = {\sqrt {1-\beta_{\text{pCN}}^{2}}}m_{k}+\beta_{\text{pCN}}\hat{m}$\;
Evaluate the potential, $\Phi_p = \Phi(m_p)$\;
\uIf{$\exp(\Phi_k-\Phi_p)\geq r\sim \text{Uniform}([0,1])$}
{
$m_{k+1} = m_{\text{p}}$\;
$\Phi_{k+1} = \Phi_{\text{p}}$\;
}
\Else{
$m_{k+1} = m_k$\;
$\Phi_{k+1} = \Phi_k$\;
}
$k \leftarrow k+1$\;
}
\caption{The preconditioned Crank--Nicolson (pCN) algorithm.
}
\label{al:mcmc}
\end{algorithm}

MCMC algorithms such as pCN generate Markov chains $\{m_j\in\mathcal{M}^h\}_{j=1}^{n_{\text{chain}}}$ that are used for further analysis. A large portion of the chains must be discarded, or ``burned", as they are influenced by initial samples of the chains that might be far from the support of the posterior. Second, the rest of the samples along each chain are correlated to certain degrees, leading to a much smaller effective sample size, i.e., the number of independently and identically distributed (i.i.d.) samples from the posterior distribution. They are typically much smaller than the actual sample size in the burned Markov chains. For problems with highly localized posterior and highly nonlinear PDEs, the computational cost, measured by the total number of iterative solves for forward operator evaluations, for generating a given number of effective samples from the posterior can be intractably high if performed without introducing advanced algorithmic or numerical techniques.
\section{Neural operators and approximation errors}\label{sec:neural_operators}

The high computational cost of infinite-dimensional Bayesian inverse problems motivates the development of \textit{surrogates} of the forward operator $\mathcal{F}$ that are fast-to-evaluate and constructed offline, i.e., prior to receiving observation data and posterior sampling. Employing surrogate forward operators may lead to a significant speedup of posterior sampling, yet typically leads to a trade-off in the accuracy of posterior representations.

In what follows, we introduce the operator learning problems associated with the learning of nonlinear maps between function spaces $\mathcal{M}$ and $\mathcal{U}$ using neural networks, or \textit{neural operators}, which is appropriate for models based on PDEs introduced in Section 2. We review some theoretical and practical aspects of neural operators and their training, focusing on sources and reduction of approximation errors of neural operators in the operator learning setting as well as the propagation of the approximation error in Bayesian inverse problems.

\subsection{Operator learning with neural networks}
We consider a general class of \textit{operator learning problems}, where we wish to determine a nonlinear map $\operator_{\boldsymbol{w}}:\mathcal{M}\to\mathcal{U}$ parameterized by $\boldsymbol{w}\in\mathcal{W}\subseteq\R^{d_w}$ such that $\operator_{\boldsymbol{w}}$ is \textit{closest} to a forward operator $\mathcal{F}$ via the following optimization problem:
\begin{equation}
    \inf_{\boldsymbol{w}\in\mathcal{W}} \mathcal{J}(\boldsymbol{w}) \coloneqq \norm{ \mathcal{F}- \operator_{\boldsymbol{w}}}_{\mathcal{T}}\,,
\end{equation}
where $\mathcal{T}$ is the suitable Bochner space for the forward operator $\mathcal{F}$. It is often defined as
\begin{equation}
    \mathcal{T}\coloneqq  L^p(\mathcal{M},\nu_M;\mathcal{U})\,,\quad \norm{\mathcal{G}}_{L^p(\mathcal{M},\nu_M;\mathcal{U})} = 
    \begin{cases}
        \left(\mathbb{E}_{M\sim\nu_M} \left[ \norm{\mathcal{G}(M)}^p_{\mathcal{U}}\right]\right)^{1/p}\,, & p\in[1,\infty)\,,\\
        \esssup_{M\sim\nu_M} \norm{\mathcal{G}(M)}_{\mathcal{U}}\,, & p=\infty\,,
    \end{cases}
\end{equation}
where the choice of $p$ depends on the regularity of the forward operator $\mathcal{F}$. The choice $p=2$ is often taken in practice, and additional learning of derivatives may be included via generalizing to $\mathcal{T} = W^{1,p}(\mathcal{M},\nu_M;\mathcal{U})$, as in~\cite{OLearyRoseberryChenVillaEtAl2022}. 

A neural operator approximates a nonlinear map, such as the forward operator $\mathcal{F}$ for PDEs, by ``learning" a neural network representation of the map. The neural network takes as its inputs a finite-dimensional representation of the parameter $m$, i.e., the degrees-of-freedom of $\mathcal{M}^h$, and produces output as a finite-dimensional representation of $\operator_{\boldsymbol{w}}(m)\in\mathcal{U}^h$ through a sequence of compositions of nonlinear functions, i.e., activation functions, acting on affine operations. The coefficient arrays in the affine operations correspond to the neural network weights $\boldsymbol{w}$, which are found by solving the operator learning problem, or \textit{training}, to be specified later. The choice of affine layers and activation functions decides the architecture of the neural operator. Many different architectures are used in practice; typical choices for the affine operations include fully-connected dense layers, residual neural networks (ResNets), and convolution layers, while typical activation functions include ReLU, sigmoid, tanh, softplus, etc. Neural operators can include additional operations such as Fourier transforms \cite{LiKovachkiAzizzadenesheliEtAl2020a}.

Many classes of neural networks are known to be universal approximators of different classes of functions. Universal approximation results for neural operators in \cite{BhattacharyaHosseiniKovachkiEtAl2020,KovachkiLiLiuEtAl2021,LuJinKarniadakis2019,OLearyRoseberryDuChaudhuriEtAl2021} establish the following result that we state in a formal and generalized way. For a nonlinear mapping $\mathcal{G}$ belong to some subset of a Bochner space $(\mathcal{T},\norm{\cdot}_{\mathcal{T}})$ and any desired error tolerance $\epsilon>0$, there exists a neural network architecture (e.g. number of layers, breadth, etc.) with weights $\boldsymbol{w}^\dagger$ that defines an operator $\operator_{\boldsymbol{w}^{\dagger}}$ such that 
\begin{equation}\label{:approximator}
    \norm{\mathcal{G} - \operator_{\boldsymbol{w}^\dagger}}_{\mathcal{T}} < \epsilon\,.
\end{equation}
Such results establish that complex high-dimensional parametric maps can, \textit{in theory}, be learned directly via neural networks arbitrarily well.

For a candidate neural network architecture, the loss function $\mathcal{J}$ can be numerically estimated via sample average approximation, leading to the following empirical risk minimization problem, or \textit{neural operator training problem}, for finite $p$:
\begin{equation}\label{eq:emp_risk_min}
    \min_{\boldsymbol{w}\in\mathcal{W}} \widetilde{\mathcal{J}}(\boldsymbol{w};\{m_j\}_{j=1}^{n_\text{train}}) \coloneqq \frac{1}{n_{\text{train}}}\sum_{j=1}^{n_{\text{train}}}\norm{u_j - \operator_{\boldsymbol{w}}(m_j;\boldsymbol{w})}_{\mathcal{U}}^p\,,\quad \left\{\left(m_j, u_j = \mathcal{F}^h(m_j)\right)|m_j\sim \nu_{M}^h\right\}_{j =1}^{n_{\text{train}}}\,,
\end{equation}
where samples of training data are generated prior to the training. In some settings, one may be able to incorporate additional information via physics constraints for the optimization problem, perhaps making the optimization problem more difficult \cite{LiZhengKovachkiEtAl2021,RaissiPerdikarisKarniadakis2019,WangWangPerdikaris2021,YuLuMengEtAl2022}, but more informed.

For many models for multiscale complex systems, fine discretization may be required to resolve necessary physics, which leads to large costs in evaluating the forward operator. In this case, a fundamental issue in neural operator learning arises: one is faced with learning very high-dimensional nonlinear maps from limited training data. This issue makes neural operator learning fundamentally different from typical machine learning due to the limitations of sparse training data. In order to address this issue, the input and output spaces of neural operators are often restricted to some finite-dimensional reduced bases of $\mathcal{M}$ and $\mathcal{U}$. Different architectures of neural operator incorporate different classical reduced basis representations such as proper orthogonal decomposition (POD) \cite{BhattacharyaHosseiniKovachkiEtAl2020,FrescaManzoni2022,OLearyRoseberryDuChaudhuriEtAl2021,OLearyRoseberryVillaChenEtAl2022}, Fourier representation \cite{LiKovachkiAzizzadenesheliEtAl2020a}, multipole graph representations \cite{LiKovachkiAzizzadenesheliEtAl2020b}, and derivative sensitivity bases \cite{OLearyRoseberryChenVillaEtAl2022,OLearyRoseberryDuChaudhuriEtAl2021,OLearyRoseberryVillaChenEtAl2022} among others. These reduced basis representations offer a scalable means of learning structured maps between infinite-dimensional spaces, by taking advantage of their compact representations in specific bases, if such representations exist. From a numerical point of view, this also mitigates direct dependence on the possibly enormous degrees-of-freedom $d_u$ and $d_m$ of the discretized state and parameter.

In the case that a sufficiently accurate network operator $\operator_{\boldsymbol{w}^\dagger}$ (i.e. specific architecture and weights) in a well-designed reduced basis representation is found, one can substitute this surrogate for the forward operator in the likelihood function,
\begin{equation}
    \mathcal{L}(m;\boldsymbol{y}^*)\mapsto\widetilde{\mathcal{L}}(m;\boldsymbol{y}^*)\coloneqq \pi_{\boldsymbol{N}}\left(\boldsymbol{y}^* - (\bdmc{B}\circ\operator_{\boldsymbol{w}^{\dagger}})(m)\right)\,,
\end{equation}
resulting in significant computational speedups for infinite-dimensional Bayesian inverse problems with models governed by nonlinear PDEs. We note that the same neural operator can also be deployed to accelerate Bayesian inverse problems governed by the same model but defined by possibly a variety of different noise models, observation operators, or data, which leads to additional computational cost reduction via the amortization of the model deployment in many different problems. Neural operators have also observed success in their deployment as surrogates for accelerating so-called ``outer-loop'' problems, such as inverse problems \cite{LiKovachkiAzizzadenesheliEtAl2020a}, Bayesian optimal experimental design \cite{Wu2022}, PDE-constrained optimization \cite{wang2021fast}, etc., where models governed by PDEs need to be solved repeatedly at different samples of input variables.

The hope with neural operators is that the ability to learn very high-dimensional complex nonlinear parametric maps via fast-to-evaluate surrogates provides an alternative approach for making the infinite-dimensional outer-loop problems governed by nonlinear parametric PDEs tractable, due to the significant reduction in per-iteration costs. In many settings, however, sufficient accuracy may be out of reach due to some mathematical and numerical issues inherent to neural network construction and training.

\subsection{Sources and reduction of the approximation errors}

While approximation theories posit the existence of arbitrarily accurate neural networks for operator learning problems, reliably realizing them in practice still remains a major obstacle in machine learning research. In particular, one often observes empirical ceilings in the approximation accuracy of neural operators, typically measured by numerically estimating a relative accuracy percentage metric given by
\begin{equation}\label{eq:t_accuracy}
    100\left(1 - \mathbb{E}_{M\sim \nu_M}\left[\frac{\norm{\mathcal{F}(M) - \operator_{\boldsymbol{w}}(M)}_{\mathcal{U}}}{\norm{\mathcal{F}(M)}_{\mathcal{U}}}\right]\right)\,.
\end{equation}
The accuracy metric above is often referred to as the \textit{generalization accuracy}. The empirical accuracy ceiling is often persistently observed in the regime where both the error in estimating the loss function, such as statistical errors due to finite samples, and truncation error due to reduced bases representation are asymptotically small. This is due to many confounding issues that we discuss here in brief.

The universal approximation theoretic understanding is typically disconnected from the way neural operators are constructed in practice. Many universal approximation results hinge on density arguments, i.e., certain classes of neural networks are capable of approximating, e.g., polynomials or simple functions to arbitrary accuracy. These density arguments are then used to tie error bounds to well-known approximation results for the class of functions that neural networks are dense in. The theoretical construction of the neural networks used in density arguments may lead to infinitely broad or deep networks in the limit \cite{Cybenko1989,Hornik1991,LiLinShen2019,LinJegelka2018,LuPuWangEtAl2017}.

In practice, however, neural network performance eventually degrades as the network gets too broad or deep, and thus neural operators are not constructed in this way. Instead one has to employ a combination of physical intuition, architecture search, and model selection techniques to produce a neural network that will work suitably for the target setting; see, e.g., \cite{anders1999model,ren2021comprehensive}. The resulting neural network is then calibrated via an empirical risk minimization on finite sample data using an optimizer that searches for local minimizers. The empirical risk minimization problem is typically nonconvex, in which case finding global minimizers is NP-hard. The empirical risk minimization problem over finite samples in \eqref{eq:emp_risk_min} introduces both statistical sampling error, as well as optimization error since the local minimizer may be significantly worse than the global minimizer for that particular neural network. 

Altogether these different errors lead to a challenging situation where, while neural operators show significant promise in learning complex parametric maps from training data up to a certain accuracy, eventually one cannot continue to improve the empirical accuracy reliably, contrary to other approximation methods such as the finite element method with hp-refinement~\cite{Babuska1994, demkowicz1989toward, demkowicz2006computing} or surrogate modeling approaches such as polynomial approximation, model order reduction, and Gaussian process approximation ~\cite{ernst2012on, chen2017reduced, teckentrup2020convergence}. In fact, at a certain point adding more representation power can make the approximation worse, and more training data does not improve the generalization accuracy. For a relatively comprehensive empirical comparison of some neural operators that demonstrates the diminishing marginal returns described above, see recent work by de Hoop et.\thinspace al in~\cite{deHoop2022cost}.

\subsection{Propagation of the approximation errors in Bayesian inverse problems}

If one employs a trained neural operator in place of the forward operator in a Bayesian inverse problem, the approximation error of the neural operator will lead to errors in the resulting posterior distribution. Depending on the conditioning of the particular Bayesian inverse problem of interest, the empirical accuracy ceiling of neural operator training might not be sufficient for reaching the accuracy requirement of the Bayesian inverse problem. 

Let $\mathcal{E}\in L^p(\mathcal{M},\nu_M;\mathcal{U})$ be the approximation error of the operator learning problem,
\begin{equation}
    \mathcal{E}(m) \coloneqq \mathcal{F}(m) -\operator_{\boldsymbol{w}}(m)\quad \nu_M\text{-a.e.} \, .
\end{equation}
The goal of the following analysis is to find \textit{a priori} bound on the error in the posterior distribution using the error in the operator learning problem, i.e., $\norm{\mathcal{E}}_{L^p(\mathcal{M},\nu_M;\mathcal{U})}$, and obtain an understanding of the bounding constant, i.e., the conditioning of Bayesian inverse problems with respect to the operator learning error. The error in the posterior distribution referred to as $\mathcal{E}_{\text{post}}$ can be represented by the Kullback--Leibler (KL) divergence  as follows
\begin{equation}
    \mathcal{E}_{\text{post}} \coloneqq D_{KL}(\widetilde{\nu_{M|\boldsymbol{Y}}}(\cdot|\boldsymbol{y}^*)||\nu_{M|\boldsymbol{Y}}(\cdot|\boldsymbol{y}^*)) = \mathbb{E}_{M\sim\postt{\cdot}}\left[\ln\left(\frac{\dd \postt{\cdot}}{\dd\post{\cdot}}(M)\right)\right]\,,
\end{equation}
where $\widetilde{\nu_{M|\boldsymbol{Y}}}(\cdot|\boldsymbol{y}^*)$ is the posterior distribution defined by Bayes' rule with the likelihood function $\widetilde{\mathcal{L}}$ based on an approximate mapping $\operator_{\boldsymbol{w}}$.

The following theorem suggests that for Bayesian inverse problems with regular observation operators and normally distributed additive noise, the error in the operator learning problem controls the error of Bayesian inverse problems in an abstract Hilbert space setting introduced in Section 2 and Section 3.1.
\begin{theorem}[Operator learning errors in Bayesian inverse problems]\label{theorem:bip_error}
Assume $\mathcal{F},\operator_{\boldsymbol{w}}\in L^p(\mathcal{M},\nu_M;\mathcal{U})$, $p\in[2,\infty]$. Assume $\bdmc{B}$ satisfies
\begin{equation*}
\begin{gathered}
    \norm{(\bdmc{B}\circ\mathcal{F})(m)}_2\leq c_B\norm{\mathcal{F}(m)}_{\mathcal{U}}\,,\quad \norm{(\bdmc{B}\circ\operator_{\bw})(m)}_2\leq \widetilde{c_B}\norm{\mathcal{F}(m)}_{\mathcal{U}}\,\quad\nu_M\text{-a.e.},\\
    \norm{(\bdmc{B}\circ\mathcal{F})(m) - (\bdmc{B}\circ\operator_{\bw})(m)}_2\leq c_L\norm{\mathcal{E}(m)}_{\mathcal{U}}\,\quad\nu_M\text{-a.e.},
\end{gathered}
\end{equation*}
for constants $c_B,\widetilde{c_B}, c_L>0$. Assume $\nu_{\boldsymbol{N}} = \mathcal{N}(\boldsymbol{0},\boldsymbol{C}_{\boldsymbol{N}})$ is normally distributed and $\widetilde{\nu_{M|\boldsymbol{Y}}}(\cdot|\boldsymbol{y}^*)$, $\nu_{M|\boldsymbol{Y}}(\cdot|\boldsymbol{y}^*)$, $\nu_M$ are mutually absolutely continuous. For a given set of observation data $\boldsymbol{y}^*\in\R^{n_y}$, we have
\begin{equation*}
    \mathcal{E}_{\text{post}}\leq c_{\text{BIP}}\norm{\mathcal{E}}_{L^p(\mathcal{M},\nu_M;\mathcal{U})}\,,\quad c_{\text{BIP}} = c_1\left(c_2(1) + c_3(p)\right)c_L\,,
\end{equation*}
where the constants $c_1, c_2, c_3>0$ are defined by
\begin{align}\label{eq:const_apriori_bayesian_error_estimate}
    c_1 &= \frac{1}{2}\norm{\bC_{\bN}^{-1}\left((\bdmc{B}\circ\mathcal{F})(\cdot) + (\bdmc{B}\circ\operator_{\boldsymbol{w}})(\cdot) - 2\by^*\right)}_{L^p(\mathcal{M},\nu_M;\R^{n_y})}\,,\\
    c_2(p) &= \norm{\exp\left(-\frac{1}{2}\norm{\by^*-(\bdmc{B}\circ\operator_{\boldsymbol{w}})(\cdot)}^2_{\bC_{\bN}^{-1}}\right)}^{-1}_{L^p(\mathcal{M},\nu_M)}\,,\\
    c_3(p) &=  \frac{c_2(1)}{c_2(q)}\in[1, c_2(1)]\,,\quad q = \begin{cases}
    \infty\,, &p =2\,;\\
    p/(p-2)\,, &p\in(2,\infty)\,;\\
    1\,, & p = \infty\,.\\
    \end{cases}
\end{align}
\end{theorem}

\begin{proof}
We only provide a sketch of proof here for $p\in[2,\infty)$. Detailed proof is provided in Appendix A. The following transformation can be made to $\mathcal{E}_{\text{post}}$
\begin{equation*}
    \mathcal{E}_{\text{post}} = \underbrace{\ln\left(\frac{Z(\by^*)}{\widetilde{Z}(\by^*)}\right)}_{\tcircle{A}} + \underbrace{\frac{1}{\widetilde{Z}(\by^*)}\mathbb{E}_{M\sim\nu_M}\left[\ln\left(\frac{\widetilde{\mathcal{L}}(M;\by^*)}{\mathcal{L}(M;\by^*)}\right)\widetilde{\mathcal{L}}(M;\by^*)\right]}_{\tcircle{B}}\,.
\end{equation*}
We seek to bound the two terms. By applying Cauchy–Schwarz inequality and Minkowski inequality, we have
\begin{gather*}
    \norm{\Phi-\widetilde{\Phi}}_{L^{p^*}(\mathcal{M},\nu_M)}\leq c_1c_L\norm{\mathcal{E}}_{L^p(\mathcal{M},\nu_M;\mathcal{U})}\,,\quad p^*\in[1, p/2]\,,\\
    c_1\leq \frac{1}{2}\norm{\bC_{\bN}^{-1}}_{2}\left(c_B\norm{\mathcal{F}}_{L^p(\mathcal{M},\nu_M;\mathcal{U})} + \widetilde{c_B}\norm{\operator_{\bw}}_{L^p(\mathcal{M},\nu_M;\mathcal{U})} + 2\norm{\by^*}_2\right)< \infty\,,
\end{gather*}
where $\Phi(m)$,$\widetilde{\Phi}(m)$ are the potentials defined with $\mathcal{F}$,$\operator_{\boldsymbol{w}}$ as in~\eqref{eq:gaussian_noise}. Using Jensen's, H\"older's, and Minkowski inequalities, we have
\begin{equation*}
    c_2(1)^{-1}\geq \exp\left(- \frac{1}{2}\norm{\bC_{\bN}^{-1}}_{2}\left(\norm{\by^*}_2 + \widetilde{c_B}\norm{\operator_{\boldsymbol{w}}}_{L^{p}(\mathcal{M},\nu_M;\mathcal{U})}\right)\right) > 0\,,
\end{equation*}
where $c_2 = c_2(p)$ as a function of $p$ is defined in \eqref{eq:const_apriori_bayesian_error_estimate}. The second term in $\mathcal{E}_{\text{post}}$ can be bounded by the following term using Jensen's and H\"older's inequalities:
\begin{equation*}
    \tcircle{B}\leq c_3(p)\norm{\Phi-\widetilde{\Phi}}_{L^{p^*}(\mathcal{M},\nu_M)}\leq c_1c_3(p)c_L\norm{\mathcal{E}}_{L^p(\mathcal{M},\nu_M;\mathcal{U})}\,.
\end{equation*}
The term involving normalization constants can be bounded using inequalities $\log(1 + x)\leq x\,, \forall x\geq0$, $|e^{-x_1} - e^{-x_2}|\leq|x_1-x_2|\,,\forall x_1,x_2\geq0$, and H\"older's inequality,
\begin{align*}
    \tcircle{A} &\leq c_2(1)\left|\mathbb{E}_{M\sim\nu_M}\left[\exp(-\Phi(m))-\exp(-\widetilde{\Phi}(m))\right]\right|\\
    &\leq c_2(1)\norm{\Phi-\widetilde{\Phi}}_{L^{p/2}(\mathcal{M},\nu_M)}
    \leq c_1c_2(1)c_L\norm{\mathcal{E}}_{L^p(\mathcal{M},\nu_M;\mathcal{U})}\,.
\end{align*}
Therefore, the bound $\mathcal{E}_{\text{post}}\leq c_{\text{BIP}}\norm{\mathcal{E}}_{L^p(M,\nu_M;\mathcal{U})}$ holds for $c_{\text{BIP}} = c_1\left(c_2(1) + c_3(p)\right)c_L$.
\end{proof}

We note that the constants $c_1$ and $c_L$ are associated with the error in the evaluation of the potential as defined in~\eqref{eq:gaussian_noise}, while $c_2$ and $c_3$ are associated with the error in the evaluation of the model evidence. The magnitude of the bounding constant is, generally speaking, dictated by the magnitude of the data misfit term $\norm{\boldsymbol{y}^* - \boldsymbol{y}^{\text{pred}}}_{\bC_{\bN}^{-1}}$, with $\boldsymbol{y}^{\text{pred}}$ being the predicted data, over the prior distribution for both the model and the surrogate. This constant is large for challenging Bayesian inverse problems that have uninformative prior, high-dimensional data, small noise corruption, or inadequate models. In these cases, the posterior distributions are typically highly localized and/or much different than the prior distributions. This is somewhat expected as neural operators are trained with data generated from prior distributions, and the performance of neural operators is likely to deteriorate when posterior sampling is concentrated in areas of the parameter space where training samples are relatively sparse. Additionally, Bayesian inverse problems are susceptible to the approximation error originating from overfitting to training samples, regardless of whether they adequately resolve the region of the parameter space with high likelihood values.

This result implies that significant magnification of the approximation error, when propagated through challenging Bayesian inverse problems, is likely. It further implies that the empirically observed accuracy ceiling for neural operator training might not be sufficient for the accuracy requirements of these types of Bayesian inverse problems. We note that this is often not the case for conventional surrogate modeling techniques, for which similar results are derived with the additional assumption that the error in Bayesian inverse problem solutions generated via a surrogate model asymptotically diminishes due to the asymptotic diminishing of the surrogate approximation error as $n_{\text{train}}\to\infty$; see~\cite{Yan2017, Marzouk2009, Stuart2017,teckentrup2020convergence}.
\section{Residual-based error correction of neural operator predictions}

Methods for estimating error, and subsequently correcting the error, in approximations of solutions to the forward problems fall into the category of \textit{a posteriori} error estimation and typically involve computing residuals representing the degree to which approximate solutions fail to satisfy the forward problems; i.e., residuals evaluated at the approximate solutions. In this section, we propose a strategy, following earlier works on \textit{a posteriori} error estimation techniques in \cite{oden2002estimation, jha2022goal}, that enhances the application of neural operators in infinite-dimensional Bayesian inverse problems using the PDE residual.

In particular, we ask, given a neural operator prediction $\operator_{\boldsymbol{w}}(m)\in\mathcal{U}$ at $m\in\mathcal{M}$ that is reasonably close to the true solution $\mathcal{F}(m)\in\mathcal{V}_u$, if it is possible to cheaply compute a \textit{corrected} solution, $u_C\in\mathcal{V}_u$, that has a significantly smaller approximation error, so that the mapping $m\to u_C(m)$ can more reliably achieve the accuracy requirements for deployment in Bayesian inverse problems. It turns out that if the residual operator $\mathcal{R}$ as in~\eqref{eq:residual} has a sufficiently regular derivative, a correction can be computed by solving a linear variational problem based on $\mathcal{R}$ at $\operator_{\boldsymbol{w}}(m)$, and this correction may lead to quadratic error reduction. For highly nonlinear problems, computing one correction step is inexpensive relative to evaluating the forward operator, thus this approach retains substantial speedups for Bayesian inverse problems.

In what follows, we first introduce the linear variational problem associated with the error correction. We show that it is a Newton step under some conditions, and thus the Newton--Kantorovich theorem can be directly applied to understand the error reduction. Then an application of this general procedure for neural operators is discussed. We discuss the possibility of conditions on the approximation error of a trained neural operator that ensures global quadratic error reduction of the correction step. Next, we make a connection between the error correction problem and \textit{a posteriori} error estimation techniques for estimating modeling error. Lastly, an analysis of the computational cost of MCMC sampling of the posterior using an error-corrected neural operator is given.

\subsection{The residual-based error correction problem}
Consider an expansion of the residual operator $\mathcal{R}$ in the state space $\mathcal{U}$, assuming for now that $\mathcal{R}$ is at least twice Gate\^aux-differentiable with respect to the state variable. For an arbitrary pair $(u, m)\in\mathcal{U}\times\mathcal{M}$, the first and second order derivatives with respect to the state variable, $\delta_u\mathcal{R}(u,m):\mathcal{U}\to\mathcal{U}_0^*$ and $\delta^2_u\mathcal{R}(u,m):\mathcal{U}\times\mathcal{U}\to\mathcal{U}_0^*$, are given by,
\begin{equation}
\begin{split}
    \delta_u\mathcal{R}(u,m)v &= \lim_{\theta\to 0}\frac{1}{\theta}\left(\mathcal{R}(u+\theta v, m)-\mathcal{R}(u,m)\right)\,,\quad\forall v\in\mathcal{U}\,,\\
    \delta_u^2 \mathcal{R}(u,m)(v, q) &= \lim_{\theta\to 0}\frac{1}{\theta}\left(\delta_u\mathcal{R}(u+\theta q,m)v - \delta_u\mathcal{R}(u,m)v\right)\,, \quad\forall v,q\in\mathcal{U}\,.
\end{split}
\end{equation} 
The Taylor series expansion of $\mathcal{R}(\mathcal{F}(m), m)$ in terms of $u$ involving up to second-order derivative of $\mathcal{R}$ is given by
\begin{equation}\label{eq:residual_expansion}
\begin{split}
    \underbrace{\mathcal{R}(\mathcal{F}(m), m)}_{=0} &= \mathcal{R}(u,m) +  \delta_u\mathcal{R}(u,m)(\mathcal{F}(m)-u) \\
    &\qquad + \int_0^1 \delta^2_u\mathcal{R}(u + s (\mathcal{F}(m)-u), m)(\mathcal{F}(m)-u, \mathcal{F}(m)-u) (1 - s) \dd s.
\end{split}
\end{equation}
In the regime where the higher order terms are small, such as where $\norm{\mathcal{F}(m)-u}_{\mathcal{U}}$ is small, the following approximation can be made
\begin{equation}\label{eq:high-order_approx}
    \begin{aligned}
    \delta_u\mathcal{R}(u,m)\mathcal{F}(m) &= -\mathcal{R}(u,m) + \delta_u\mathcal{R}(u,m)u + O(\norm{\mathcal{F}(m) -u}_{\mathcal{U}}^2)\\
    &\approx -\mathcal{R}(u,m) + \delta_u\mathcal{R}(u,m)u\,.
    \end{aligned}
\end{equation}
This leads to the following linear variational problem, which we refer to as the \textit{residual-based error correction problem}, to compute the new estimation $u_C$ of $\mathcal{F}(m)$ given an initial estimation $u\in\mathcal{U}$:
\begin{equation}~\label{eq:error_correction}
    \text{Given } m\in\mathcal{M} \text{ and }u\in\mathcal{U},\text{ find } u_C\in \mathcal{V}_u \text{ such that } \quad    \delta_u\mathcal{R}(u,m)u_C = -\mathcal{R}(u,m) + \delta_u\mathcal{R}(u,m)u\,.
\end{equation}

Assuming the existence of a unique solution to this linear variational problem under Lax--Milgram theorem~\cite[pg. 310]{Ciarlet2013}, solving this problem is equivalent, under some additional conditions, to generate a single Newton iteration for the nonlinear equation \eqref{eq:residual} reformulated as
\begin{equation}\label{eq:newton_fixed_point}
\text{Given } m\in\mathcal{M}, \text{ find } v\in \mathcal{U}_0 \text{ such that}\quad \widetilde{\mathcal{R}}(v,m) = 0\,,\quad\widetilde{\mathcal{R}}(v,m)\coloneqq\mathcal{R}(v + u_L,m)\,,
\end{equation}
for a given $u_L\in\mathcal{V}_u$. The Newton iteration $\{v_j\in\mathcal{U}_0\}_{j=0}^{\infty}$ is given by
\begin{equation}
    v_{j+1} = v_j - \delta_v\widetilde{\mathcal{R}}(v_j,m)^{-1}\widetilde{\mathcal{R}}(v_j,m)\,,\quad j>0\,.
\end{equation}
Consider one Newton step with $u - u_L - u_{\perp}$, where $\mathcal{U}_0^{\perp}$ is the orthogonal complement of $\mathcal{U}_0$, and $u_{\perp}$ is the orthogonal projection of $u-u_L$ to $\mathcal{U}_0^{\perp}$.
\begin{equation*}
    v_C = u - u_L - u_{\perp} - \delta_v \widetilde{\mathcal{R}}(u - u_L  - u_{\perp},m)^{-1}\widetilde{\mathcal{R}}(u - u_L - u_{\perp},m)\,,
\end{equation*}
The equivalency $v_C = u_C - u_L$ can be established when, for example,  (i) $\mathcal{U} = \mathcal{U}_0 = \mathcal{V}_u$, (ii) $u\in\mathcal{V}_u$, i.e., $u$ has the correct strongly-enforced boundary and initial conditions with $u_{\perp}= 0$, or (iii) $\mathcal{R}(u + \mathcal{U}_0^{\perp},m) \equiv \mathcal{R}(u,m)$, i.e., the residual operator is invariant to changes in the strongly-enforced boundary and initial conditions of $u$. 

If the equivalency between the Newton step problem and the error correction problem can be established, the latter can be understood in the setting of the Newton--Kantorovich theorem in Banach spaces; see, \cite{Ortega1968, Ciarlet2013} and references therein. The theorem gives sufficient conditions for the quadratic convergence of a Newton iteration starting at $u$ for solving the fixed point problem associated with~\eqref{eq:newton_fixed_point}. It implies that within a regime in which $u$ is sufficiently close to $\mathcal{F}(m)$, the solution to the error correction problem is guaranteed a quadratic error reduction, i.e.,
\begin{equation}
    \norm{\mathcal{F}(m)-u_C}_{\mathcal{U}}\leq c\norm{\mathcal{F}(m) - u}_{\mathcal{U}}^2
\end{equation}
for some constant $c>0$. Consequently, we expect $u_C$ to have a much smaller error than $u$ in such a regime.

We remark that in the rare cases where the equivalency of the two problems cannot be established, one may still retain the conditional quadratic error reduction property for the error correction problem, but perhaps with stronger conditions than those stated by the Newton--Kantorovich theorem. This situation may arise in practice, for example, when both conditions for equivalency suggested above fail, and the numerical implementation of the orthogonal projection from $\mathcal{\mathcal{U}}$ to $\mathcal{V}_u$ is not straightforward. The stronger conditions for quadratic error reduction should incorporate additional sensitivity factors for when the initial position of a Newton iteration orthogonally deviates from the domain of the Jacobian operator of the nonlinear system.

\subsection{Error correction of neural operator prediction}\label{ss:neural_op_corr}

For well-designed and well-trained neural operators, we expect they are within the regime where the approximation error is small. However, in such a regime, the reduction of the approximation error via neural network construction or training is ineffective, as discussed in Section 3.2. On the other hand, solving a residual-based error correction problem returns a significant reduction in approximation error that may not be reliably achieved in the training phase of neural operators. This reduction of the approximation errors may lead to crucial improvements in the quality of the solutions to Bayesian inverse problems accelerated by neural operators. To this end, we consider the following linear problem that returns a neural operator with error correction $\operator_C:\mathcal{M}\ni m\mapsto u_C\in\mathcal{V}_u$:
\begin{equation}\label{eq:error_correction_neural_operators}
\begin{aligned}
    &\text{Given }m\in\mathcal{M}, \text{ find }u_C\in\mathcal{V}_u\text{ such that }\\
    &\qquad\qquad\delta_u\mathcal{R}(\operator_{\boldsymbol{w}}(m), m)u_C = -\mathcal{R}(\operator_{\boldsymbol{w}}(m), m) + \delta_u\mathcal{R}(\operator_{\boldsymbol{w}}(m), m)\operator_{\boldsymbol{w}}(m) ,
\end{aligned}
\end{equation}
i.e., we let $u = \operator_{\boldsymbol{w}}(m)$ in~\eqref{eq:error_correction}. The mapping $\operator_C$ is then used in likelihood evaluations at $m$ for solving Bayesian inverse problems:
\begin{equation}
    \mathcal{L}(m;\boldsymbol{y})\mapsto \widetilde{\mathcal{L}}_C(m;\boldsymbol{y}^*)\coloneqq \pi_N(\boldsymbol{y}^*-\bdmc{B}(\operator_{C}(m))) .
\end{equation}

Under the framework of the Newton--Kantorovich theorem, if the following two conditions are satisfied at states that are nearby $\operator_{\boldsymbol{w}}(m)$ for all $m\in\mathcal{M}$: (i) the neural operator error correction problems have unique solutions under the Lax--Milgram theorem, and (ii) the residual derivative is locally Lipschitz continuous with respect to the state variables, then for $\mathcal{F}\in L^{\infty}(\mathcal{M},\nu_M;\mathcal{U})$ there exists a global convergence radius upper bound $r_C > 0$ that depends on the properties of $\mathcal{R}$ such that 
\begin{equation}
    \norm{\mathcal{E}}_{L^\infty(\mathcal{M},\nu_M;\mathcal{U})}<r_C\implies \norm{\mathcal{E}_C}_{L^{\infty}(\mathcal{M},\nu_M;\mathcal{U})}\leq c_R\norm{\mathcal{E}}^{2}_{L^{\infty}(\mathcal{M},\nu_M;\mathcal{U})}\,,
\end{equation}
where $\mathcal{E}_C(m)\coloneqq \operator_C(m)-\mathcal{F}(m)$, $\nu_M$-a.e., and $c_R>0$ is a constant. We state the full form of this corollary in Appendix B.

We note that even though the global convergence radius upper bound is enforced in the $L^{\infty}$ Bochner norm, which may not be compatible with the operator learning setting, the global quadratic error reduction is observed in practice for neural operators trained in the $L^2$ Bochner norm with pre-asymptotic accuracy; see, e.g., Figure~\ref{fig:nonlinearpoisson_nns} and Figure~\ref{fig:hyperelasticity_nns} in our numerical examples to be specified in Section 5. We expect that these results are natural consequences of the reduced basis representation built into the neural operator architectures, in which the operator learning problem can be essentially represented in finite-dimensional coefficient spaces, even though the optimization is typically performed in the $L^2(\mathcal{M},\nu_M;\mathcal{U})$ topology. Moreover, in the cases where the global convergence radius upper bound is small, trained neural operators with pre-asymptotic accuracy may not be sufficient for a guaranteed global quadratic error reduction via error correction. In these cases, performance-improving techniques for neural operator construction and training can be crucial. Even though these techniques may not produce neural operators with arbitrary accuracy, they may be able to reduce the approximation error so that its norm is smaller than the global convergence radius upper bound for a guaranteed quadratic error reduction according to the Newton--Kantorovich theorem.

\subsection{Connection to goal-oriented a posteriori error estimation}

The procedure we described to correct the predictions of a trained neural operator is based on the estimation of modeling error using the goal-oriented a posteriori error estimates that provide computationally inexpensive estimates of the error in the quantities of interests (QoIs), see \cite{prudhomme1999goal, prudhomme2003computable, becker2001optimal, oden2002estimation, oden2001goal, jha2022goal, ainsworth1997posteriori, Ainsworth-2000, rannacher1997feed, van2011goal, giles2002adjoint, pierce2000adjoint}. Originally, a goal-oriented error estimation technique was developed \cite{becker2001optimal, rannacher1997feed, ainsworth1997posteriori, Ainsworth-2000, oden2001goal, prudhomme1999goal} to perform mesh adaption based on the specific measure of error -- error in QoIs -- as opposed to the energy norm. Techniques were developed to adapt the mesh to control the error in QoIs which is not necessarily possible by relying on the energy norm. It was soon realized that goal-oriented error estimates can be used to estimate the modeling error -- error in predictions via the fine (high-fidelity) and coarse (low-fidelity) models -- further expanding the utility of the technique; see \cite{oden2002estimation, jha2022goal}. Here we briefly cover the topic of estimation of modeling error to relate to the development in previous subsections. 

Suppose $u\in \mathcal{U}$ is a solution to the variational problem, $\mathcal{R}(u) = 0\, \in \mathcal{U}^\ast$, assuming test and trial spaces are the same for simplicity, and $p\in \mathcal{U}$ is the solution to the dual problem: $-\dualDot{\delta_u \mathcal{R}(u)v}{p} = \dualDot{\delta_u \mathcal{Q}(u)}{v}$ $\forall v\in \mathcal{U}$, where $\langle\cdot,\cdot\rangle$ denotes a duality pairing. Here $\mathcal{Q}:\mathcal{U}\to\R$ is a differentiable QoI functional. For any arbitrary functions $\widetilde{u}, \widetilde{p} \in \mathcal{U}$ the following holds,
\begin{equation}
    \mathcal{Q}(u) - \mathcal{Q}(\widetilde{u}) = \dualDot{\mathcal{R}(\widetilde{u})}{p} + r(u, p, \widetilde{u}, \widetilde{p}) \approx \dualDot{\mathcal{R}(\widetilde{u})}{p},
\end{equation}
where $r$ is the remainder term involving derivatives of residual and errors in forward and dual solutions, $e = u - \widetilde{u}$ and $\varepsilon = p - \widetilde{p}$, respectively. When $e$ and $\varepsilon$ are sufficiently small, $r$ can be ignored. We note that several versions of such estimates as the equation above can be derived; see \cite{oden2002estimation}. While estimates like the above provide an approximation of the QoI error, they still require access to the pair of the forward and dual solutions, $(u,p)$. Because $u$ can be written as $u = \widetilde{u} + e$, if the error $e$ can be approximately computed, referring to the approximation of $e$ by $\hat{e}$, then $u$ can be approximated using $u \approx \widetilde{u} + \hat{e}$. The approximation $\hat{e}$ of the error $e$ can be obtained by solving the following linear variational problem:
\begin{equation}\label{eq:error_approx_goal_oriented}
    \text{Find }\hat{e}\in\mathcal{U}\text{ such that}
    \quad \delta_u \mathcal{R}(\widetilde{u})\hat{e} = -\mathcal{R}(\widetilde{u}) \qquad \in \mathcal{U}^*\,.
\end{equation}
The same steps can be used to also compute $p$ approximately following the discussion in the earlier references. 

The above problem is formally derived via the linearization of the nonlinear variational problem similar to the one presented in Section 4.1. The steps discussed above are somewhat similar to estimating the error $u_C-\operator_{\boldsymbol{w}}(m)\approx \mathcal{F}(m) - \operator_{\boldsymbol{w}}(m)$ given a neural operator prediction $\operator_{\boldsymbol{w}}(m)$ proposed in Section \ref{ss:neural_op_corr}. However, one needs to also consider the error in $\mathcal{U}^{\perp}$ for when $\mathcal{U}_0\subset\mathcal{U}$, as the neural operator prediction $u$ may not satisfy the strongly-enforced boundary or initial conditions. Note that the error correction problem is essentially an extended error estimation problem, where one solves for the error in the affine space $\mathcal{U}_0-\widetilde{u}$, thus similar to an error estimate that includes the error of $\widetilde{u}$ at those boundaries.

\subsection{Discussion of computational costs}

In this subsection, we briefly discuss the computational costs when the error-corrected neural operator is deployed as a surrogate of the forward operator in Bayesian inverse problems, in particular the expected computational speedups in MCMC sampling of the posterior. The neural operator construction and training have offline costs due to training data generation, building or finding appropriate architecture, and training the neural operator. The cost of generating training data is the same as the cost of solving the PDEs and sampling from the prior; here we neglect the cost due to the latter. The additional computation may be required for finding reduced bases that involve expectations, as is done in the network strategies discussed in \cite{OLearyRoseberryVillaChenEtAl2022,OLearyRoseberryDuChaudhuriEtAl2021}, which we use in our numerical examples presented in Section 5. The cost of constructing an appropriate architecture is more abstract, since this procedure may involve repeating many different training runs for different architectural hyperparameters. 

In the case that reduced basis neural operator methods are used, the training cost can be made independent of the discretization dimensions $d_m,d_u$, and instead be made to scale with the dimension of the reduced bases, since in this case the training data can be projected into these bases in a pre-training step. For this reason, the cost of neural network training can be ignored asymptotically for very high-dimensional discretization. We note that neural network training may represent a significant cost in the pre-asymptotic regime, i.e. on the order of the cost of training data generation, depending on the type of architecture used and the number of iterations required to train a sufficiently accurate neural operator.

The cost of posterior sampling to acquire a Markov chain of size $n_\text{chain}$, according to our settings in Section 2.2, is dictated by the cost of solving the governing PDEs for likelihood evaluations. The total cost is a sum of the online posterior sampling cost and the offline cost:
\begin{equation}
    \text{Total Cost } = \left(n_\text{chain}\times \text{Evaluation Cost}\right) + \text{Offline Cost}\,.
\end{equation}
The speedups realized in practice by using a neural operator are measured by the ratio of the total cost of MCMC sampling using the forward operator to the total cost of MCMC sampling using a neural operator,
\begin{equation}
    \text{Speedup}_{NO} = \frac{\left(n_\text{chain}\times \text{Cost}_{PDE}\right) }{\left(n_\text{chain}\times \text{Cost}_{NO}\right) + \text{Offline Cost}}\,.
\end{equation}

In the regime where a large number of likelihood evaluations are required for generating a Markov chain of length $n_{\text{chain}}$, the offline cost of neural operator training is likely negligible in relation to the online cost of posterior sampling. Assuming the acceptance rates are the same for Markov chains generated by the forward operator with a neural operator as its surrogate, the speedup of the neural operator asymptotically becomes the ratio between the averaged evaluation cost for the forward operator and the neural operator over the posterior distribution.
\begin{equation}
    \text{Asymptotic Speedup }_{NO} \approx \frac{\text{Cost}_{PDE}}{\text{Cost}_{NO}}\,.
\end{equation}
In general, the cost ratio of the PDE solves to the neural operator can be very high, for highly nonlinear problems with fine discretization, i.e., large $d_u$ and $d_m$, in particular, if a dimension-independent neural operator strategy is used. 

When a neural operator with error correction is used as a surrogate of the forward operator in posterior sampling, the asymptotic speedup is approximately equal to the number of iterations within a given iterative scheme for nonlinear equations used for evaluating the forward operator, averaged over the posterior distribution. We assume at each iteration a linear problem of a similar size and numerical properties to the error correction problem is solved. Therefore, the asymptotic speedup is given by
\begin{equation}
    \text{Asymptotic Speedup}_{ECNO} = \frac{\text{Cost}_{PDE}}{\text{Cost}_{NO}+ \text{Cost}_{EC}} = \frac{N_\text{nonlinear} \times \text{Cost}_{EC}}{\text{Cost}_{NO}+ \text{Cost}_{EC}} \approx N_\text{nonlinear}\,.
\end{equation}
The proposed method thus yields favorable speedups for models governed by highly nonlinear PDEs with fine discretization, where evaluating the forward operator requires many iterative solves, and the posterior sampling requires many queries of the forward operator.

\section{Numerical Examples}
In this section, we demonstrate through numerical examples how the residual-based error correction both improves the accuracy of posterior representations of challenging Bayesian inverse problems relative to using a trained neural operator alone, and offers a computational advantage over directly using the expensive forward model in likelihood evaluations. The first example is concerned with inferring the uncertain coefficient field in a nonlinear reaction--diffusion problem with a cubic reaction term. The second example involves the inference of Young's modulus as a spatially-varying uncertain field for a hyperelastic neo-Hookean material undergoing deformation. For each of the two examples, we generate samples from the posterior distributions for likelihood functions defined using (i) model predictions, (ii) predictions by three trained neural operators whose measured accuracy is close to the ceiling for a given architecture with an increasing number of training samples, and (iii) error-corrected predictions of these neural operators. We visualize and compare the accuracy of their posterior predictive means, and analyze the computational cost.

In what follows, we first discuss the neural operator architecture and training as well as the software used in our calculations. Then the physical, mathematical, and numerical settings of the two examples are provided. Lastly, we present the visualization of posterior predictive means and cost analysis.

\subsection{Derivative-informed reduced basis neural operator}\label{section:derivate_neural_operator}

We consider a derivative-informed reduced basis neural network~\cite{OLearyRoseberry2020,OLearyRoseberryVillaChenEtAl2022,OLearyRoseberryDuChaudhuriEtAl2021} for constructing neural operators in the two numerical examples. This neural network architecture uses both derivative information and principal information of the solution map to construct appropriate reduced bases for the inputs and outputs. This strategy exploits the compactness of the forward operator, if it exists, in order to construct a parsimonious, mesh-independent neural network approximation, making it a suitable strategy for learning discrete representations of mappings between function spaces. The input reduced basis uses derivative information to detect subspaces of the inputs that the outputs are most sensitive in expectation over $\nu_M$. The output reduced basis is proper orthogonal decomposition (POD) which is an optimal linear basis for learning operators in the $L^2$ norm \cite{ManzoniNegriQuarteroni2016,QuarteroniManzoniNegri2015}. The neural operator learns a nonlinear coefficient mapping between these two subspaces via an adaptively constructed and trained ResNet \cite{OLearyRoseberryDuChaudhuriEtAl2021}. These networks provide a reliable way of detecting appropriate breadth for a map by detecting useful bases directly from the map, instead of during neural network training. The adaptive training adds layers in a way that only marginally perturbs the existing coefficient mapping, allowing one to adaptively detect appropriate nonlinearity (depth) until overfitting is detected. We only consider this network since it gives a good practical performance, outperforming other conventional reduced basis strategies and overparametrized networks, while suffering from the same limitations that are typical of neural operators in general: limitations in achieving high accuracy given more data, and approximation power. It is thus general enough to demonstrate the efficacy of our proposed method. 

For both numerical examples, we use networks with reduced basis dimensions of $50$, and $10$ nonlinear ResNet layers, each with a layer rank of $10$, and softplus activation functions. The networks are trained and constructed adaptively, one ResNet layer at a time, using the Adam optimizer with a learning rate of $10^{-3}$, and gradient batch size of $32$. Each adaptive layer training problem executes $100$ epochs, starting from $2$ layers to $10$, with one final end-to-end training, accounting for a total of $1000$ epochs. The efficacy of this adaptive approach is demonstrated in \cite{OLearyRoseberryDuChaudhuriEtAl2021}. The networks are trained for varying availability of training data, which are reported when appropriate. For both problems a derivative sensitivity basis of rank $50$ is computed via samples using matrix-free randomized algorithms,
\begin{equation}
    \mathbb{E}_{M\sim\nu_M}\left[\delta\mathcal{F}^*(M)\delta\mathcal{F}(M)\right] \approx \frac{1}{n_\text{sample}}\sum_{j=1}^{n_\text{sample}}\delta\mathcal{F}^*(m_j)\delta\mathcal{F}(m_j)\,,
\end{equation}
where $\delta\mathcal{F}^*(\cdot)$ denotes the adjoint of the derivative of $\mathcal{F}$. In both cases $n_\text{sample} = 256$ samples are used, and the additional computational costs are taken into account when computing the offline cost of neural operators. The matrix-free actions of $\delta\mathcal{F}^*(M)\delta\mathcal{F}(M)$ only require solving linearized PDEs similar to that of the error correction problem and can be computed at training data sample points, thus ensuring a reasonable computational efficiency of the matrix-free action evaluations. 

We note that unlike \cite{OLearyRoseberryDuChaudhuriEtAl2021}, we do not employ the additional stochastic Newton optimizer \cite{OLearyRoseberryAlgerGhattas2019,OLearyRoseberryAlgerGhattas2020} since this is not typically used in neural operator learning. However, we observe substantially the performance improvement of neural operators trained with this optimizer. Moreover, unlike \cite{OLearyRoseberryChenVillaEtAl2022} this network does not include derivative approximations during the formulation of the operator learning problem and training, and the derivative information is only used in the architecture. For more information on the neural network architectures employed in the two numerical examples, see \cite{OLearyRoseberryDuChaudhuriEtAl2021,OLearyRoseberryVillaChenEtAl2022}.

\subsection{Software}

The following software is used to implement the numerical examples: FEniCS \cite{alnaes2015fenics} for finite element method, hIPPYlib \cite{VillaPetraGhattas2018} for prior and posterior sampling, and hIPPYflow~\cite{hippyflow} in combination with TensorFlow~\cite{tensorflow2015-whitepaper} for neural operator construction and training. We note that the derivative of the residual operator is implicitly derived and implemented using the automatic differentiation of weak forms supported by FEniCS and Unified Form Language (UFL)~\cite{Alnaes2014}.

\subsection{Inverting a coefficient field of a nonlinear reaction--diffusion problem}\label{ss:poisson}
Here we introduce the setting for a demonstrative Bayesian inverse problem. The model is governed by an equilibrium reaction--diffusion equation with a nonlinear cubic reaction term. The model solutions are driven by the Dirichlet boundary conditions at the top and bottom of the square domain $\Omega = (0,1)^2$ and the coefficient field $\kappa:\Omega\to\R_+$:
\begin{subequations}
\begin{align}
    -\nabla \cdot \kappa(\boldsymbol{x})\nabla u(\boldsymbol{x}) + u(\boldsymbol{x})^3 &= 0, && \boldsymbol{x}\in \Omega\,;\\
    \kappa(\boldsymbol{x})\grad u(\boldsymbol{x}) \cdot\boldsymbol{n} &= 0, && \boldsymbol{x}\in \Gamma_{l}\cup \Gamma_{r}\,;\\
    u(\boldsymbol{x}) &= 1, && \boldsymbol{x}\in \Gamma_{t}\,;\\
    u(\boldsymbol{x}) &= 0, && \boldsymbol{x}\in \Gamma_{b}\,,
\end{align}
\end{subequations}
where $\Gamma_t$, $\Gamma_r$, $\Gamma_b$, and $\Gamma_l$ denote the top, right, bottom, and left boundary of the domain, and $\boldsymbol{n}$ is the outward normal vector. We assume prior knowledge of an epistemically uncertain and spatially-varying coefficient field $ K:\Omega\to\R_+$, following a log-normal prior distribution:
\begin{equation}
    K = \exp(M) \,,\qquad M\sim \nu_M\coloneqq\mathcal{N}(0, \mathcal{C}_{\text{pr}})\,,
\end{equation}
where $M$ is the parameter random variable through which we represent the uncertainty in $\kappa$. The prior distribution $\nu_M$ is specified with a covariance constructed according to~\eqref{eq:gaussian_prior} with hyperparameters of $d = 2$, $\alpha = 0.08$, $\beta =2$, and $\boldsymbol{\Theta} = \boldsymbol{I}$. They correspond to isotropic Gaussian random fields with a pointwise variance of approximately $1$, correlation length of approximately $0.2$, and small boundary artifacts. We visualize several samples of the prior as well as their corresponding coefficient fields and model solutions via the finite element method, to be specified later, in Figure~\ref{fig:poisson_prior}.

\begin{figure}[H]
    \centering
    \addtolength{\tabcolsep}{-6pt} 
    \begin{tabular}{ccccc}
        $M\sim \nu_M$ & $K$ & $\mathcal{F}^h(M)$ & $|\mathcal{F}^h(M) - \operator_{\boldsymbol{w}}(M)|$ & $|\mathcal{F}^h(M) - \operator_{C}^h(M)|$\\
        \includegraphics[width = 0.19\linewidth]{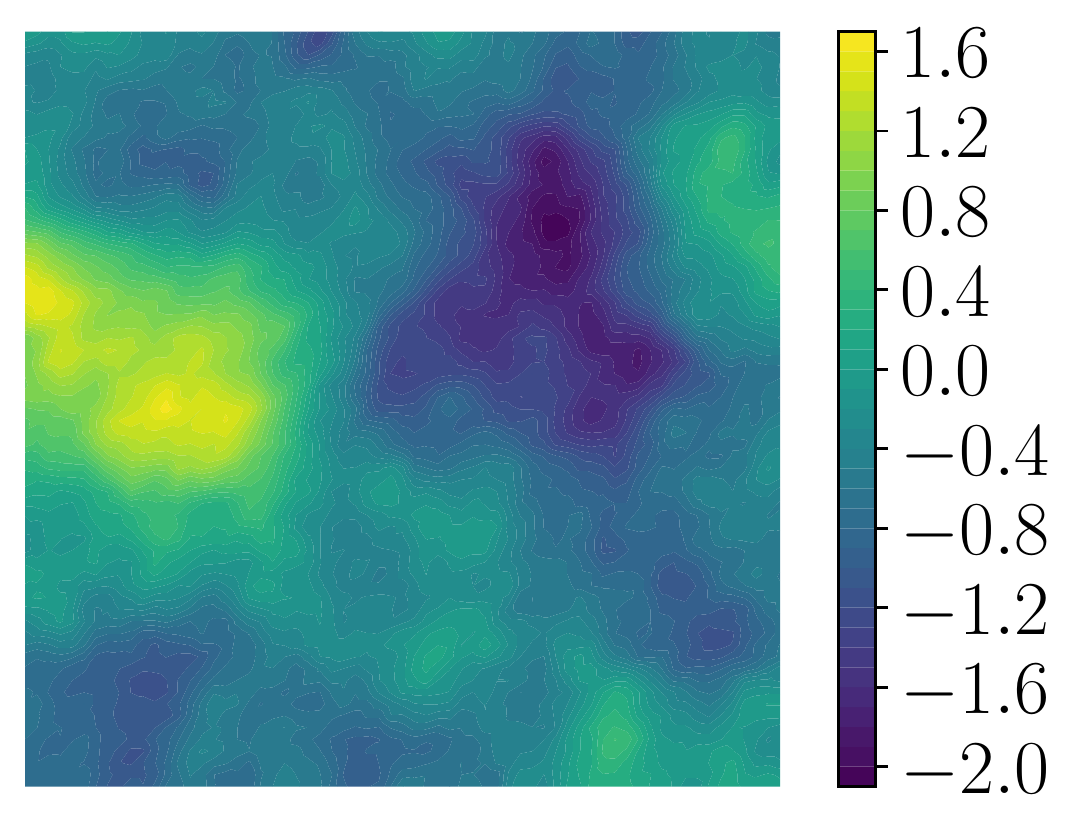}  & \includegraphics[width = 0.19\linewidth]{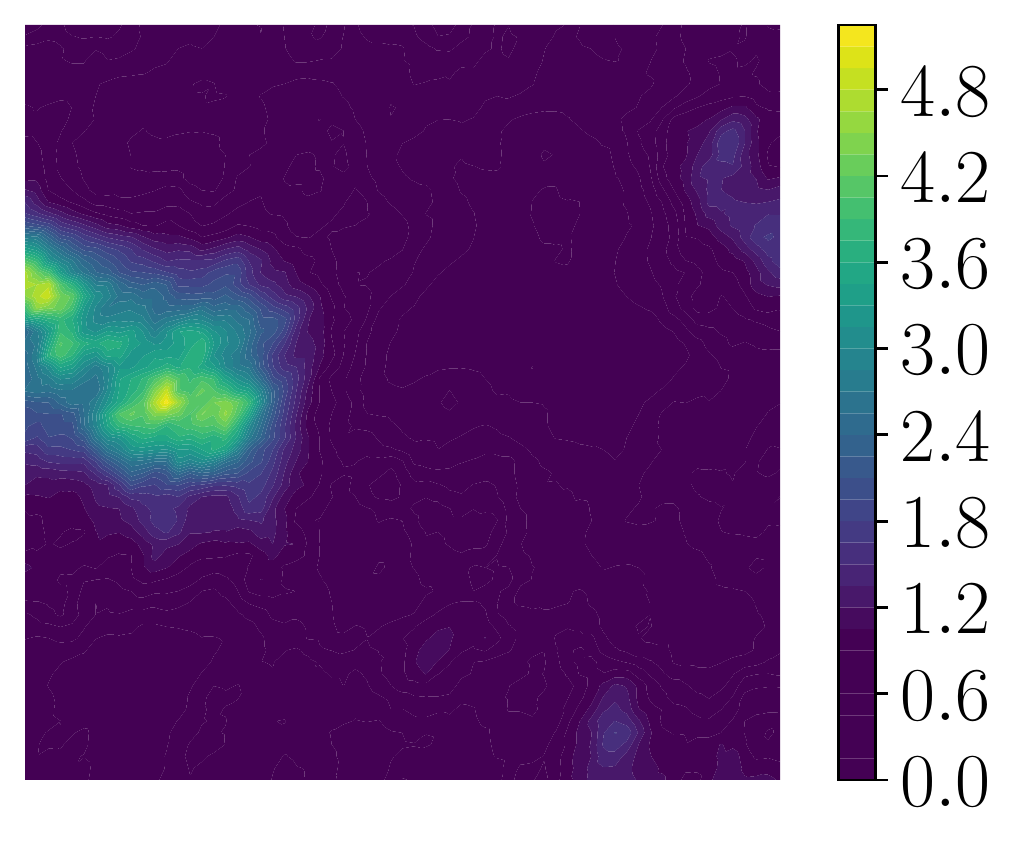} & \includegraphics[width = 0.19\linewidth]{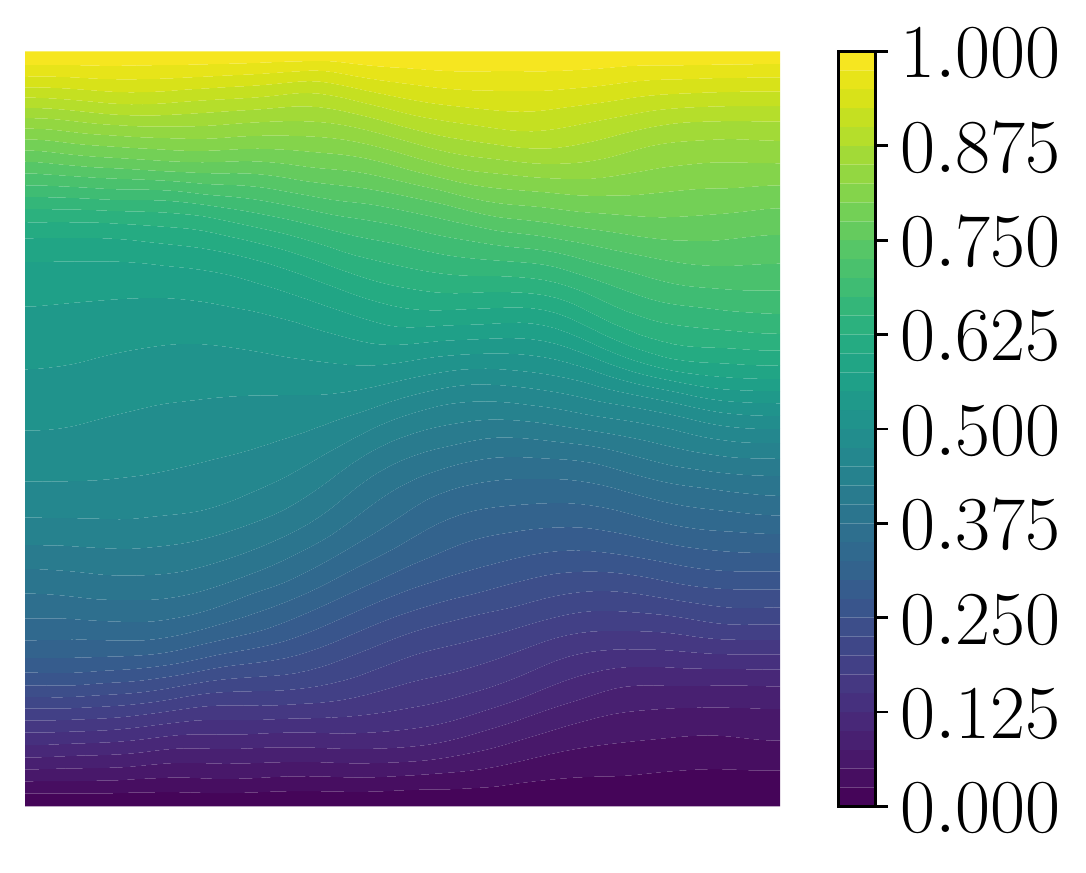} & \includegraphics[width = 0.19\linewidth]{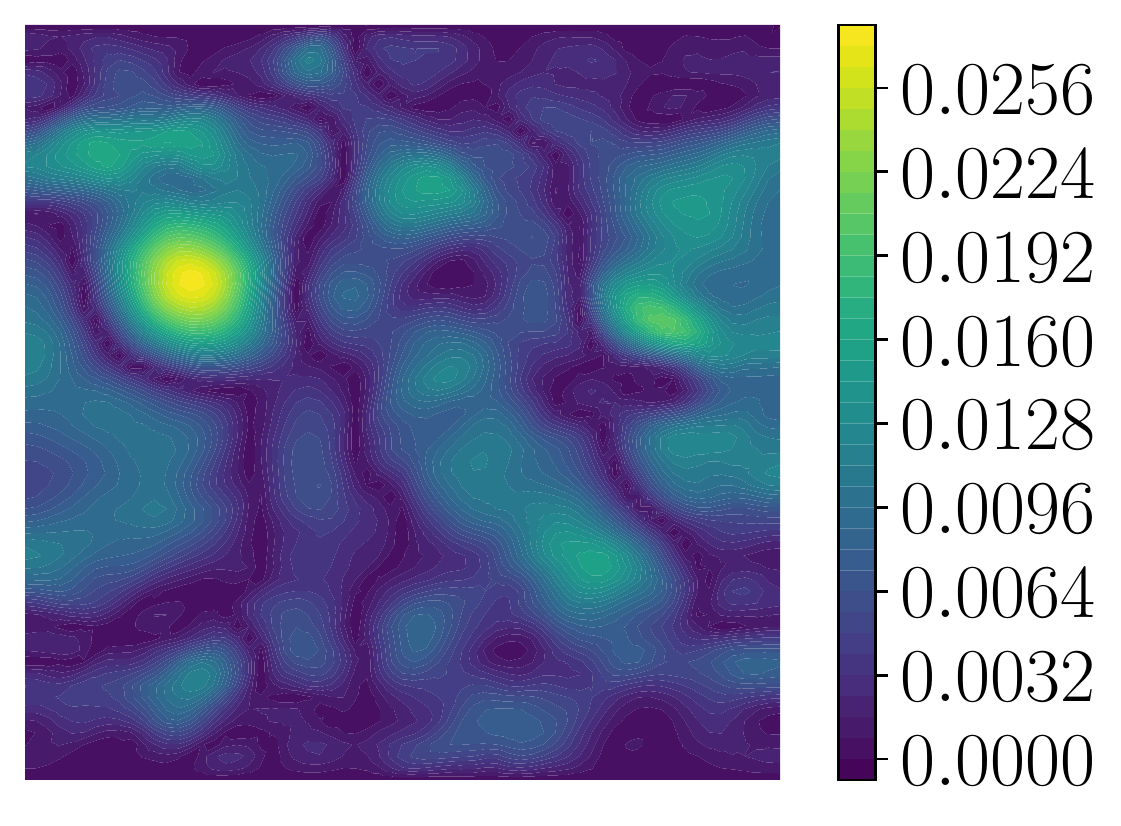} & \includegraphics[width = 0.19\linewidth]{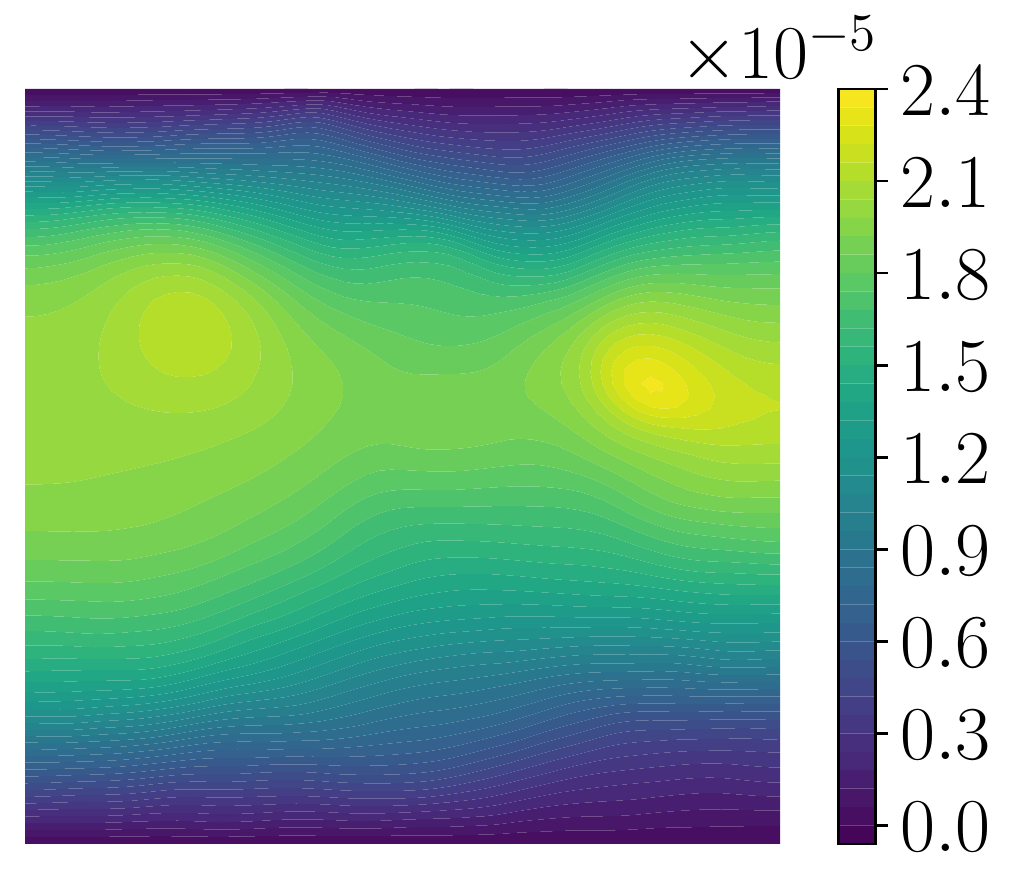}\\
        \includegraphics[width = 0.19\linewidth]{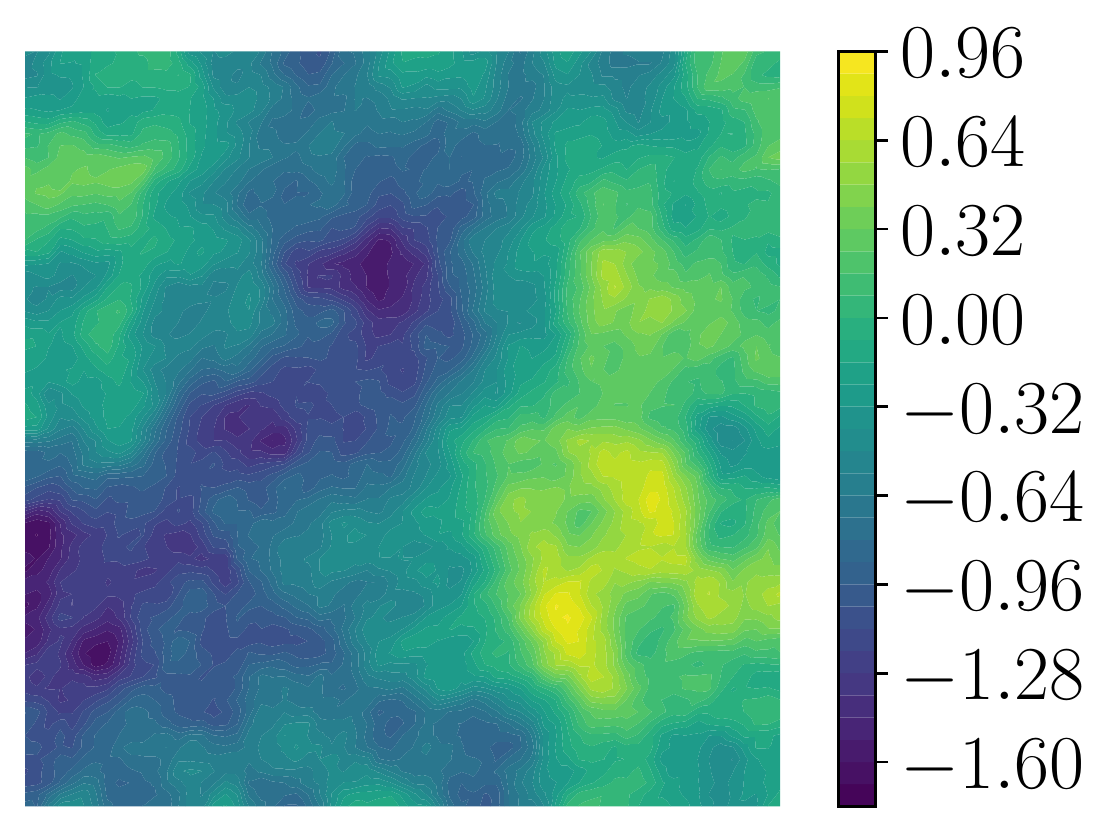} &\includegraphics[width = 0.19\linewidth]{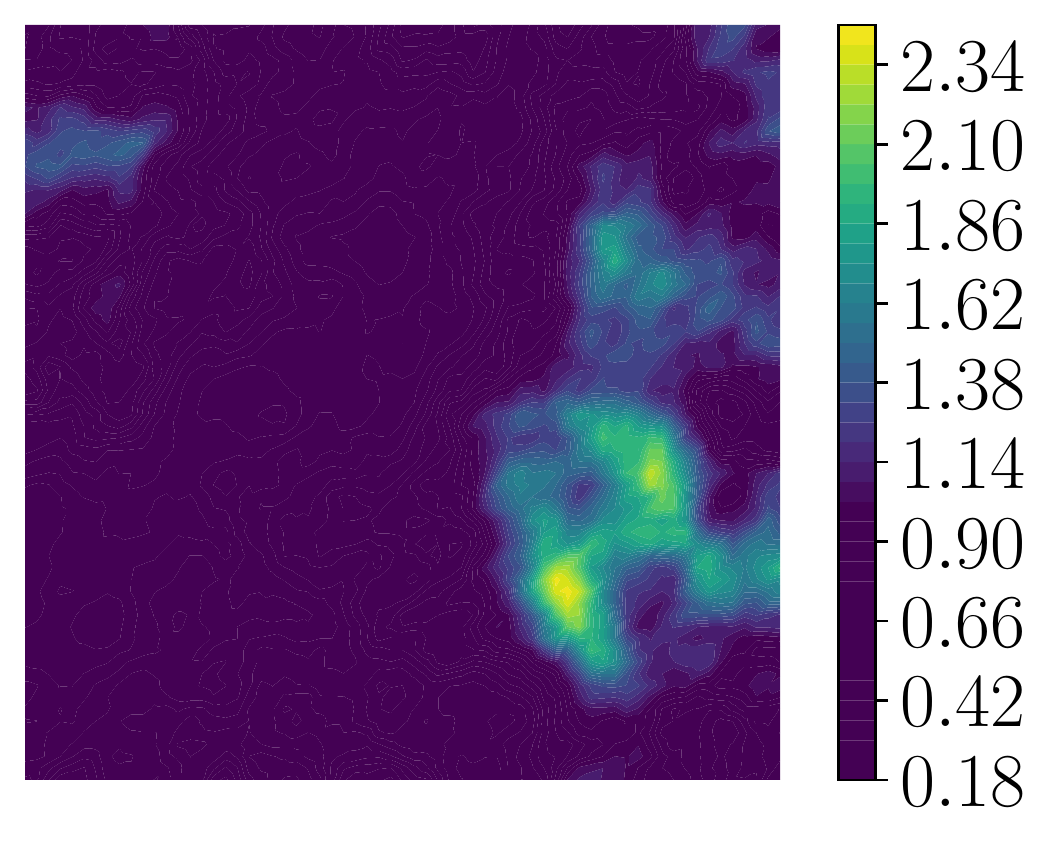} &\includegraphics[width = 0.19\linewidth]{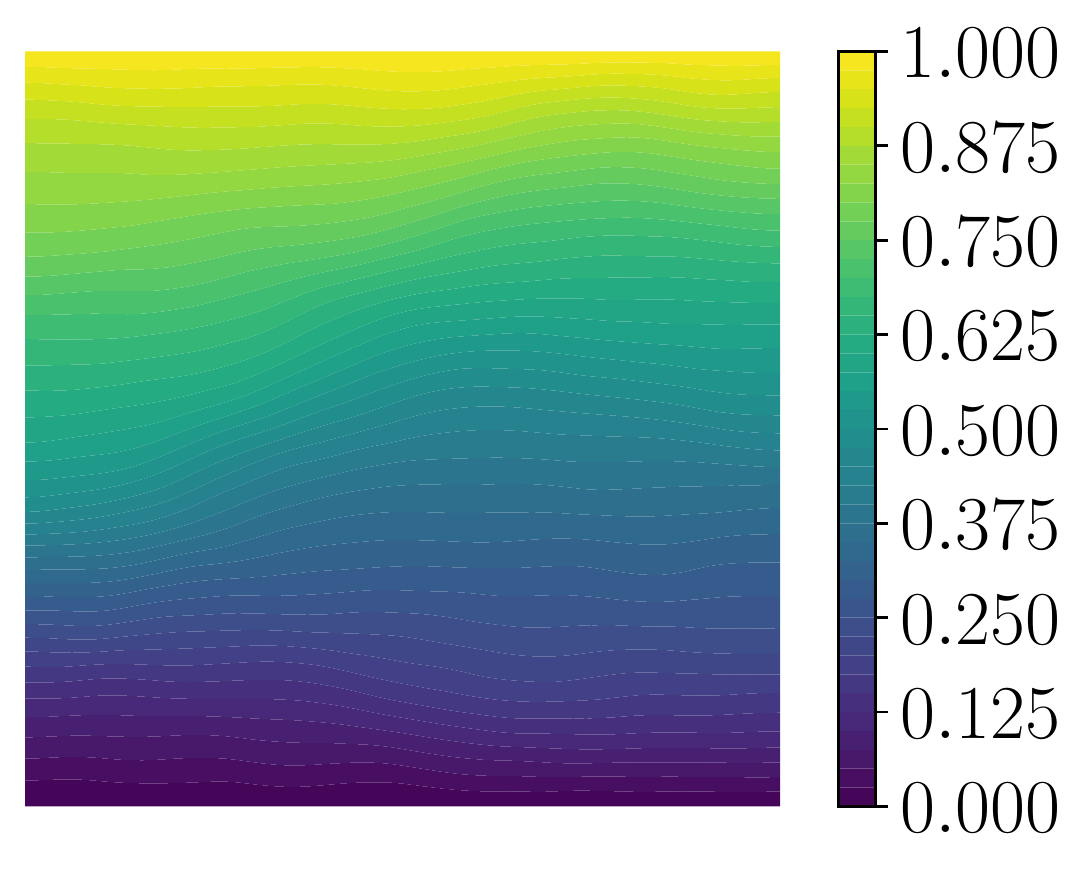} & \includegraphics[width = 0.19\linewidth]{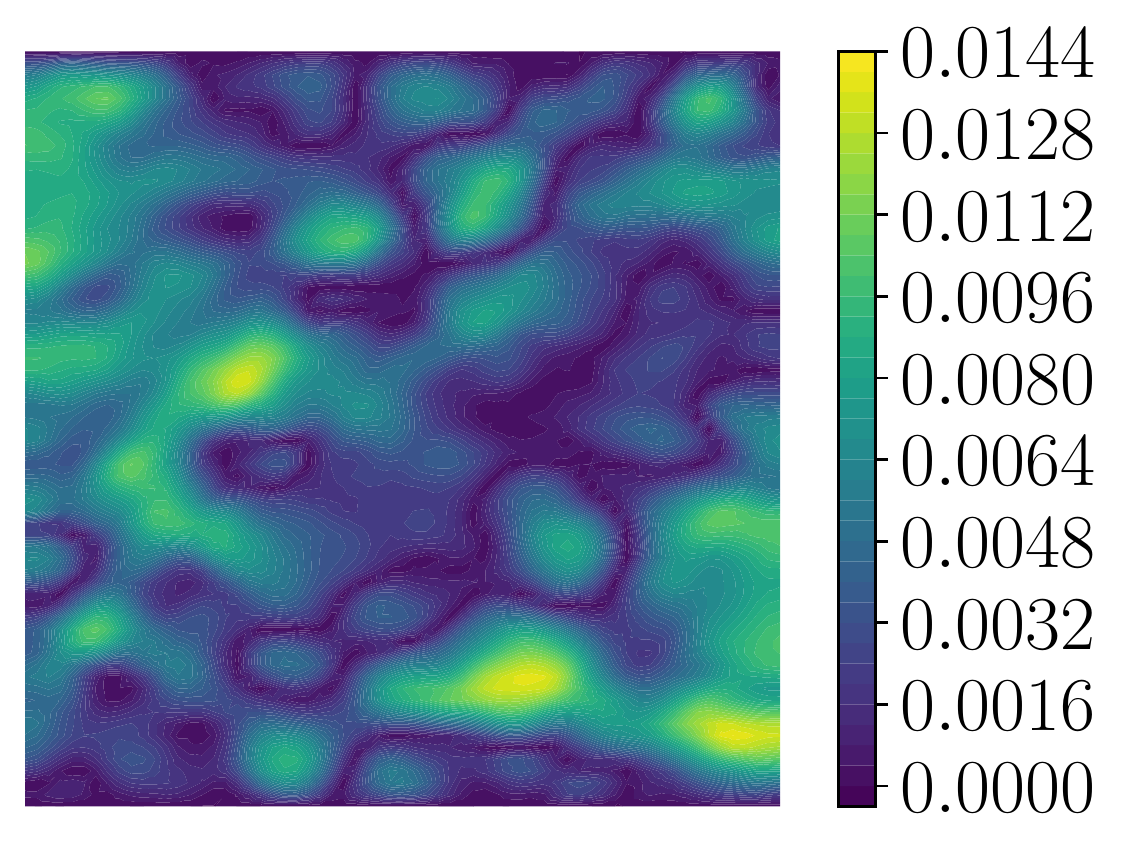} & \includegraphics[width = 0.19\linewidth]{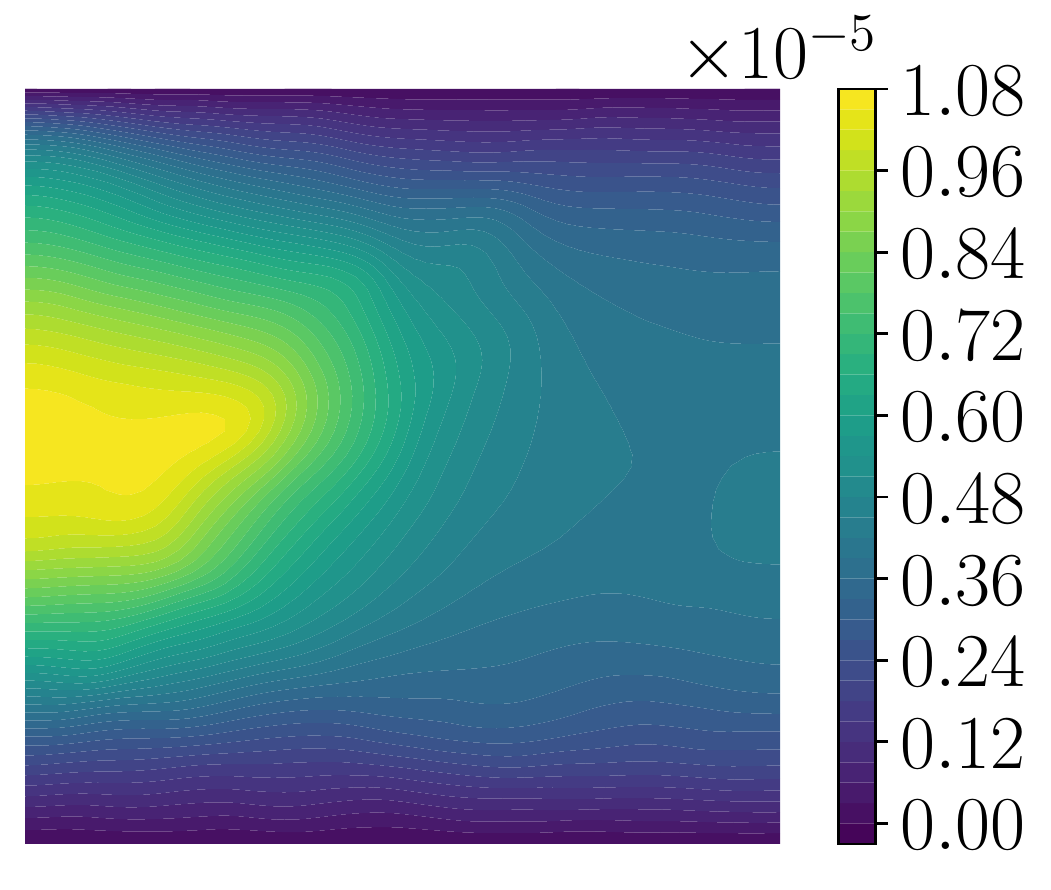}\\
        \includegraphics[width = 0.19\linewidth]{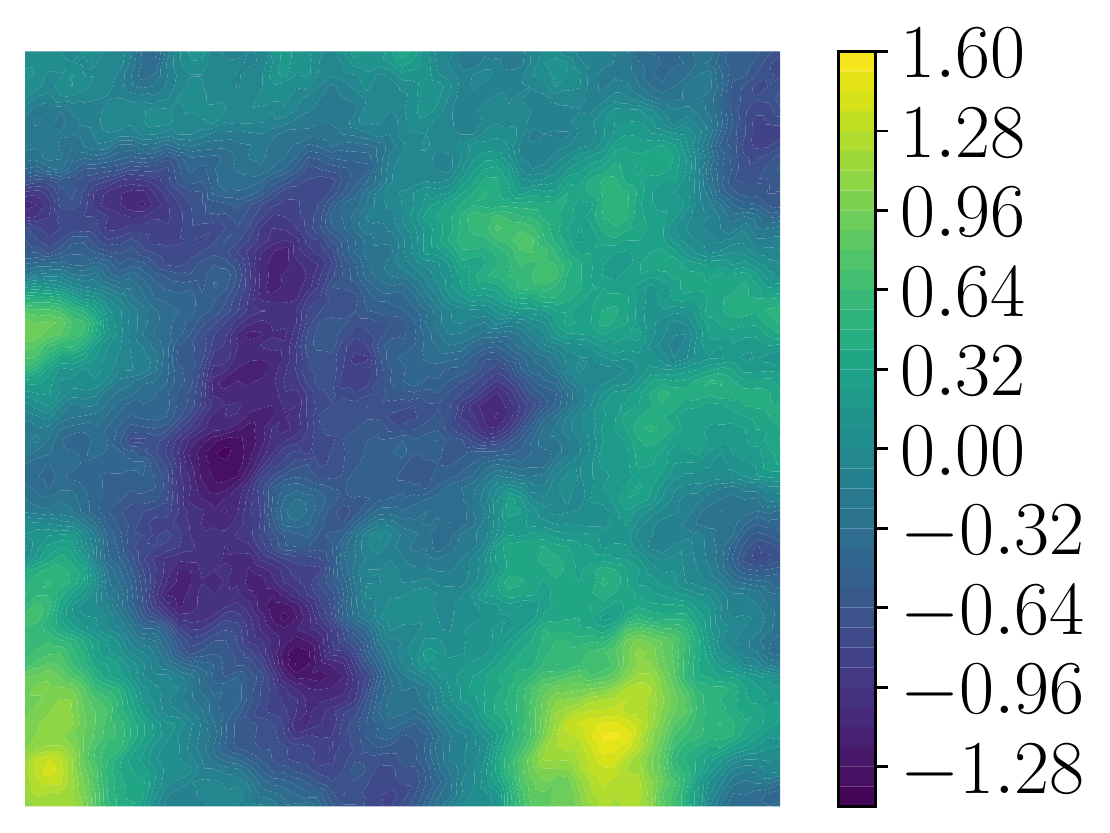} & \includegraphics[width = 0.19\linewidth]{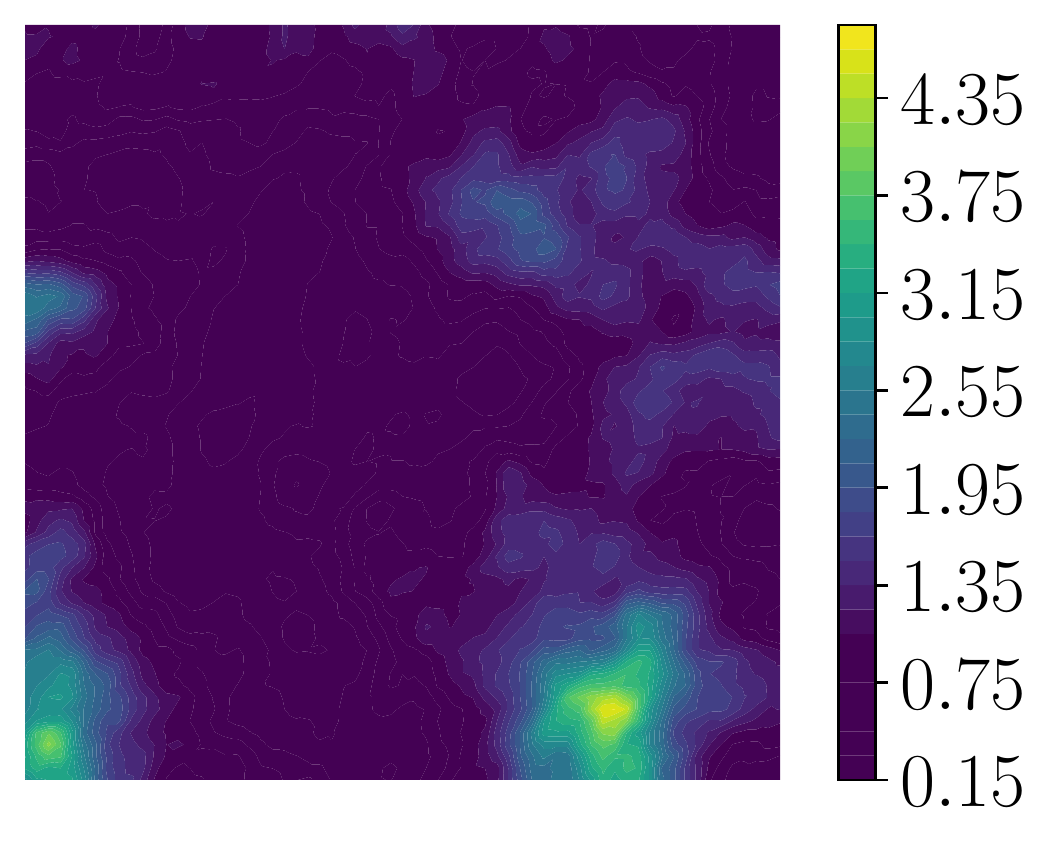} &\includegraphics[width = 0.19\linewidth]{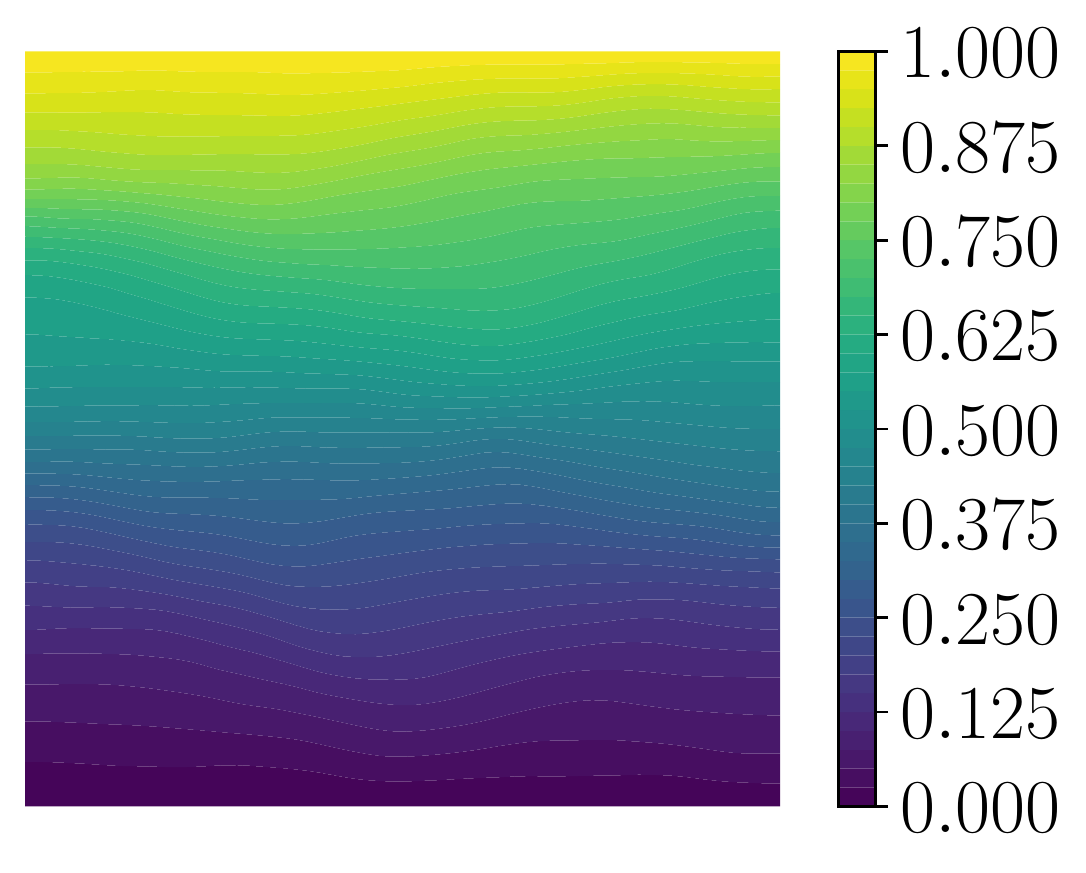} & \includegraphics[width = 0.19\linewidth]{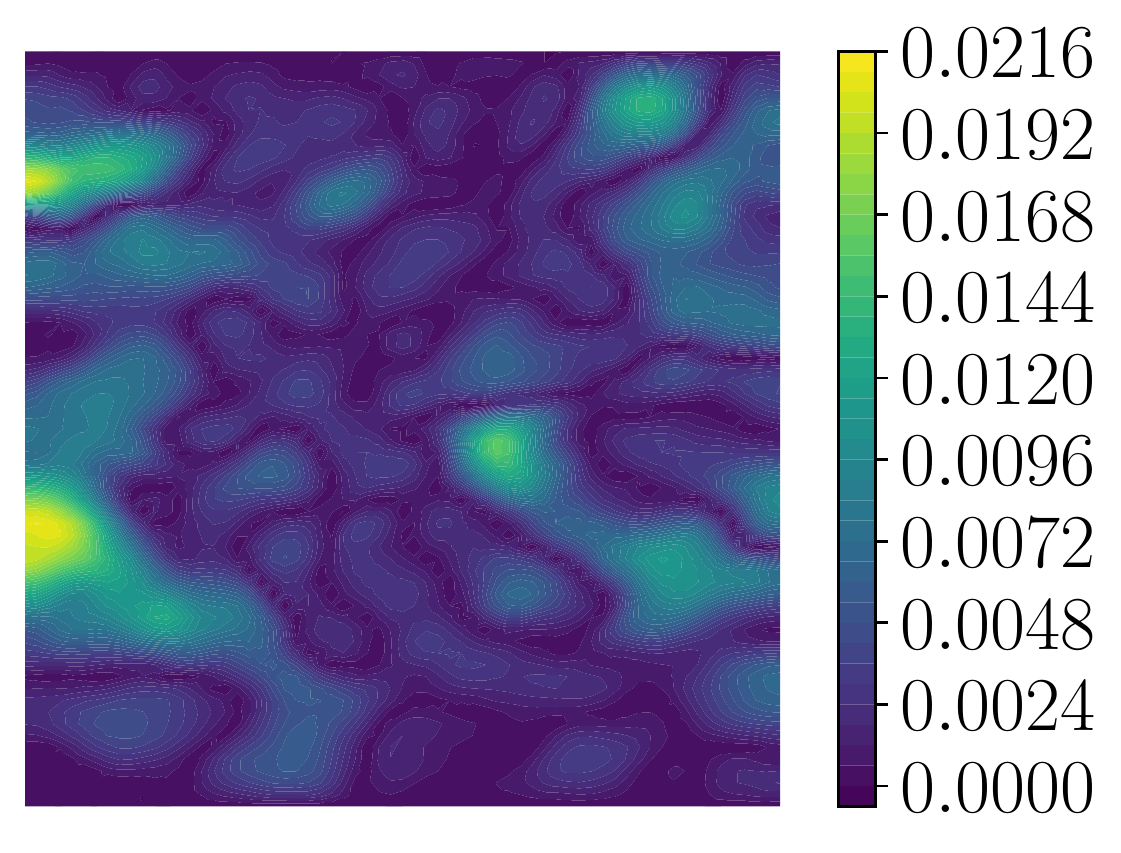} & \includegraphics[width = 0.19\linewidth]{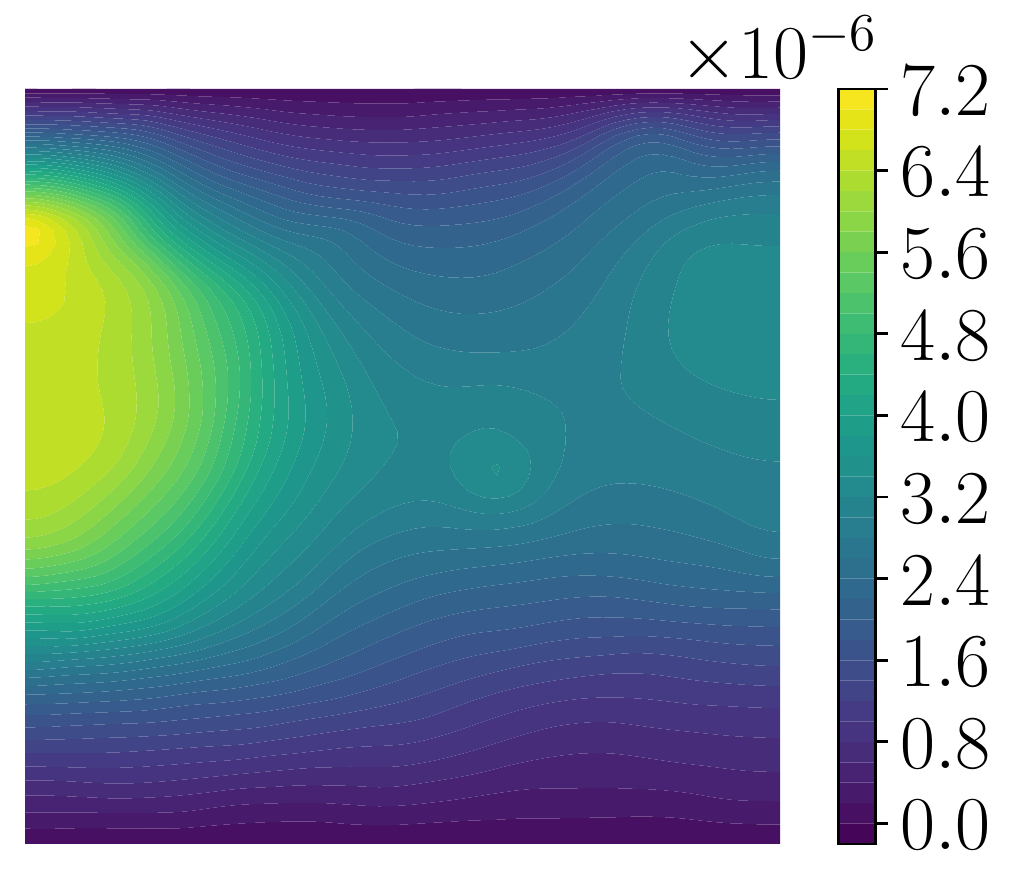}\\
    \end{tabular}
    \addtolength{\tabcolsep}{6pt} 
    \caption{Visualization of the prior samples, model solutions, and neural operator performance with and without error correction for the reaction--diffusion problem introduced in Section~\ref{ss:poisson}. From left to right, we have (i) three Gaussian random fields, $m_j$, $j=1,2,3$, sampled from the prior distribution $\nu_M$ with approximately a pointwise variance of $1$ and correlation length of $20\%$ of the domain size length, each occupies a row, (ii) the corresponding coefficient field samples, $\kappa_j$, defined by $\kappa_j = \exp(m_j)$, (iii) the corresponding finite element solutions of the reaction--diffusion problem, $\mathcal{F}^h(m_j)$, (iv) the absolute prediction errors at $m_j$ for the \textit{best performing} neural operator ($\sim90\%$ accurate as shown in Figure~\ref{fig:nonlinearpoisson_nns}), $|\mathcal{F}^h(m_j) - \operator_{\boldsymbol{w}}(m_j)|$, and (v) the absolute errors at $m_j$ for the error-corrected predictions based on the best performing neural operator.}
    \label{fig:poisson_prior}
\end{figure}

We consider the following parameter space, state space, and solution set for the nonlinear PDE problem above:
\begin{equation}
\begin{gathered}
    \mathcal{M} \coloneqq L^2(\Omega)\,,\quad
    \mathcal{U} \coloneqq H^1(\Omega)\,,\\
    \mathcal{U}_0 \coloneqq  \{u\in H^1(\Omega):u|_{\Gamma_b} = 0 \text{ and } u|_{\Gamma_t} = 0\}\,,\quad \mathcal{V}_u\coloneqq \{u\in H^1(\Omega):u|_{\Gamma_b} = 0 \text{ and } u|_{\Gamma_t} = 1\}\,,
\end{gathered}
\end{equation}
where the restriction to the boundary is defined with the trace operator. The variational problem for the model is given by:
\begin{equation}\label{eq:poissonWeakForm}
\begin{aligned}
    &\text{Given } m\in \mathcal{M}, \text{ find }u \in \mathcal{V}_u \text{ such that }\\
	&\qquad\dualDot{\mathcal{R}(u, m)}{v}\coloneqq\int_{\Omega} \exp(m)\grad u(\boldsymbol{x}) \cdot\grad v(\bx)\,\dd \bx + \int_\Omega u(\bx)^3 v(\bx)\,\dd\bx = 0\,,\quad\forall v\in\mathcal{U}_0\,.
\end{aligned}
\end{equation}

\subsubsection{Numerical approximation and neural operator performance} 
The numerical evaluation of the forward operator, $\mathcal{F}^h$, sampling of the prior distribution, $\nu_M^h$, and the residual-based error correction problem is implemented via the finite element method. In particular, the domain $\Omega$ is discretized with $64\times64$ cells of uniform linear and parabolic Lagrangian triangular elements that form the finite element spaces $\mathcal{M}^h\subset\mathcal{M}$ and $\mathcal{U}^h\subset\mathcal{U}$, respectively. The state finite element space has $d_u = 16641$ degrees of freedom and the parameter finite element space has $d_m = 4225$. The variational problems of the model, prior sampling, and error correction are then approximated and solved in these finite element spaces. We employ the Newton iteration for solving the nonlinear reaction--diffusion problem, which takes on an average of $2.5$ iterations to converge for parameter samples generated from the prior distribution.

Applying neural operator construction and training specified in Section~\ref{section:derivate_neural_operator} to the reaction--diffusion problem, we produced $7$ neural operators using increasing number of training samples, $n_{\text{train}} = 100$, $201$, $403$, $806$, $1382$, $3225$, $6912$, assuming $p=2$ in the Bochner norm. In Figure~\ref{fig:nonlinearpoisson_nns}, the accuracy of the trained neural operators at different numbers of training samples is shown. The accuracy number is computed according to~\eqref{eq:t_accuracy} using $512$ samples from the prior distribution that are unseen during training. The accuracy ceiling around $90\%$ is reached at $n_{\text{train}} = 1382$.

For each trained neural operator, the accuracy for the error-corrected neural operators using the same 512 samples is also computed and shown in Figure~\ref{fig:nonlinearpoisson_nns}. The error-corrected mapping for all $6$ trained operators are close to $100\%$ accurate. The visualization of absolution errors for the predictions by the best performing neural operator with $n_{\text{train}} = 6912$ and its error-corrected predictions at samples from the prior distribution are shown in Figure~\ref{fig:poisson_prior}. We observe that the error correction step leads to a drop of maximum absolute pointwise error from the order of $10^{-2}$ to the order of $10^{-5}$.

\begin{figure}[H]
\center
\includegraphics[width = 0.6\textwidth]{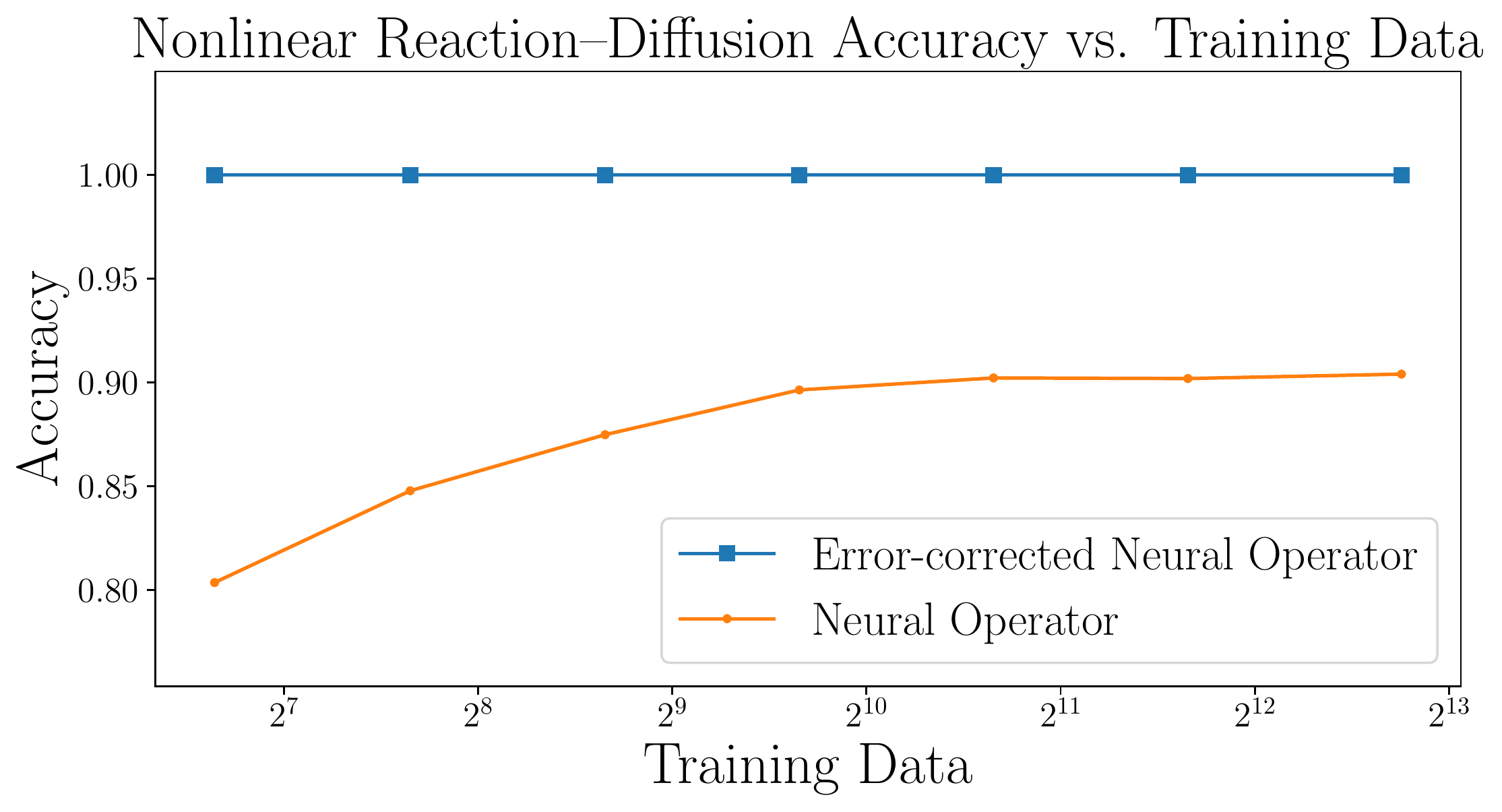}
\caption{A study of the generalization accuracy, as defined in~\eqref{eq:t_accuracy}, of $7$ neural operators trained using a varying number of training samples for the nonlinear reaction--diffusion problem. The neural operators are constructed using a derivative-informed projected ResNet; see Section 5.1 for details on its construction and training. Generalization accuracy is computed using $512$ data unseen during training. An empirical accuracy ceiling of $\sim 90\%$ is reached for the given neural network architecture.
}
\label{fig:nonlinearpoisson_nns}
\end{figure}

\subsubsection{Bayesian inverse problem setting}\label{subsec:poisson_inverse}
We consider a set of synthetic observation data $\boldsymbol{y}^*$ generated according to the data model in~\eqref{eq:data_model} for the reaction--diffusion problem at a synthetic parameter fields $m^*$. It has distinctive curvatures generated using a Rosenbrock function. We visualize the synthetic parameter $m^*$ as well as its corresponding coefficient field $\kappa^*$ and model solution $u^* = \mathcal{F}^h(m^*)$ in Figure~\eqref{fig:nonlinearpoisson_setting}.

To complete the data model, we define an observation operator and a noise distribution. Here we consider a linear observation operator that extracts discrete observations of the model solution at a uniform grid of $10\times 10$ points in the domain $\Omega$. Let $\{\boldsymbol{x}_j\}_{j=1}^{100}$ denote the observation points. Given a function $u\in\mathcal{U}$, the observation operator $\bdmc{B}(u) \in \R^{100}$ returns local averages of a given state around the observation points:
\begin{equation}
    \bdmc{B}(u) = \begin{bmatrix}
    \displaystyle |B_r(\bx_1)|^{-1}\int_{B_r(\bx_1)} u(\bx)\,\dd\bx &  \cdots & \displaystyle|B_r(\bx_{100})|^{-1}\int_{B_r(\bx_{100})} u(\bx)\,\dd\bx
    \end{bmatrix} \,,
\end{equation}
where $B_r(\bx_j)$ is a circle with a small radius $r>0$ centered at the observation point and $|B_r(\bx_j)|$ is the size of the circle. We assume the observation is corrupted with white noise, i.e., $\boldsymbol{N} \sim \mathcal{N}(\boldsymbol{0}, \sigma^2\boldsymbol{I})$, with a standard deviation of $\sigma = 0.0073$ that is $1\%$ of the maximum value in $\bdmc{B}(u^*)$.

\begin{figure}[H]
    \centering
    \addtolength{\tabcolsep}{-6pt} 
    \begin{tabular}{cccc}
        $m^*$ & $\kappa^*$ & $\mathcal{F}^h(m^*)$ & $\boldsymbol{y}^*\in\R^{100}$ \\
        \includegraphics[width = 0.24\linewidth]{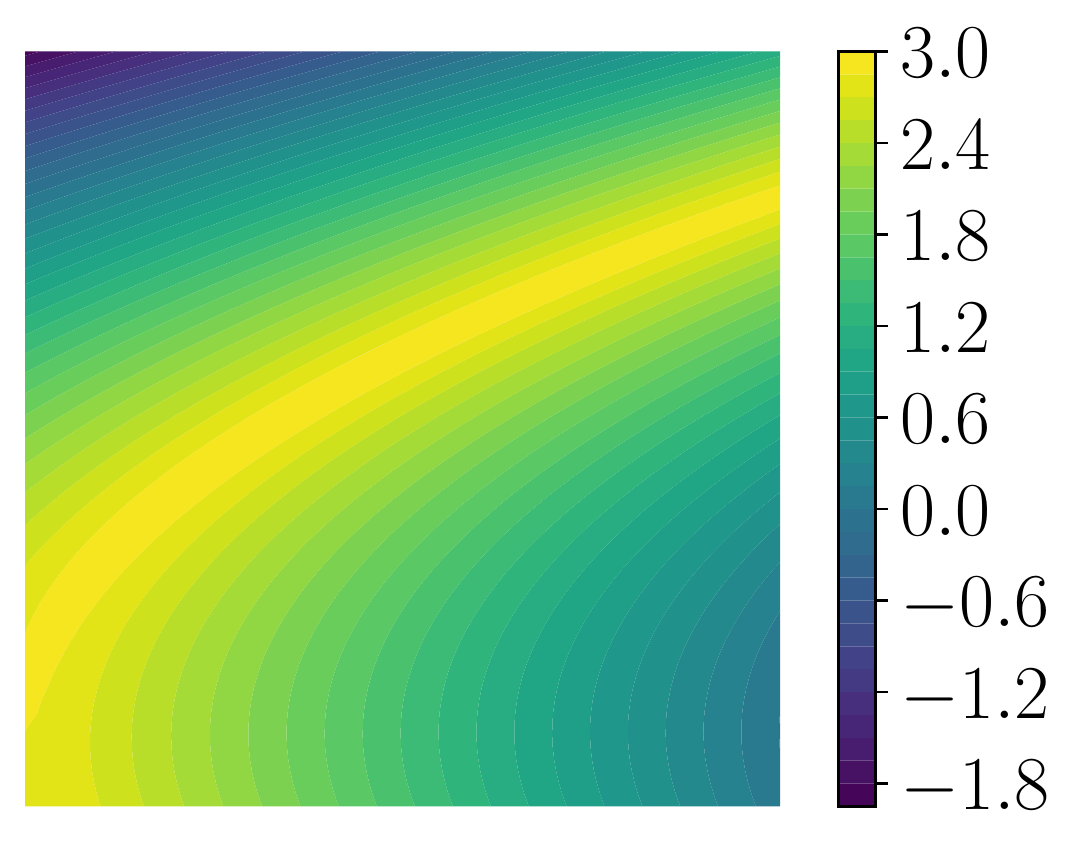} & \includegraphics[width = 0.24\linewidth]{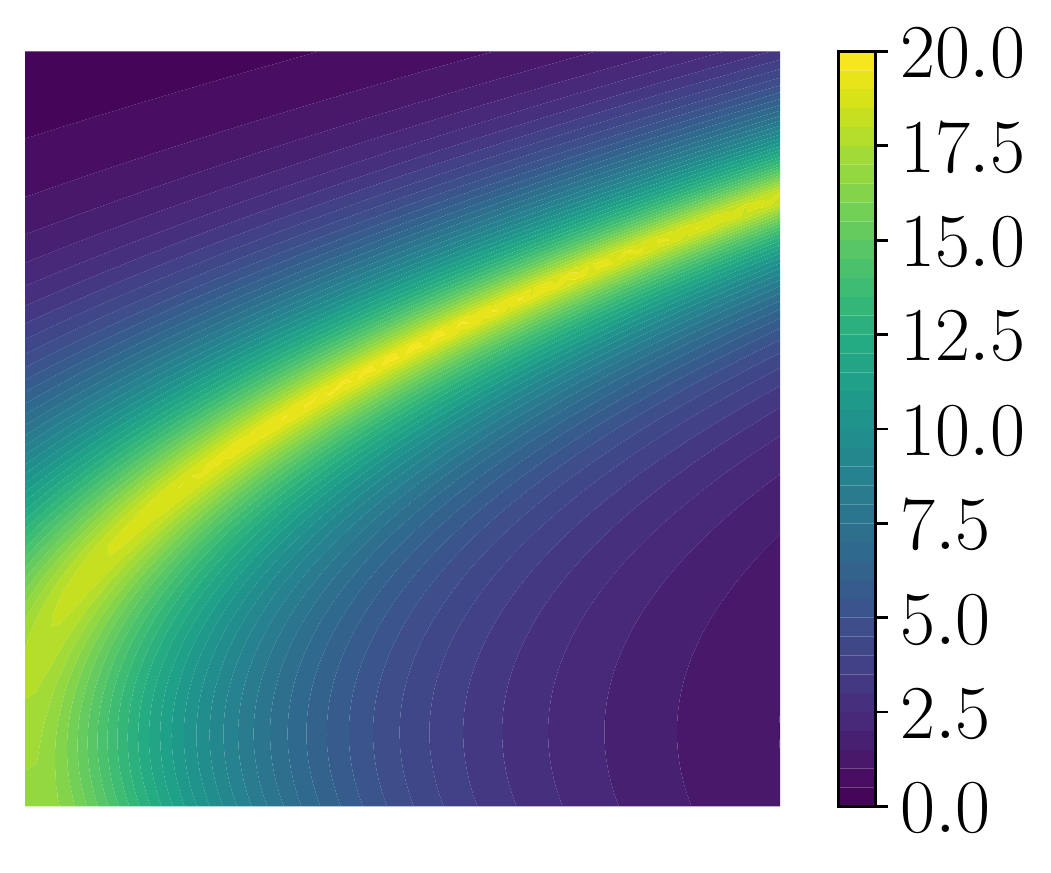} &
        
        \includegraphics[width = 0.24\linewidth]{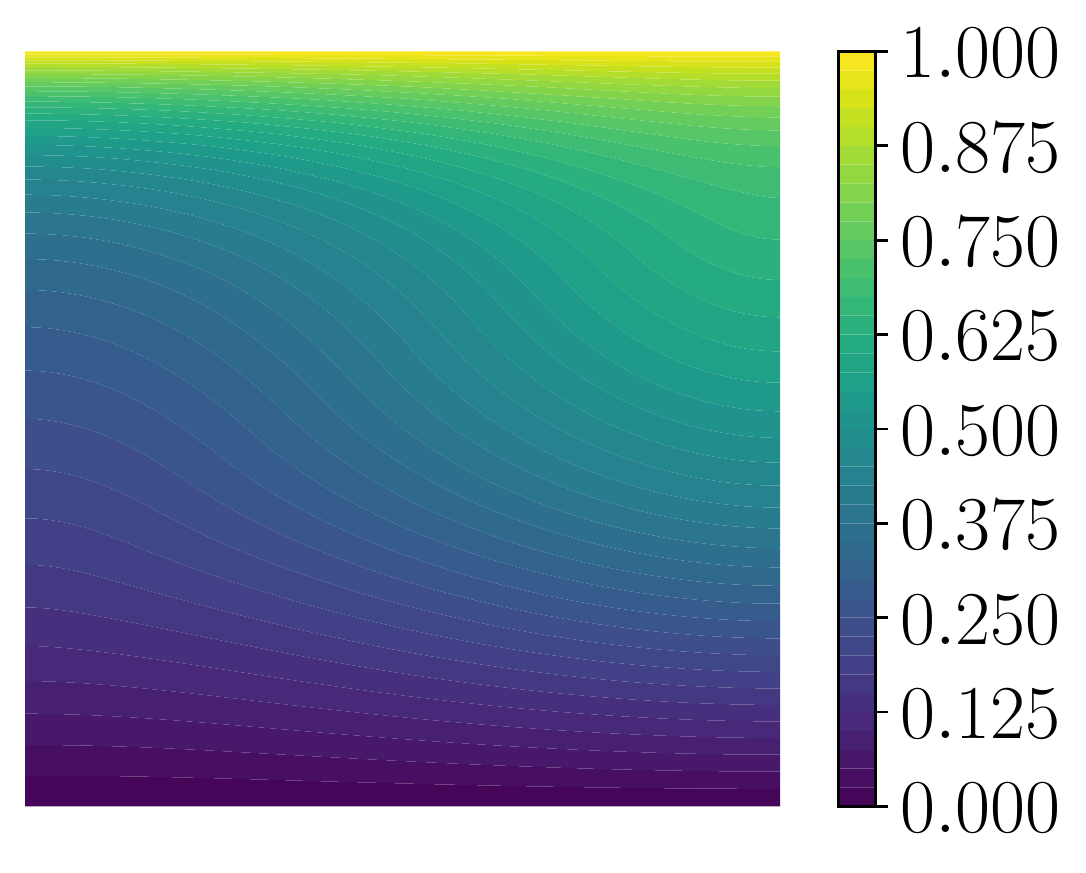} & \includegraphics[width = 0.24\linewidth]{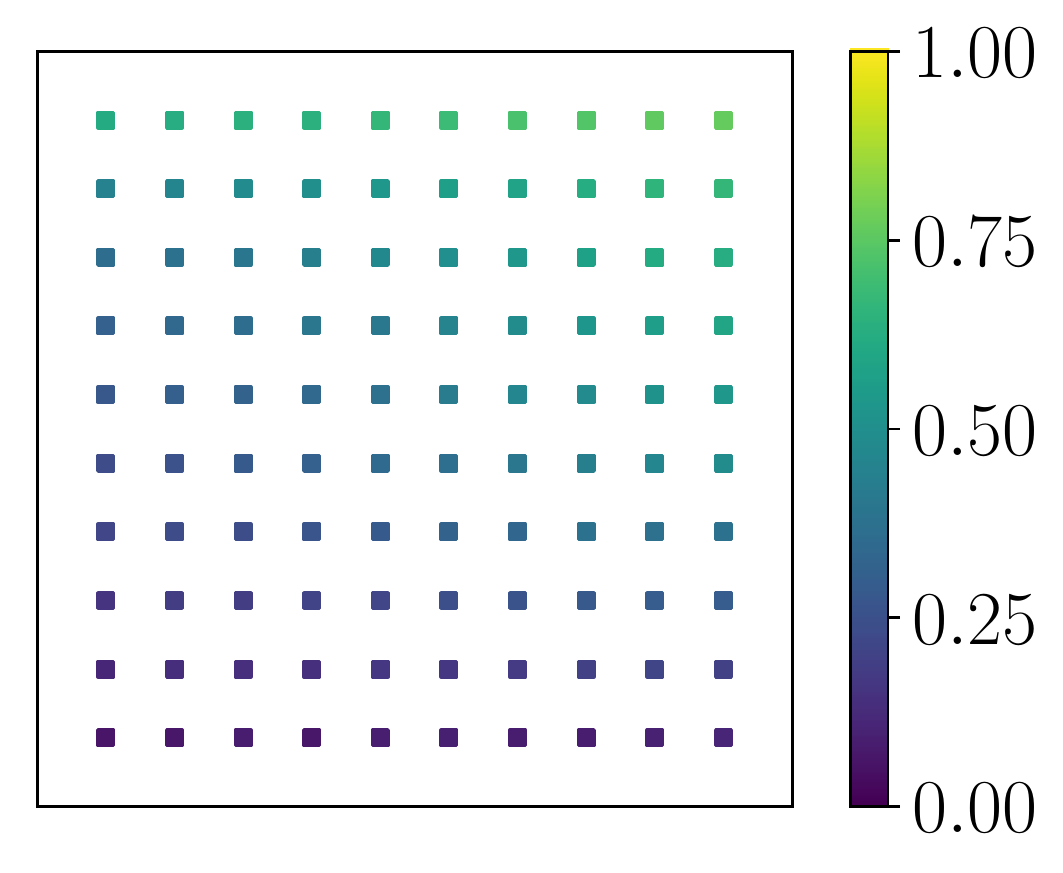}\\    
        \end{tabular}
    \caption{Visualization of the setting for a synthetic Bayesian inverse problem based on the nonlinear Poisson problem introduced in Section~\ref{subsec:poisson_inverse}. From left to right, we have (i) the synthetic parameter field, $m^*$, defined using a Rosenbrock function, (ii) the corresponding synthetic coefficient field, $\kappa^*$, (iii) the finite element solution at $m^*$, $\mathcal{F}^h(m^*)$, (iv) the synthetic observed data, $\by^*$, extracted from locally averaged values of $\mathcal{F}^h(m^*)$ at a $10\times10$ grid of observation points, corrupted by a randomly sampled additive white noise with a standard deviation of $0.0073$.}
        \addtolength{\tabcolsep}{6pt} 
    \label{fig:nonlinearpoisson_setting}
\end{figure}

\subsubsection{Posterior visualization and cost analysis}
To visualize and compare posterior distributions with likelihood functions evaluated with the model via the finite element method, neural operators, and error-corrected neural operators, we generate samples from these posterior distributions via the pCN algorithm introduced in Section 2.3, and visualize their posterior predictive means of the coefficient field by sample average approximation,
\begin{equation}\label{eq:poisson_pred_mean}
    \kappa_{\text{mean}} \approx \frac{1}{n_{\text{post}}}\sum_{j=1}^{n_{\text{post}}} \exp(m_j)\,,
\end{equation}
where $\{m_j\in\mathcal{M}^h\}_{j=1}^{n_{\text{post}}}$ are posterior samples. For each posterior distributions, $8$ MCMC chains are constructed with a mixing parameter of $\beta_{\text{pCN}} = 0.03$ and samples of total $n_{\text{post}} = 120,000$ are collected. While the mixing of the chains is seen to be rapid, a conservative burn-in rate of $25\%$ is used. The average sample acceptance rate for MCMC sampling using the model is around $20\%$. 

In Figure~\ref{fig:poisson_no_mean}, we visualize the model generated posterior predictive mean estimate of the coefficient field alongside the ones generated by the three best-performing neural operators. We observe that the estimates by these neural operators are unable to recover the distinctive curvature of the coefficient field captured by the model estimate, even though they are near the accuracy ceiling of neural operator training where the accuracy increment is quickly diminishing with respect to the size of the training data. In Figure~\ref{fig:poisson_correction_mean}, we provide the same visualization of the estimates generated by the three neural operators with error correction. We observe that they produce estimates that are much more consistently similar to the one generated by the model.

\begin{figure}[H]
    \centering
    \addtolength{\tabcolsep}{-6pt} 
    \begin{tabular}{cccc}
    \textbf{Model $\mathcal{F}^h$} & \multicolumn{3}{c}{\textbf{Neural operator} $\operator_{\boldsymbol{w}}$}\\
     &   $n_{\text{train}} = 1382$ & $n_{\text{train}} = 3225$ & $n_{\text{train}} = 6912$  \\
    \includegraphics[width = 0.23\linewidth]{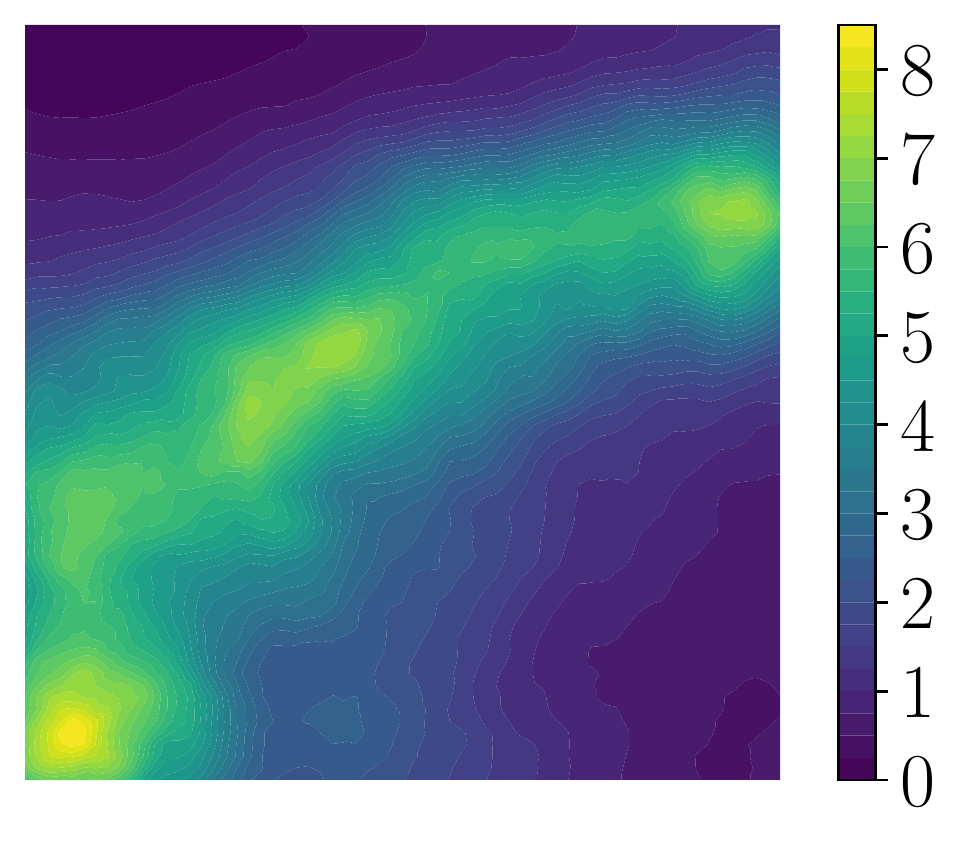} & \includegraphics[width = 0.25\linewidth]{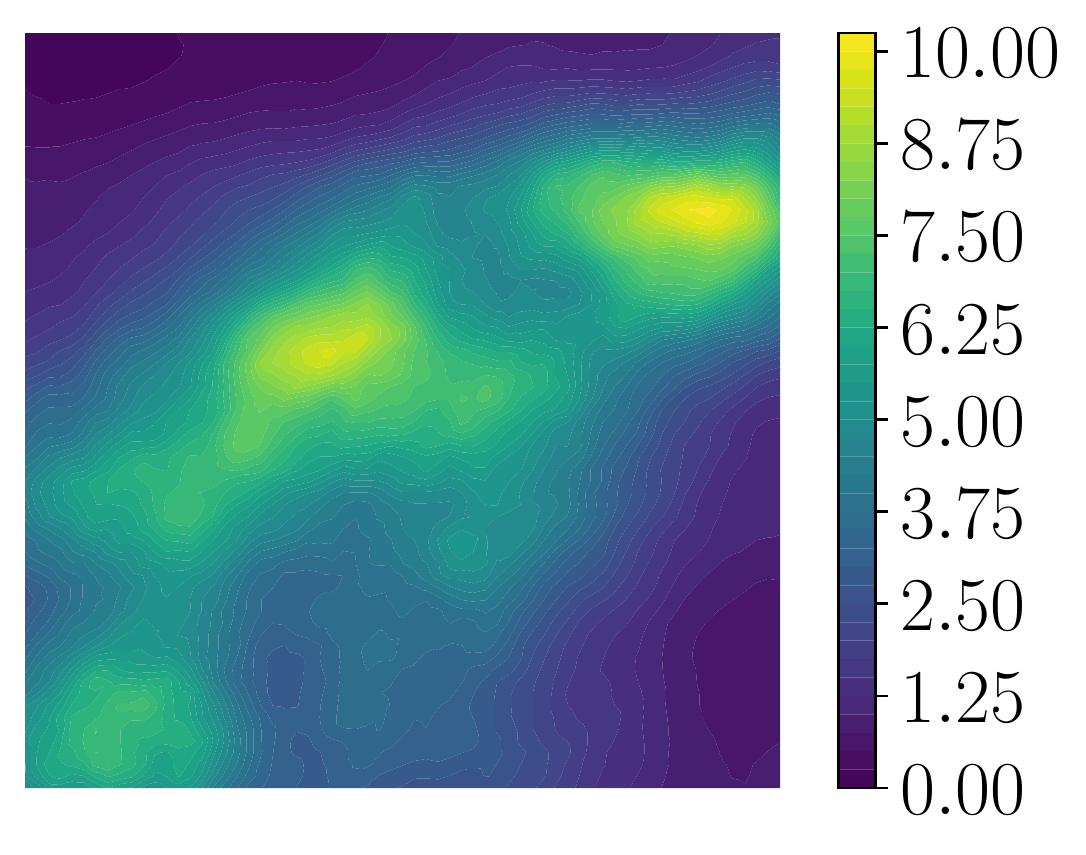}  & \includegraphics[width = 0.26\linewidth]{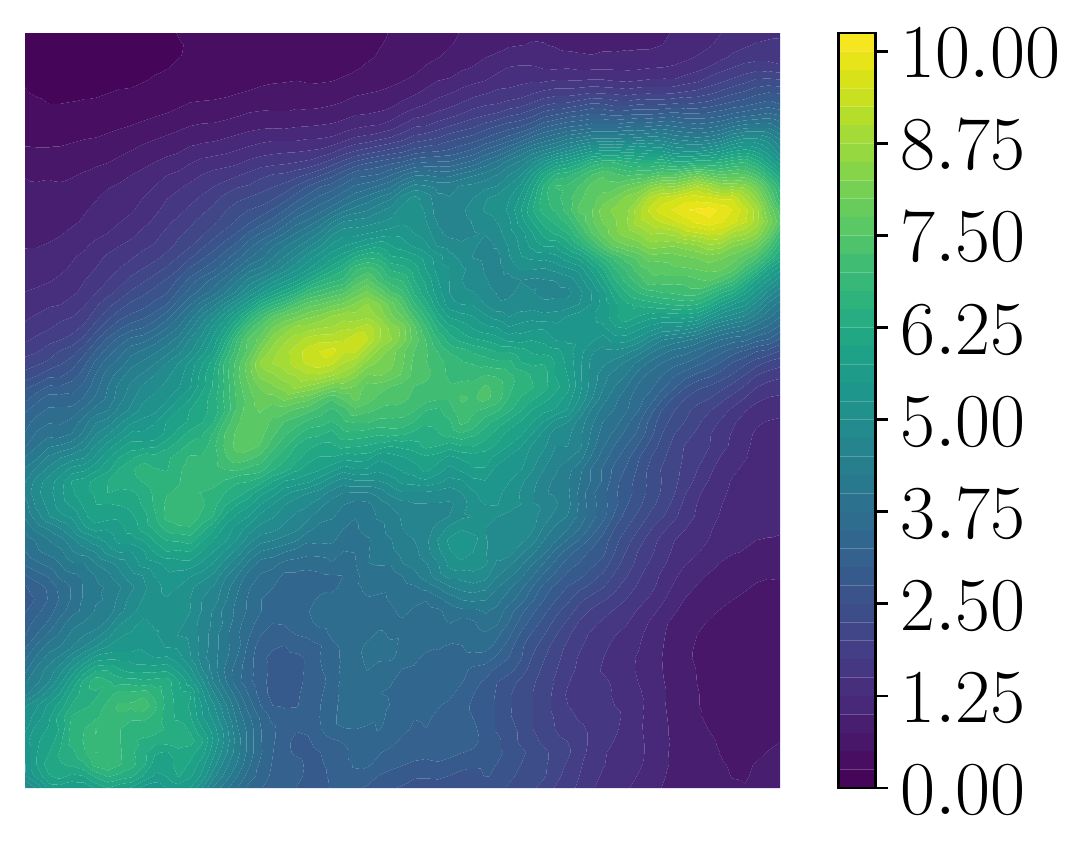} &  \includegraphics[width = 0.23\linewidth]{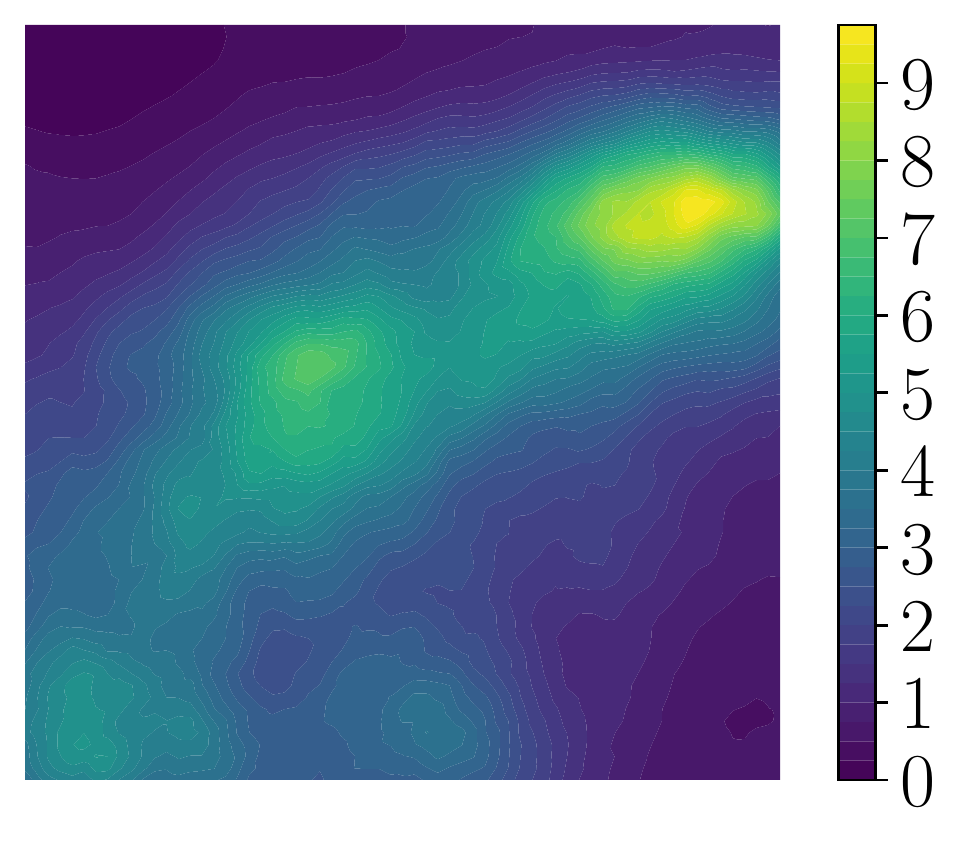}\\
    \end{tabular}
    \addtolength{\tabcolsep}{6pt} 
    \caption{Visualization of posterior predictive mean estimates in~\eqref{eq:poisson_pred_mean} of the coefficient field $\kappa$ for a synthetic Bayesian inverse problem introduced in Section~\ref{subsec:poisson_inverse}. From left to right, we have the estimates by (i) the model via the finite element method, (2) the neural operator trained with $n_{\text{train}} = 1382$, (3) the neural operator trained with $n_{\text{train}} = 3225$, (4) and the neural operator trained with $n_{\text{train}} = 6912$. The accuracy of the neural operators is around $90\%$ as shown in Figure~\ref{fig:nonlinearpoisson_nns}.}
    \label{fig:poisson_no_mean}
\end{figure}

\begin{figure}[H]
    \centering
    \addtolength{\tabcolsep}{-6pt} 
    \begin{tabular}{cccc}
    \textbf{Model $\mathcal{F}^h$} & \multicolumn{3}{c}{\textbf{Neural operator with error correction $\operator_C$}}\\
     &   $n_{\text{train}} = 1382$ & $n_{\text{train}} = 3225$ & $n_{\text{train}} = 6912$  \\
    \includegraphics[width = 0.24\linewidth]{model_mean_coefficient.pdf} & \includegraphics[width = 0.24\linewidth]{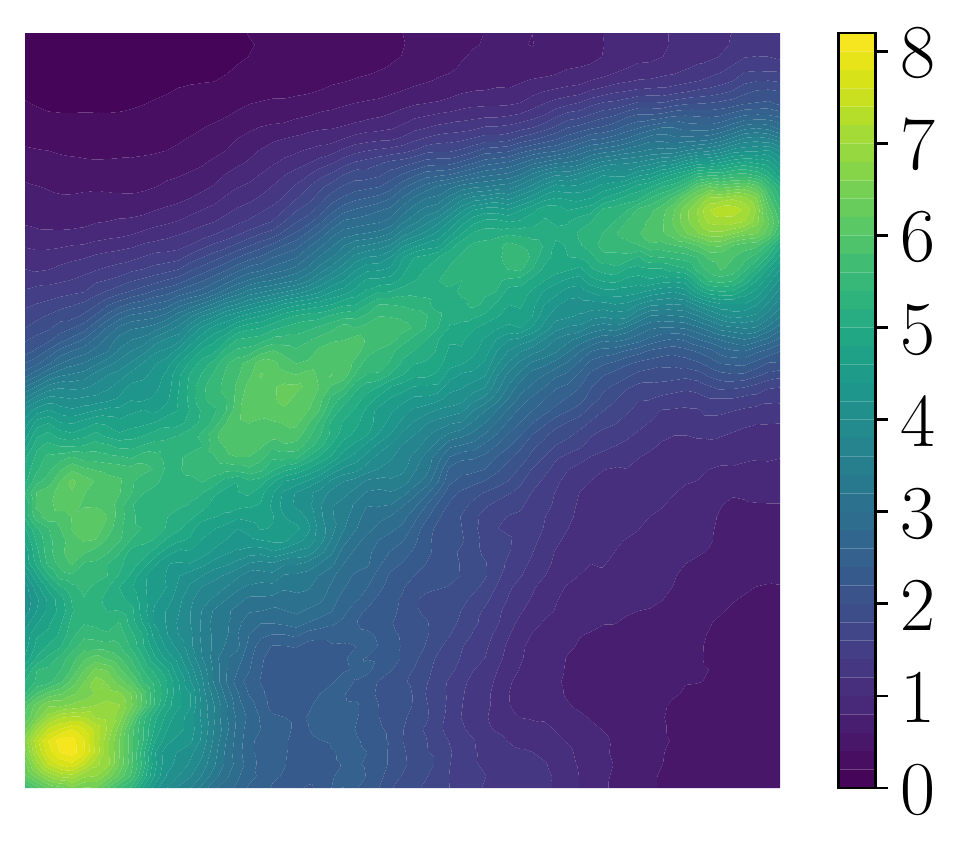} & \includegraphics[width = 0.24\linewidth]{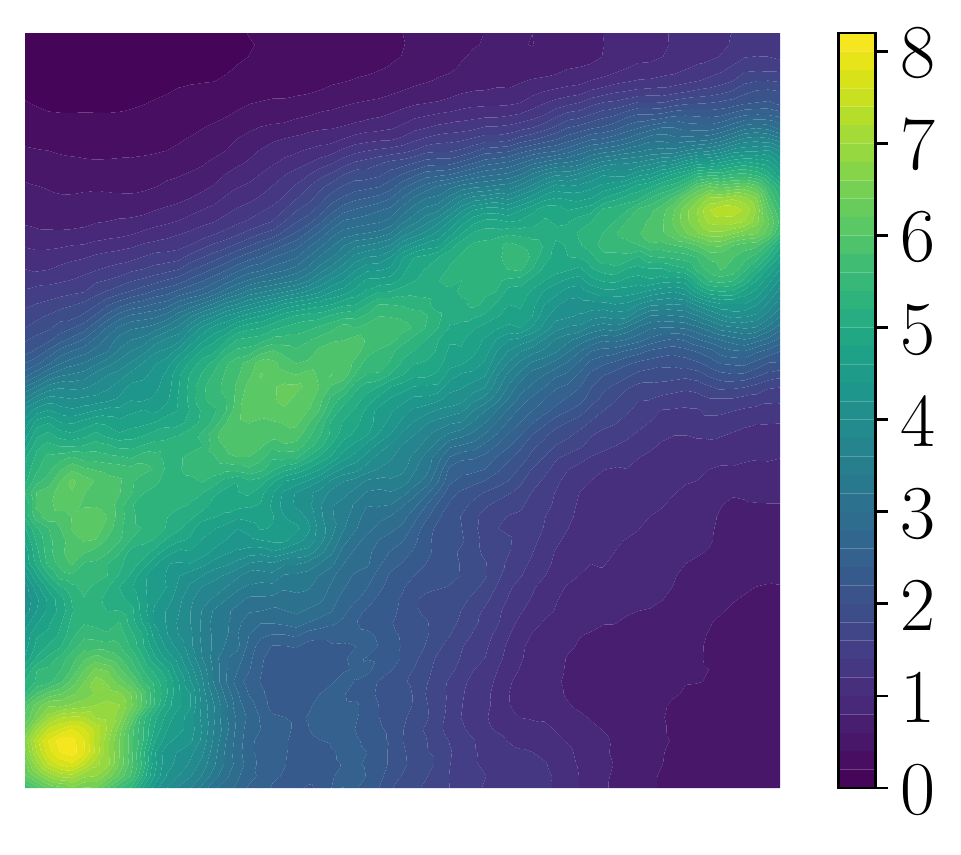} &  \includegraphics[width = 0.24\linewidth]{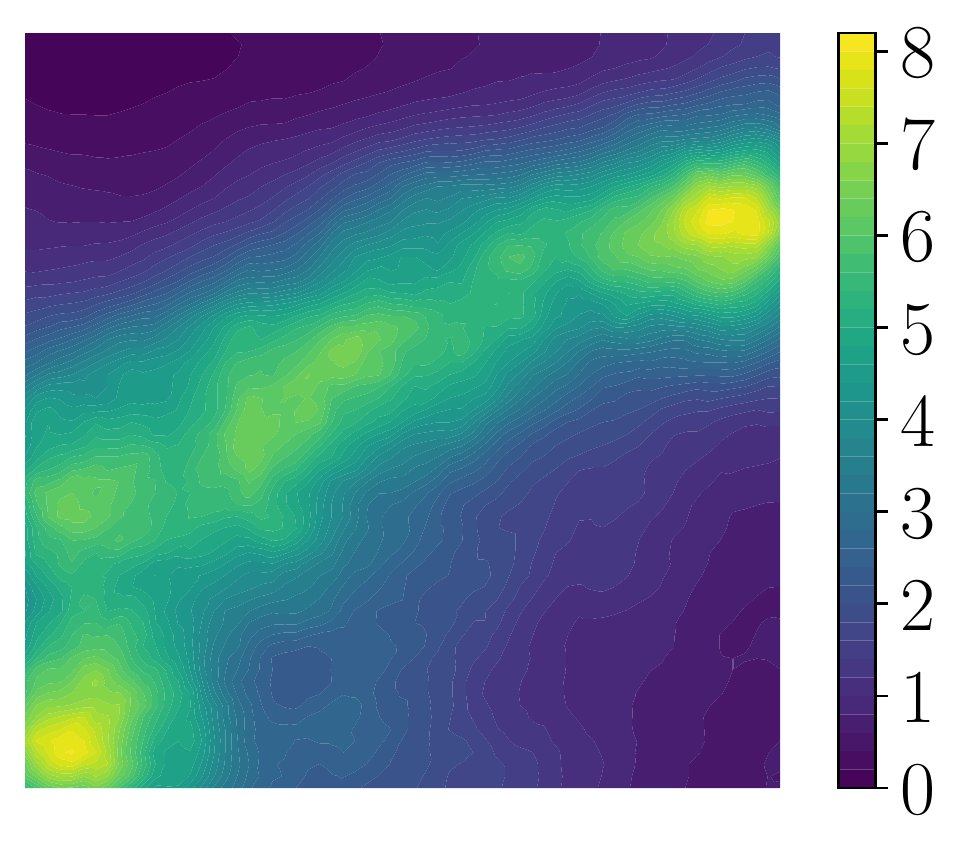}\\
    \end{tabular}
    \addtolength{\tabcolsep}{6pt} 
    \caption{Visualization of posterior predictive mean estimates of the coefficient field in~\eqref{eq:poisson_pred_mean} for a synthetic Bayesian inverse problem introduced in Section~\ref{subsec:poisson_inverse}. From left to right, we have the estimates by (i) the model via the finite element method, (2) the neural operator trained with $n_{\text{train}} = 1382$ with error correction, (3) the neural operator trained with $n_{\text{train}} = 3225$ with error correction, (4) and the neural operator trained with $n_{\text{train}} = 6912$ with error correction. The accuracy of the neural operators with error correction is close to $100\%$ as shown in Figure~\ref{fig:nonlinearpoisson_nns}.}
    \label{fig:poisson_correction_mean}
\end{figure}

In Figure~\ref{fig:nonlinearpoisson_speedup}, we visualize the observed and asymptotic speedups for the posterior sampling using the $7$ trained neural operators with or without the error correction. The asymptotic speedup, as defined in Section 4.4, assumes $n_{\text{chain}}\to\infty$ in the sense that a long Markov chain is generated or repetitive use of the trained neural operators in different problems. As a result, the offline costs of the neural operator construction and training are neglected. The asymptotic speedup of the error-corrected neural operators is about $2.5$, which is the number of Newton iterations for solving the nonlinear problem, averaged over the posterior distribution. The asymptotic speedup of the neural operators is over two orders of magnitude. The observed speedups additionally account for the offline cost of reduced basis approximation, training data generation, and optimization. The finite $n_{\text{chain}}$ used for the posterior sampling presented above is used for the computation of the observed speedups. As a result, the observed speedups for the neural operator decay substantially as the number of training data increases. We observe almost an order of magnitude drop of speedup for $n_{\text{train}} = 6912$ compared to $n_{\text{train}} = 100$ A similar decay is observed for the error-corrected neural operators, but the dominant cost remains the cost of solving the linear systems associated with the error correction steps.

\begin{figure}[H]
\centering
\includegraphics[width = 0.9\textwidth]{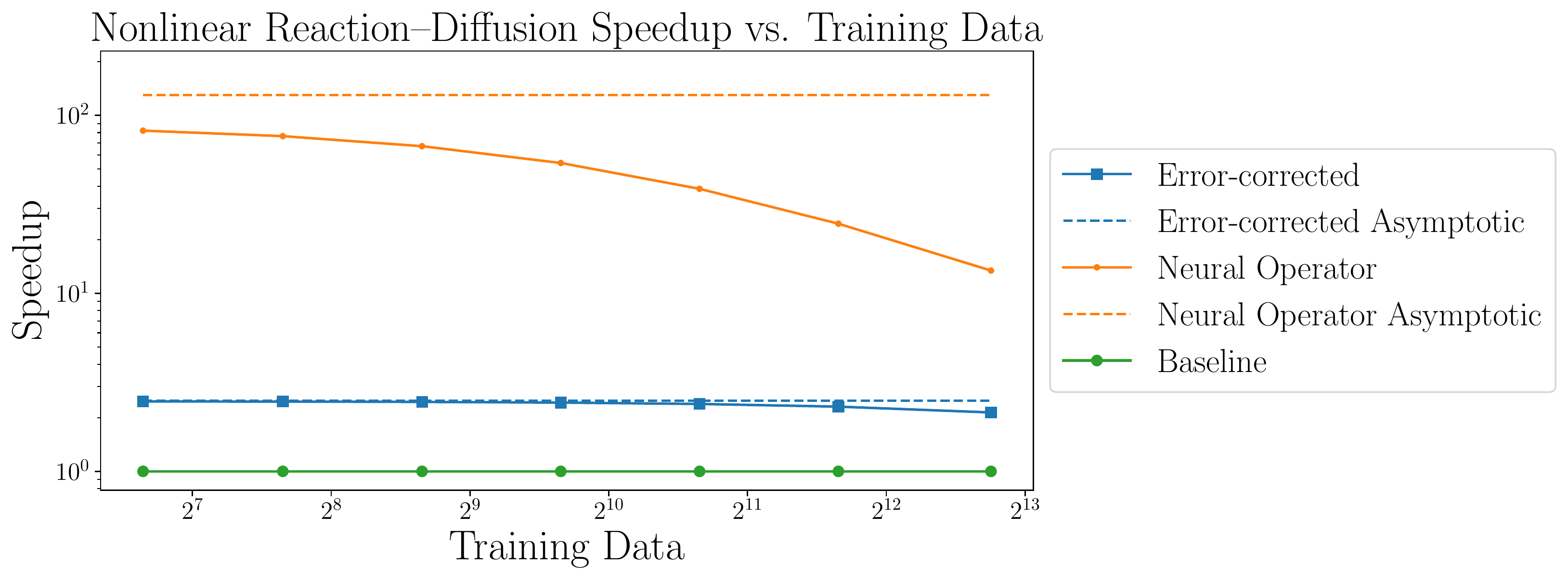}
\caption{The observed and asymptotic speedups, as defined in Section 4.4, for the posterior sampling via pCN for the neural operators and the error-corrected neural operators for the nonlinear reaction--diffusion problem. The asymptotic speedups assume $n_{\text{chain}}\to\infty$, thus neglecting the offline cost of neural operator construction and training. The observed speedups additionally account for the offline cost and the total number of iterative solves within generated Markov chains used for the posterior visualization in Figure~\ref{fig:poisson_no_mean} and~\ref{fig:poisson_correction_mean}.}
\label{fig:nonlinearpoisson_speedup}
\end{figure}

\subsection{Hyperelastic material properties discovery}\label{ss:hyperelastic}
Here we introduce the physical and mathematical setting for hyperelastic material properties discovery. This is a problem that has attracted many research interests for its important role in various engineering and medical applications~\cite{gokhale2008solution, goenezen2011solution, affagard2015identification, mei2018comparative}. We consider an experimental scenario where a square thin film of a hyperelastic material is fixed on one edge, with a traction force applied on the opposite edge, leading to a deformation of the material. We attempt to infer the material properties as spatially-varying functions from noisy measurements of material displacement at discrete positions via Bayes' rule. The schematic of the problem setup is shown in~\ref{fig:hyperSetup}.

\begin{figure}[H]
    \centering
    \includegraphics[width=0.35\linewidth]{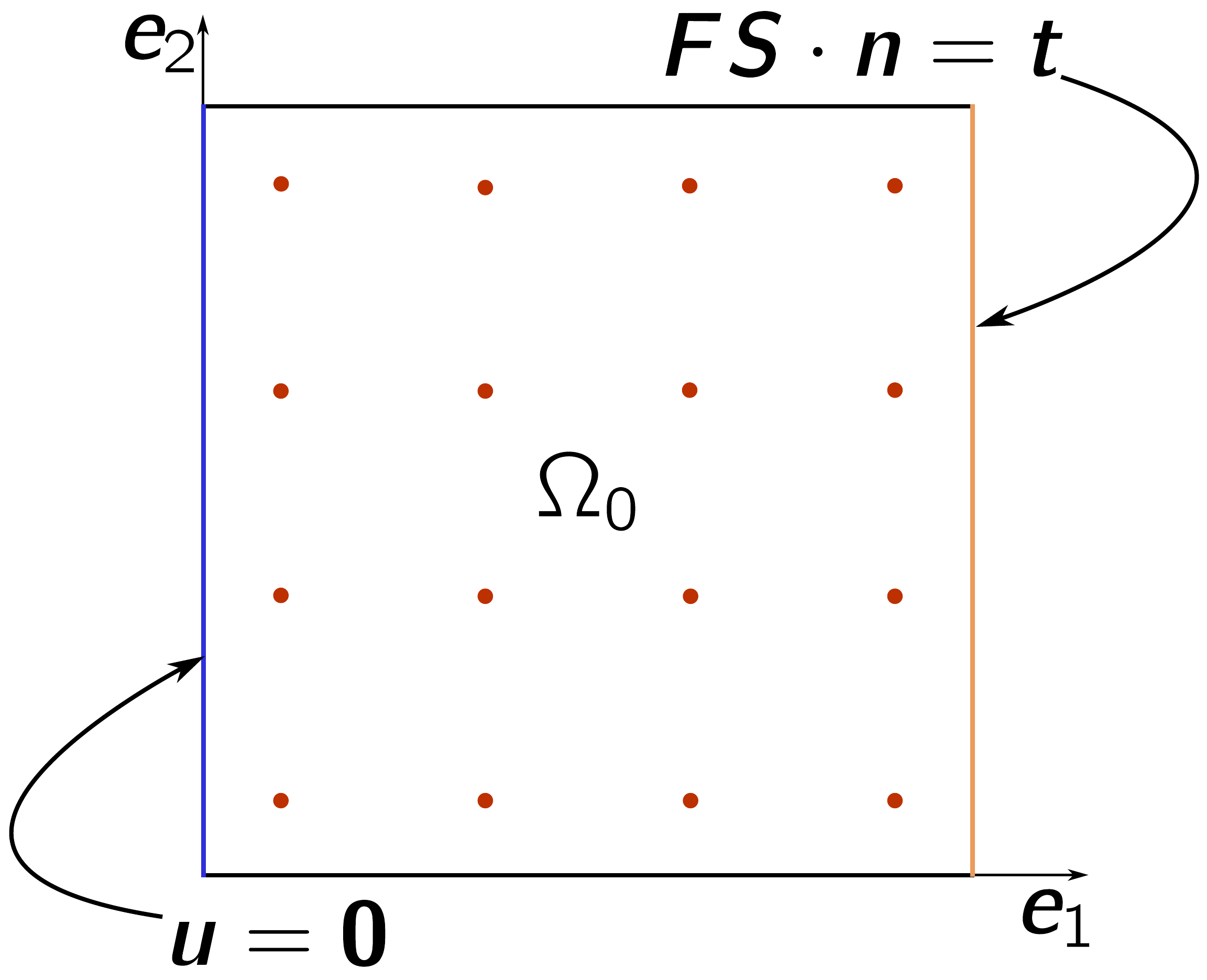}
    \caption{Schematics of a numerical experiment involving deformation of a hyperelastic material. Here the red dots are the points on the reference domain at which displacement (both components) are observed. On the left edge, the material is clamped so that the displacement is zero at this edge. On the right edge, traction in the form of a function $\bt = \bt(\bX)$ is prescribed. The top and bottom edges are traction free.}
    \label{fig:hyperSetup}
\end{figure}

\subsubsection{A model for hyperelastic material deformation}\label{subsubsec:hyperelastic_model}
Let $\Omega_0 = (0,1)^2$ be a normalized unit square material domain for the material of interest under the thin film approximation, and $\bX \in \Omega_0$ denote the material point. The current configuration of the material after deformation is represented by a map $\bchi: \Omega \to \R^2$. The material point $\bX$ is mapped to a spatial point $\boldsymbol{x} = \bchi(\bX) = \bX + \bu(\bX)$, where $\bu  = \bu(\bX)$ is the displacement of material points. As material deforms relative to a given reference configuration, internal forces are developed within the body. These internal forces depend on the underlying stored internal energy; for a hyperelastic material, there is a strain energy density, $W = W(\bX, \bC)$, as a function of material coordinate $\bX$ and the right Cauchy--Green strain tensor $\bC = \bF^T \bF$, $\bF = \nabla \bchi = \bI + \nabla \bu$ being the deformation gradient. The stress in the reference configuration, $\bS = \bS(\bX, \bC)$ is the \textit{second Piola--Kirchhoff stress tensor} given by
\begin{equation}
\bS(\bX, \bC) = 2\frac{\partial W(\bX, \bC)}{\partial \bC}\,.
\end{equation}
We consider the model for neo-Hookean materials \cite{jog2009energy, jog2015continuum} for which strain energy density function takes the form
\begin{equation}
	W(\bX, \bC) = \frac{\mu(\bX)}{2} (\mathrm{tr}(\bC) - 3) + \frac{\lambda(\bX)}{2} \left(\ln(J)\right)^2 - \mu(\bX) \ln(J)\,,
\end{equation}
where $\mathrm{tr}(\bC)$ is the trace of the second order tensor $\bC$, and $J$ is the determinant of the deformation gradient $\bF$. Here $\lambda$ and $\mu$ are the so-called \textit{Lam\'e parameters} which we assume to be related to Young's modulus of elasticity, $E$, and Poisson ratio, $\nu$, as follows:
\begin{equation}
\lambda = \frac{E \nu}{(1+\nu)(1-2\nu)}\,, \qquad \mu = \frac{E}{2(1+\nu)}\,.
\end{equation}
In what follows, we assume prior knowledge of (i) Poisson ratio $\nu = 0.4$, and (ii) an epistemically uncertain and spatially-varying Young's modulus, $E:\Omega_0\to\R_+$, that follows a log-normal prior distribution:
\begin{equation}
    E = \exp(M) + E_L\,,\quad M\sim \nu_M\coloneqq\mathcal{N}(m_{\text{pr}}, \mathcal{C}_{\text{pr}})\,,
\end{equation}
where $E_L>0$ is a lower bound on the pointwise value of the random function, normalized to have a value of $1$, and $m_{\text{pr}} = 0.37$ is a constant over $\Omega_0$. The goal of the Bayesian inverse problem is to learn the material properties $E$ through the parameter random field $M$ from observed displacements via Bayes' rule.

The prior distribution of $M$ is specified with a mean of a constant value $0.36$ and covariance constructed according to~\eqref{eq:gaussian_prior} with hyperparameters of $d = 2$, $\alpha = 4/3$, $\beta =0.12$, and $\boldsymbol{\Theta} = \boldsymbol{I}$. They correspond to Gaussian random fields with a pointwise variance of approximately $1$, a correlation length of approximately $0.3$, and small boundary artifacts. In Figure~\ref{fig:hyper_elasticity_prior}, we visualize several samples of the prior and their corresponding Young's modulus fields.

\begin{figure}[H]
    \centering
    \addtolength{\tabcolsep}{-6pt} 
    \begin{tabular}{ccccc}
        $M\sim \nu_M$ & $E$ & $\mathcal{F}^h(M)$ & $|\mathcal{F}^h(M) - \operator_{\boldsymbol{w}}(M)|$ & $|\mathcal{F}^h(M) - \operator^h_{C}(M)|$\\
        \includegraphics[width = 0.19\linewidth]{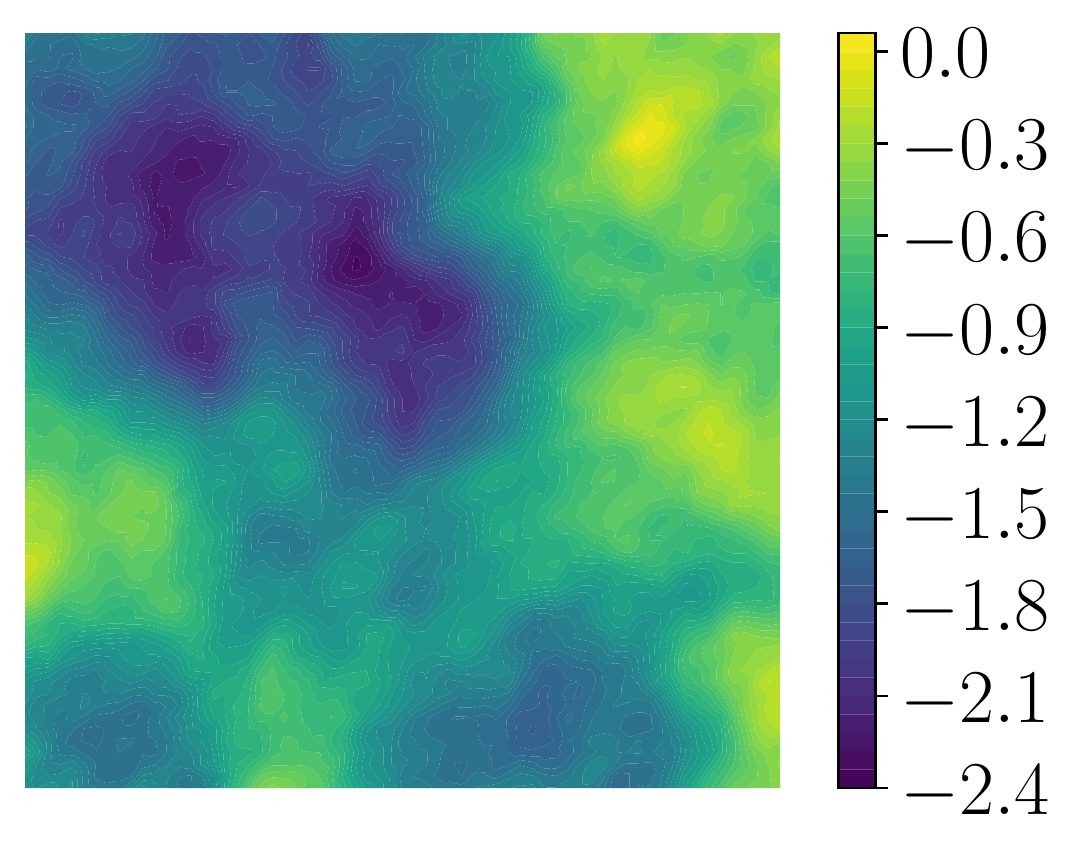}  & \includegraphics[width = 0.19\linewidth]{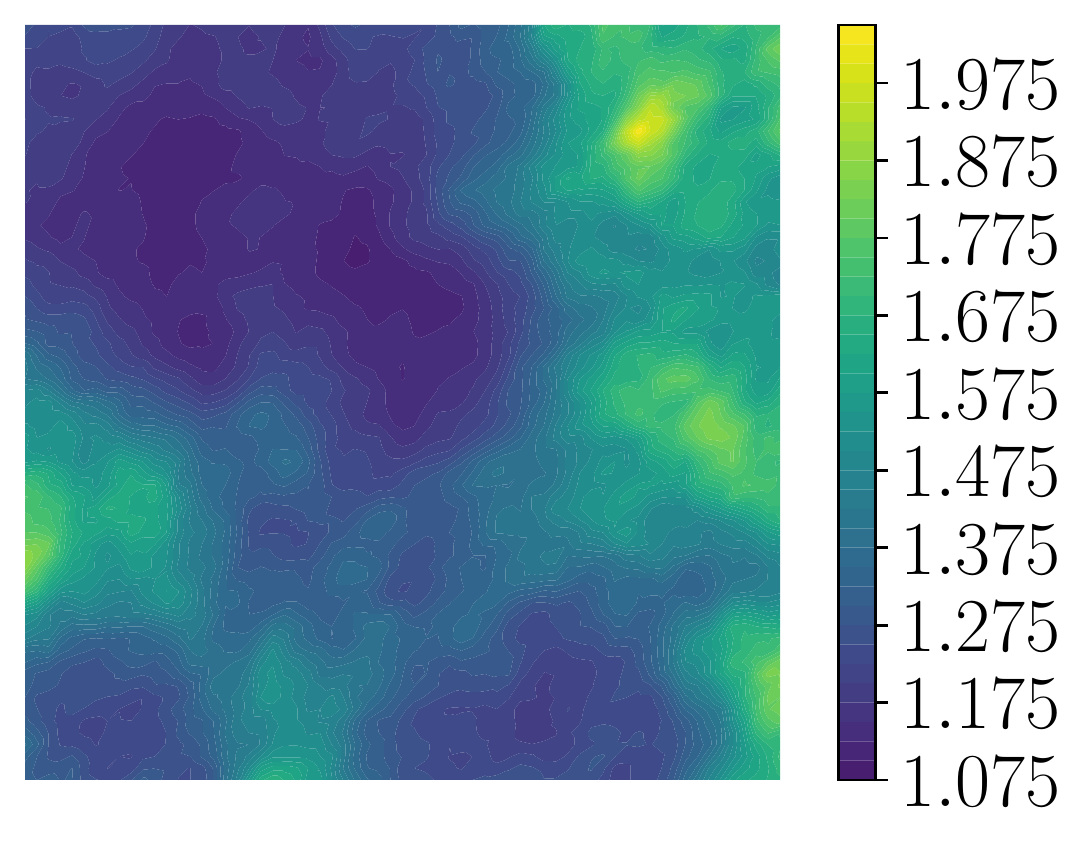} & \includegraphics[width = 0.19\linewidth]{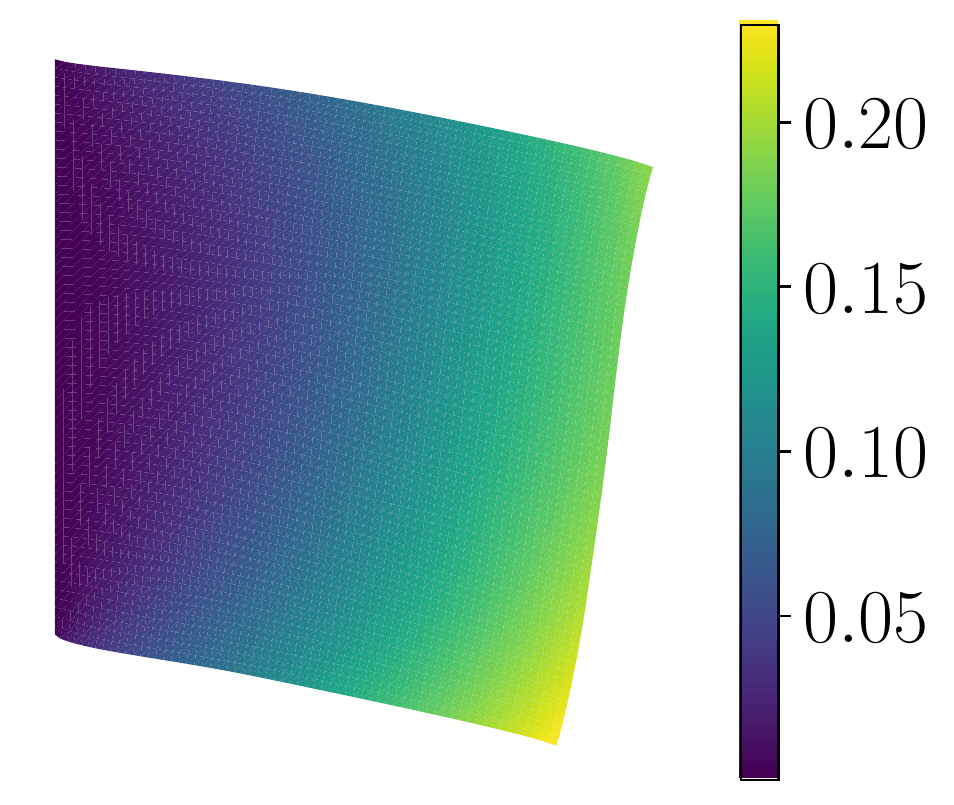} & \includegraphics[width = 0.19\linewidth]{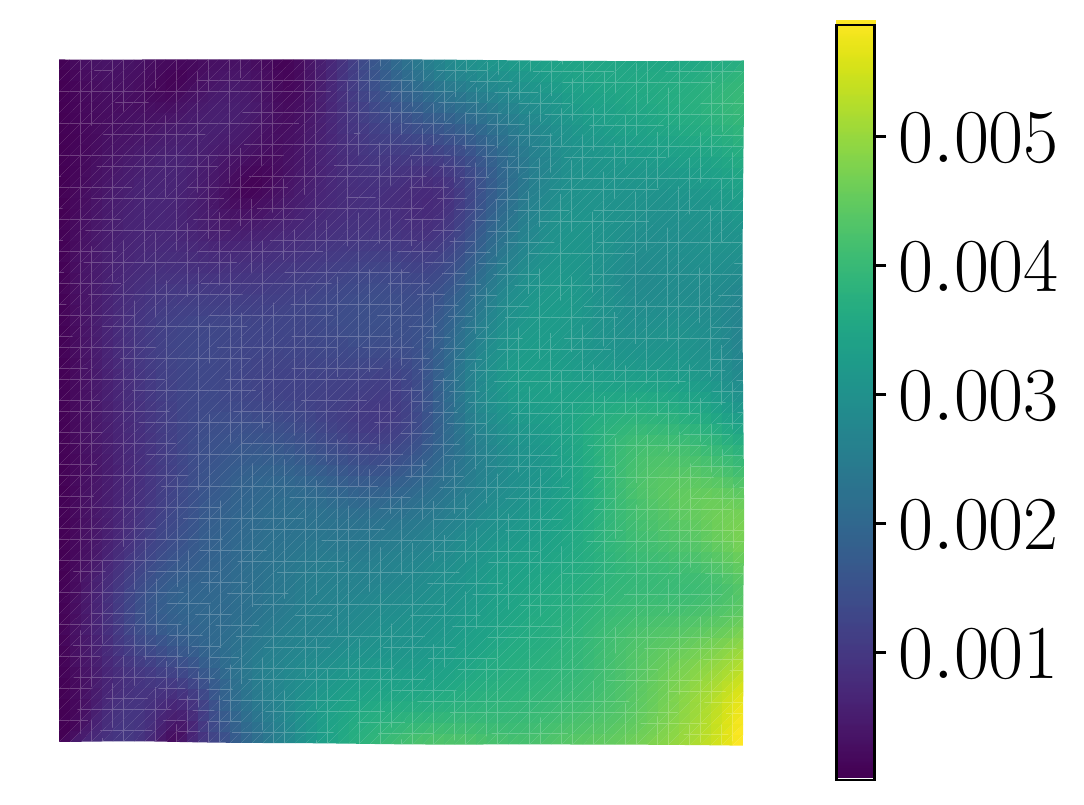} & \includegraphics[width = 0.19\linewidth]{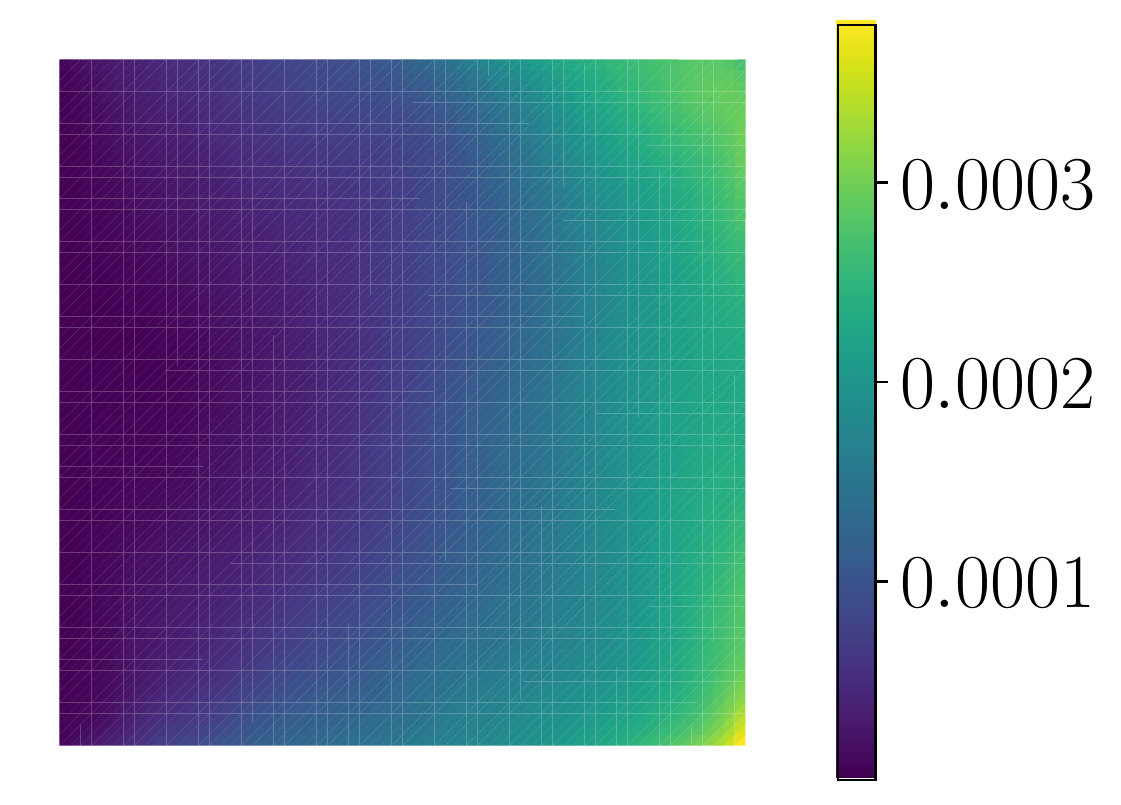}\\
        \includegraphics[width = 0.19\linewidth]{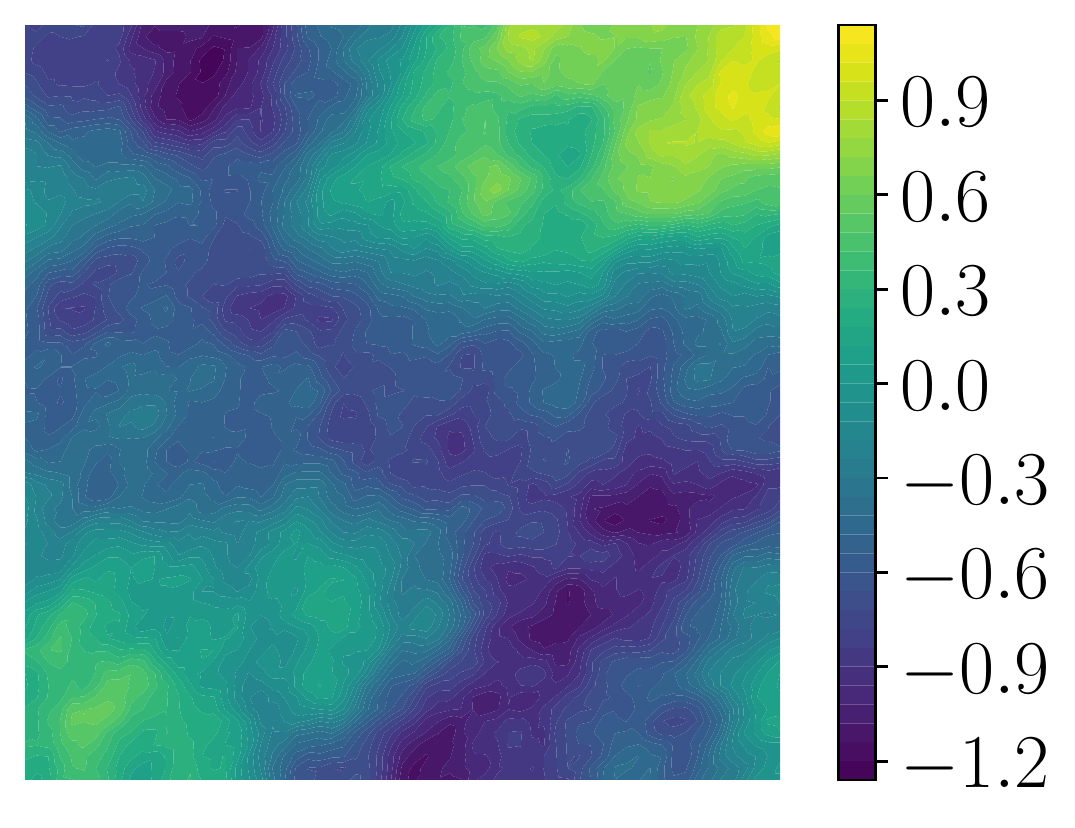} & \includegraphics[width = 0.19\linewidth]{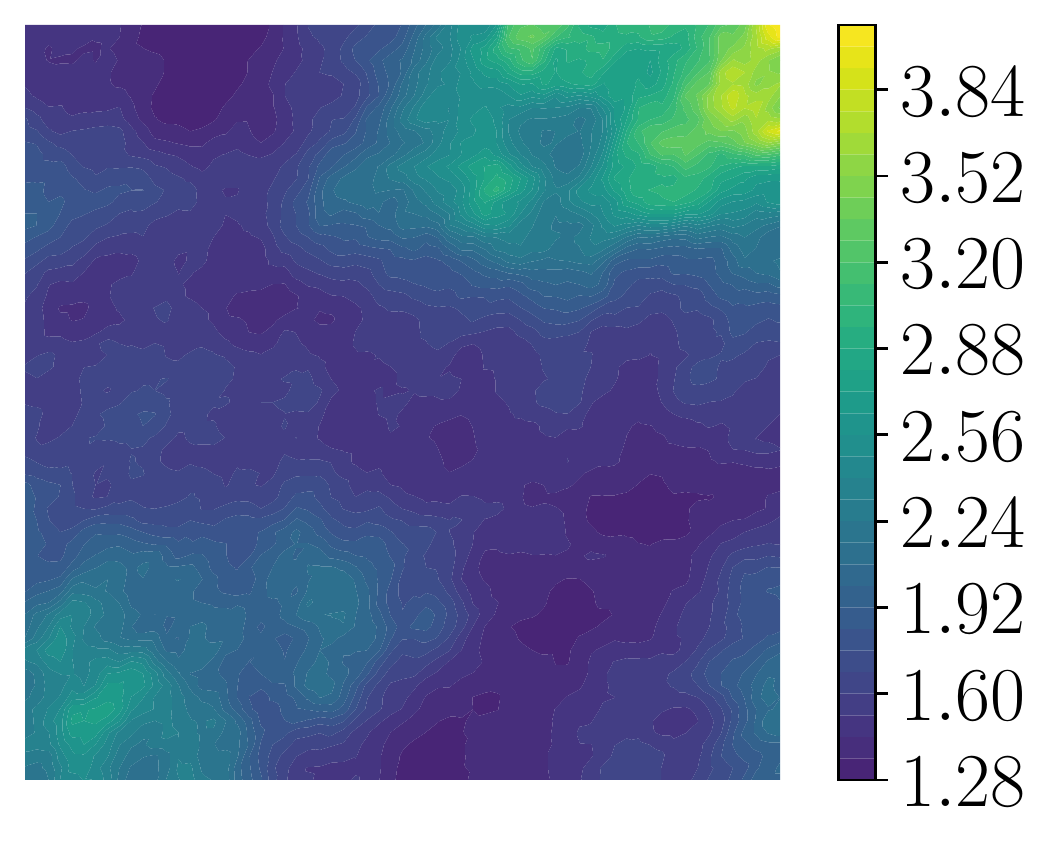} &\includegraphics[width = 0.19\linewidth]{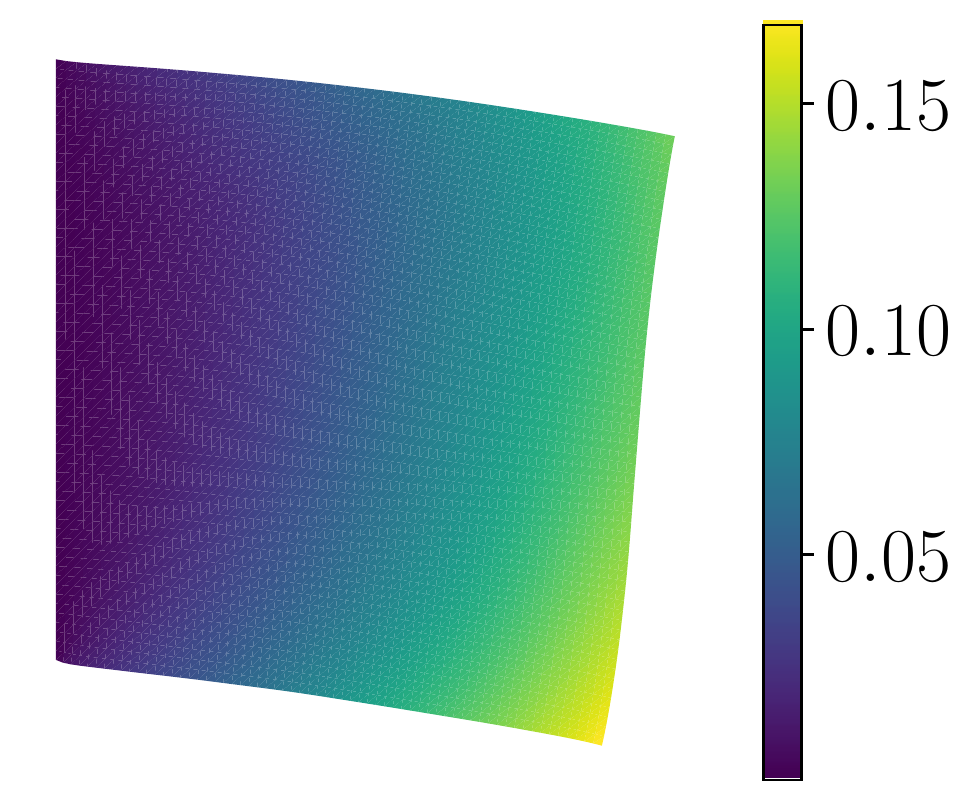} & \includegraphics[width = 0.19\linewidth]{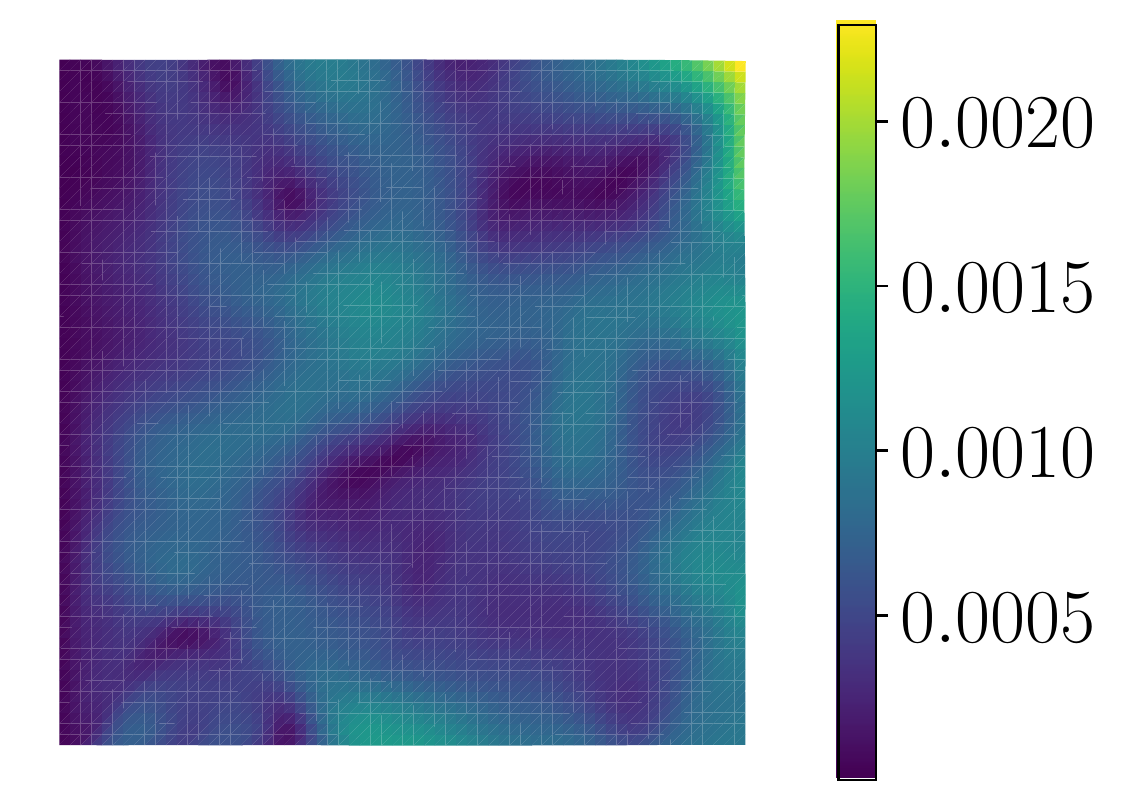} & \includegraphics[width = 0.19\linewidth]{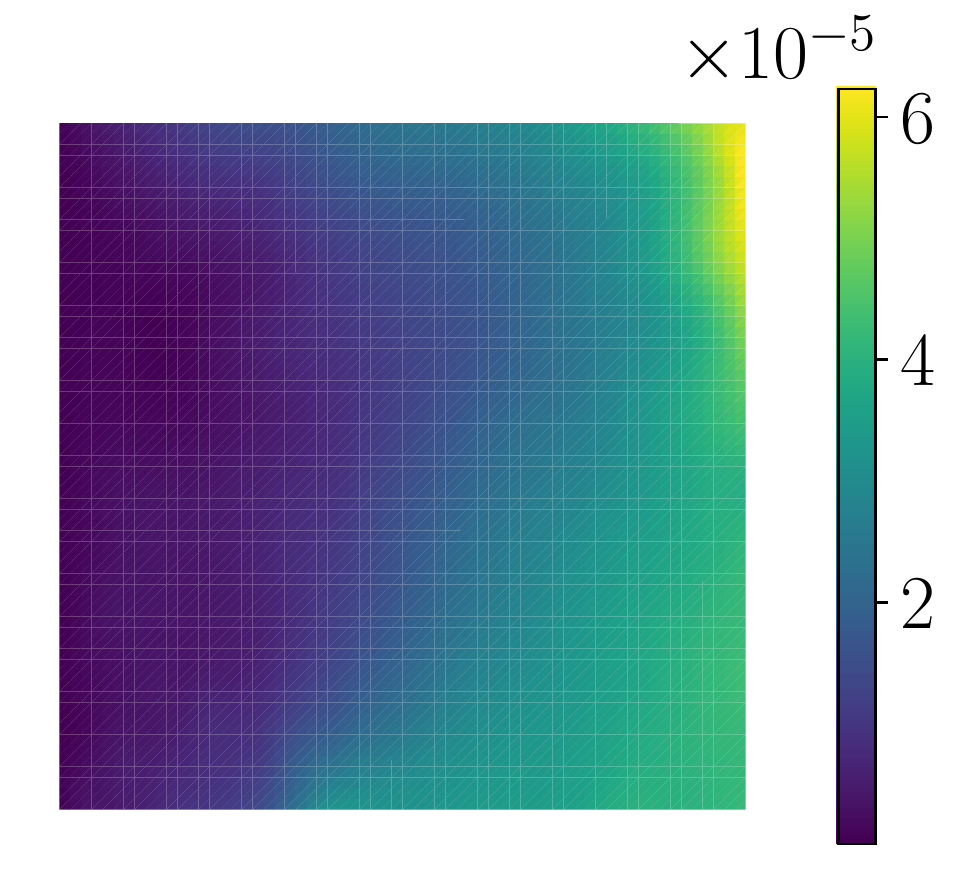}\\
        \includegraphics[width = 0.19\linewidth]{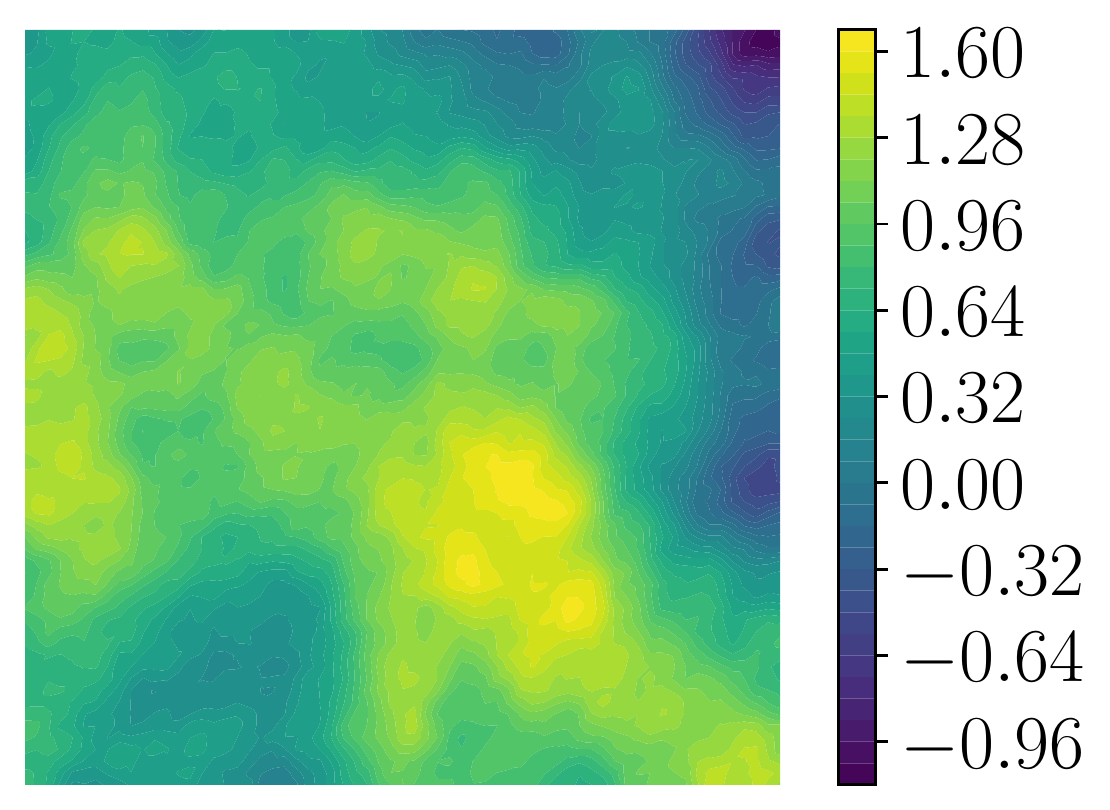} & \includegraphics[width = 0.19\linewidth]{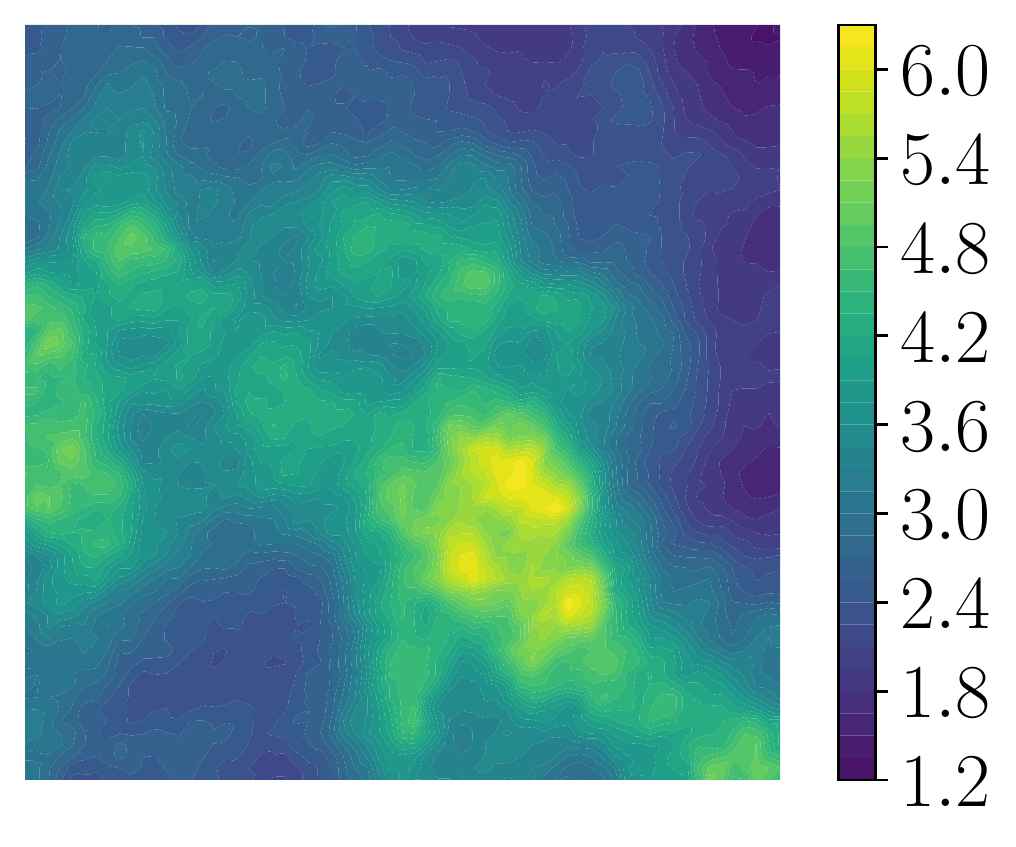} &\includegraphics[width = 0.19\linewidth]{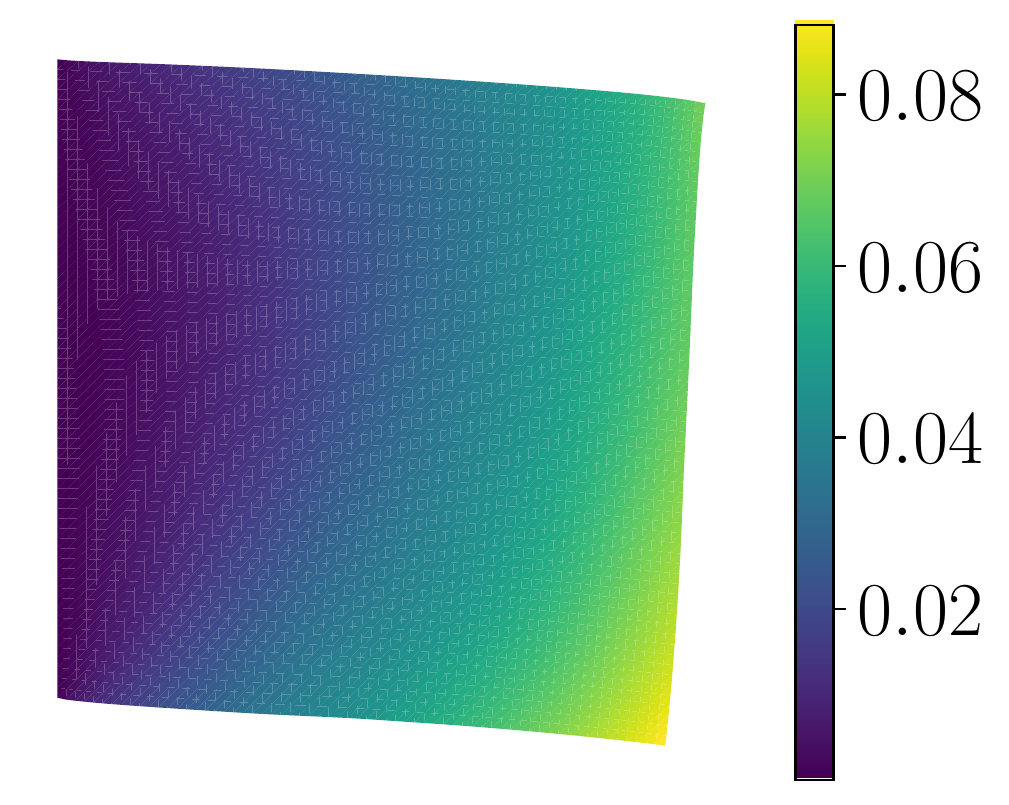} & \includegraphics[width = 0.19\linewidth]{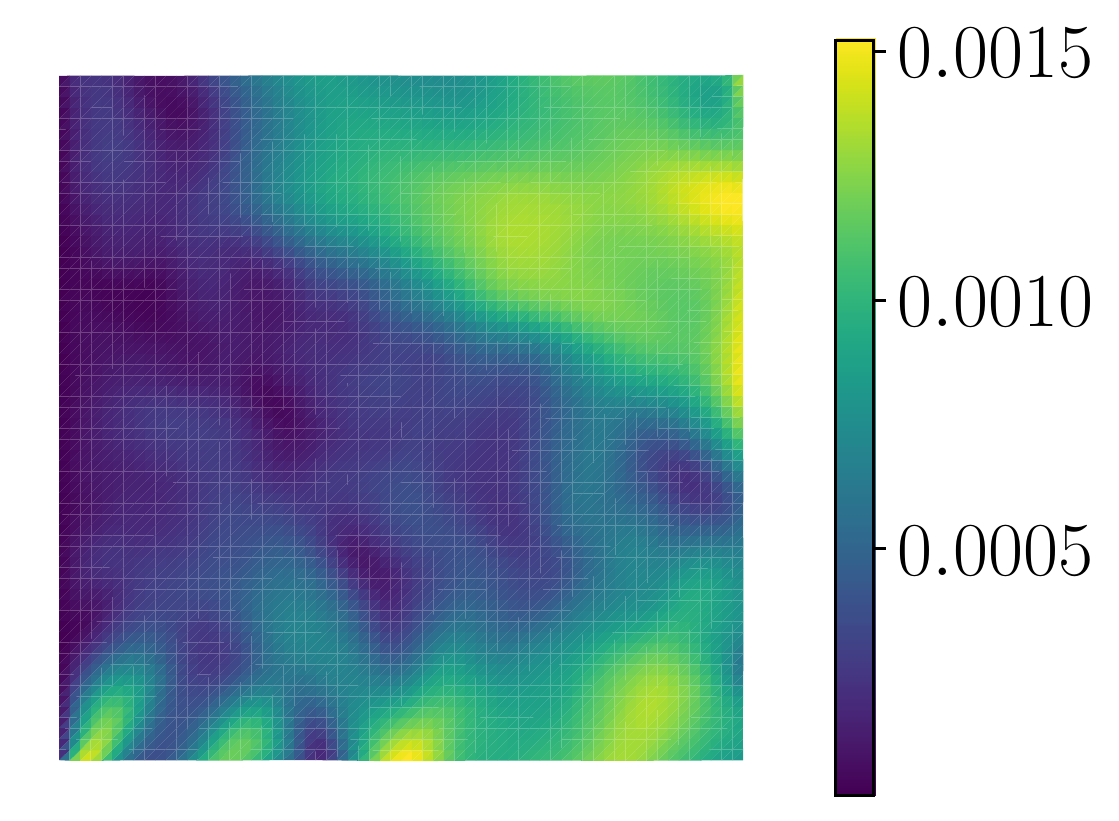} & \includegraphics[width = 0.19\linewidth]{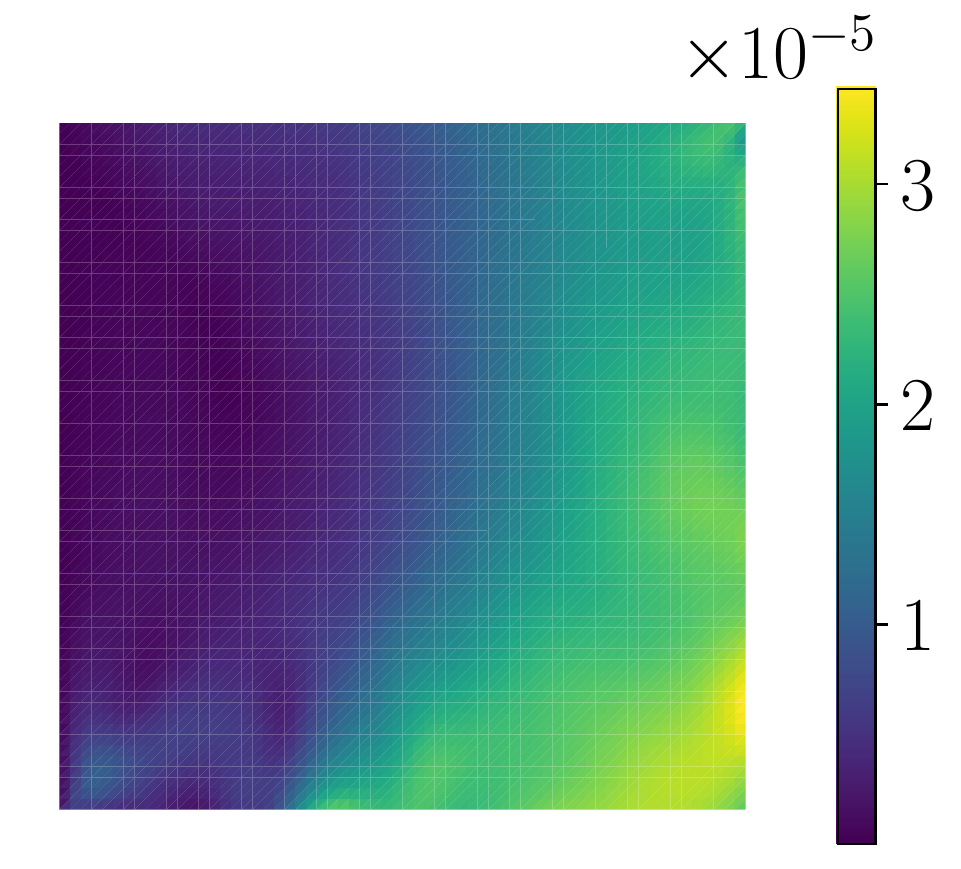}\\
    \end{tabular}
    \addtolength{\tabcolsep}{6pt} 
    \caption{Visualization of the prior distribution, model solutions, and neural operator performance with and without error correction for the hyperelastic material deformation problem introduced in Section~\ref{subsubsec:hyperelastic_model}. From left to right, we have (i) three Gaussian random fields, $m_j$, $j=1,2,3$, sampled from the prior distribution $\nu_M$ with approximately a pointwise variance of $1$ and a correlation length of $30\%$ domain side length, each occupies a row; (ii) the corresponding Young's modulus samples, $E_j$, defined by $E_j = \exp(m_j) + 1$, (viii) the corresponding finite element solutions of the current configurations, $\mathcal{F}^h(m_j)$, (iv) the absolute prediction errors at $m_j$ for the best performing neural operator $\operator_{\boldsymbol{w}}$ ($\sim95\%$ accurate as shown in in Figure~\ref{fig:hyperelasticity_nns}), $|\mathcal{F}^h(m_j) - \operator_{\boldsymbol{w}}(m_j)|$, and (v) the absolute errors at $m_j$ for the error-corrected predictions based on the best performing neural operator.}
    \label{fig:hyper_elasticity_prior}
\end{figure}

Assuming a quasi-static model in which external forces are applied very slowly so that dependence of displacement on time can be ignored and negligible body forces, the balance of linear momentum leads to the following nonlinear PDE for a given parameter $m:\Omega_0\to\R$.
\begin{subequations}
\begin{align}\label{eq:linmombal}
	\nabla \cdot (\bF(\bX)\bS(\bX, m(\bX), \bC(\bX))) &= \bzero\,, && \bX\in \Omega_0\,;\\
	\bu &= \bzero\,, \qquad && \bX\in \Gamma_{l}\,;\\
	\bF(\bX) \bS(\bX, m(\bX), \bC(\bX)) \cdot \bn  &= \bzero\,, && \bX\in \Gamma_{t}\cup\Gamma_{b};\\
	\bF(\bX) \bS(\bX, m(\bX), \bC(\bX)) \cdot \bn &= \bt(\bX)\,, && \bX\in \Gamma_{r}\,.
\end{align}
\end{subequations}
Here, $\Gamma_t$, $\Gamma_r$, $\Gamma_b$, and $\Gamma_l$ denote the top, right, bottom, and left boundary of the material domain, and $\bt$ is the traction given by 
\begin{equation}
\bt(\bX) = a \exp\left(-\frac{|X_2 - 0.5|^2}{b}\right) \be_1 + c \left(1 + \frac{X_2}{d}\right) \be_2
\end{equation}
with $a = 0.06$, $b = 4$,  $c = 0.03$, and $d = 10$. The applied traction on the right side is a combination of shear and tensile forces.

We consider the following parameter space, state space, and solution set for the nonlinear PDE problem above:
\begin{equation}
    \mathcal{M} \coloneqq L^2(\Omega_0)\,,\quad
    \mathcal{U} \coloneqq H^1(\Omega_0;\R^2)\,,\quad
    \mathcal{U}_0 = \mathcal{V}_u \coloneqq  \{\bu\in H^1(\Omega_0;\R^2):\bu|_{\Gamma_l} = \bzero\}\,,
\end{equation}
where the restriction to the boundary is defined with the trace operator. The variational problem for the experimental scenario is:
\begin{equation}\label{eq:hyperWeakForm}
\begin{aligned}
    &\text{Given } m\in \mathcal{M}, \text{ find }\bu \in \mathcal{V}_u \text{ such that }\\
	&\qquad\dualDot{\mathcal{R}(\bu, m)}{v}\coloneqq\int_{\Omega_0} \bF\bS(\bX, m(\bX) \bC(\bX)\bcolon \nabla \bv \,\dd \bX - \int_{\Gamma_{t}} \bt \cdot \bv\, \dd \bX = 0\,,\quad\forall \bv\in\mathcal{U}_0\,.
\end{aligned}
\end{equation}
In Figure~\ref{fig:hyper_elasticity_prior}, we show the model predictions via the finite element method, to be specified in the following subsection, of the current configurations at random samples from the prior.

\subsubsection{Numerical approximation and neural operator performance} 

The numerical evaluation of the forward operator, $\mathcal{F}^h$, sampling of the prior distribution, $\nu_M^h$, and the residual-based error correction problem is implemented via the finite element method. In particular, the domain $\Omega_0$ is discretized with $64\times64$ cells of uniform linear Lagrangian triangular elements, which forms finite element spaces $\mathcal{M}^h\subset\mathcal{M}$ and $\mathcal{U}^h\subset\mathcal{U}$ with $8450$ and $4225$ degrees of freedom, respectively. The variational problems of the model, prior sampling, and error correction are then approximated and solved in these finite element spaces. We employ the Newton iteration for solving the hyperelastic material deformation problem, which takes on average 8 iterations to converge for parameters samples generated from the prior distribution

\begin{figure}[H]
\center
\includegraphics[width = 0.6\textwidth]{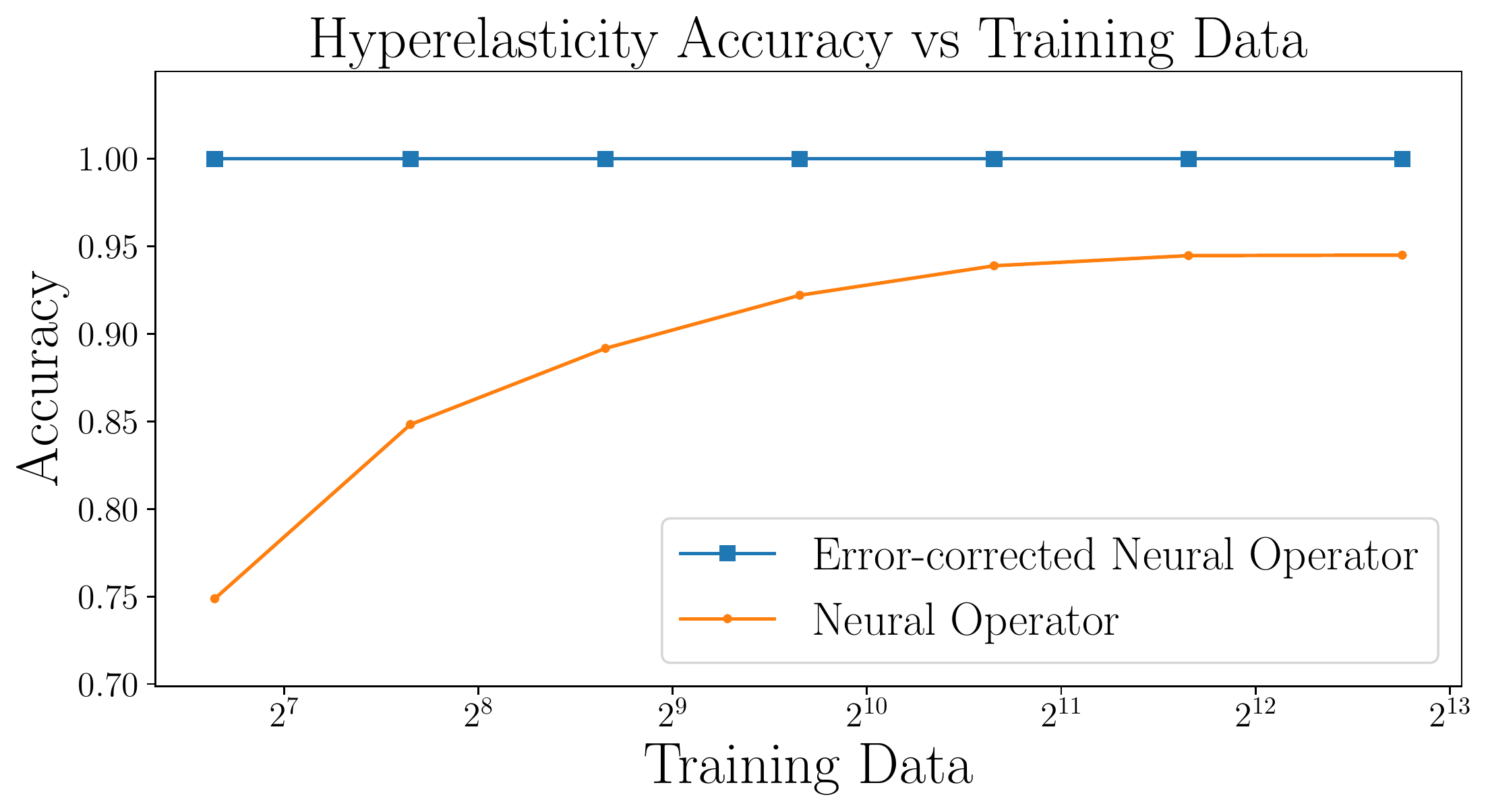}
\caption{A study of the generalization accuracy, as defined in~\eqref{eq:t_accuracy}, of $7$ neural operators trained using a varying number of training samples for the hyperelastic material deformation problem. The neural operators are constructed using a derivative-informed projected ResNet; see Section 5.1 for details on its construction and training. Generalization accuracy is computed using $512$ data unseen during training. An empirical accuracy ceiling of $\sim 95\%$ is reached for the given neural network architecture.
}
\label{fig:hyperelasticity_nns}
\end{figure}

Applying neural operator construction and training specified in Section~\ref{section:derivate_neural_operator} to the hyperelastic material deformation problem, we produced $7$ neural operators using a increasing number of training samples, $n_{\text{train}} = 100$, $201$, $403$, $806$, $1612$, $3225$, $6912$, assuming $p=2$ in the Bochner norm. In Figure~\ref{fig:hyperelasticity_nns}, the accuracy of the trained neural operators at different numbers of training samples is shown. The accuracy number is computed according to~\eqref{eq:t_accuracy} using $512$ samples from the prior distribution that are unseen during training. The accuracy ceiling around $95\%$ is reached at $n_{\text{train}} = 1612$, and a slight drop in accuracy is observed at $n_{\text{train}} = 6912$ samples.

For each trained neural operator, the accuracy for the error-corrected neural operator using the same 512 samples is also computed and shown in Figure~\ref{fig:hyperelasticity_nns}. The error-corrected neural operator mapping for all $7$ trained operators is close to $100\%$ accurate.

The visualization of absolution errors for the predictions by the best performing neural operator with $n_{\text{train}} = 3225$ and its error-corrected predictions at samples from the prior distribution is shown in Figure~\ref{fig:hyper_elasticity_prior}. We observe that the error correction step leads to a drop of maximum absolute pointwise error from the order of $10^{-3}$ to the order of $10^{-4}$. For prior samples that lead to small deformations, an additional order of magnitude drop in the maximum absolute pointwise error is observed.

\subsubsection{Bayesian inverse problem setting}\label{subsub:hyperelastic_inverse}
We consider a set of synthetic observation data $\boldsymbol{y}^*$ generated according to the data model in~\eqref{eq:data_model} for hyperelastic material deformation at a synthetic parameter field $m^*$ generated by a sum of three Gaussian bumps with different diagonal covariance and weights. This parameter corresponds to a synthetic Young's modulus field of $E^*$ that may represent a mixture of a small amount of spatially concentrated stiffer materials and a large amount of softer materials, e.g., cancer mass enclosed in healthy tissue. We visualize the synthetic parameter $m^*$, corresponding Young's modulus $E^*$, and displacement field $\boldsymbol{u}^* = \mathcal{F}^h(m^*)$ in Figure~\eqref{fig:hyperelastictity_setting}.

To complete the data model, we define an observation operator and a noise distribution. We, again, consider a linear observation operator that extracts discrete observations of the displacement vector at a uniform grid of $10\times 10$ points in the reference domain $\Omega_0$. This form of the observation operator is compatible with image analysis techniques such as digital image correlation; see, e.g., \cite{chevalier2001digital, moerman2009digital, MCCORMICK201052, jurjo2015analysis, li2017algorithm, ribeiro2017hybrid, flaschel2021unsupervised}. Let $\{\bX_j\}_{j=1}^{100}$ denote the observation points. Given a displacement field $\bu\in\mathcal{U}$, the observation operator $\bdmc{B}(\bu) \in \R^{2\times 100}$ returns local averages of the displacements around the observation points:
\begin{equation}
    \bdmc{B}(\bu) = \begin{bmatrix}
    \displaystyle |B_r(\bX_1)|^{-1}\int_{B_r(\bX_1)} \bu(\bX)\,\dd\bX &  \cdots & \displaystyle|B_r(\bX_{100})|^{-1}\int_{B_r(\bX_{100})} \bu(\bX)\,\dd\bX
    \end{bmatrix} \,.
\end{equation}
We assume the observation is corrupted with white noise, i.e., $\boldsymbol{N} \sim \mathcal{N}(\boldsymbol{0}, \sigma^2\boldsymbol{I})$, with a standard deviation of $\sigma = 0.0089$ that is $1\%$ of the maximum values in $\bdmc{B}(\bu^*)$.

\begin{figure}[H]
    \centering
    \addtolength{\tabcolsep}{-6pt} 
    \begin{tabular}{cccc}
        $m^*$ & $E^*$ & $\mathcal{F}^h(m^*)$ & $|\boldsymbol{y}^*|\in\R^{100}$ \\
        \includegraphics[width = 0.24\linewidth]{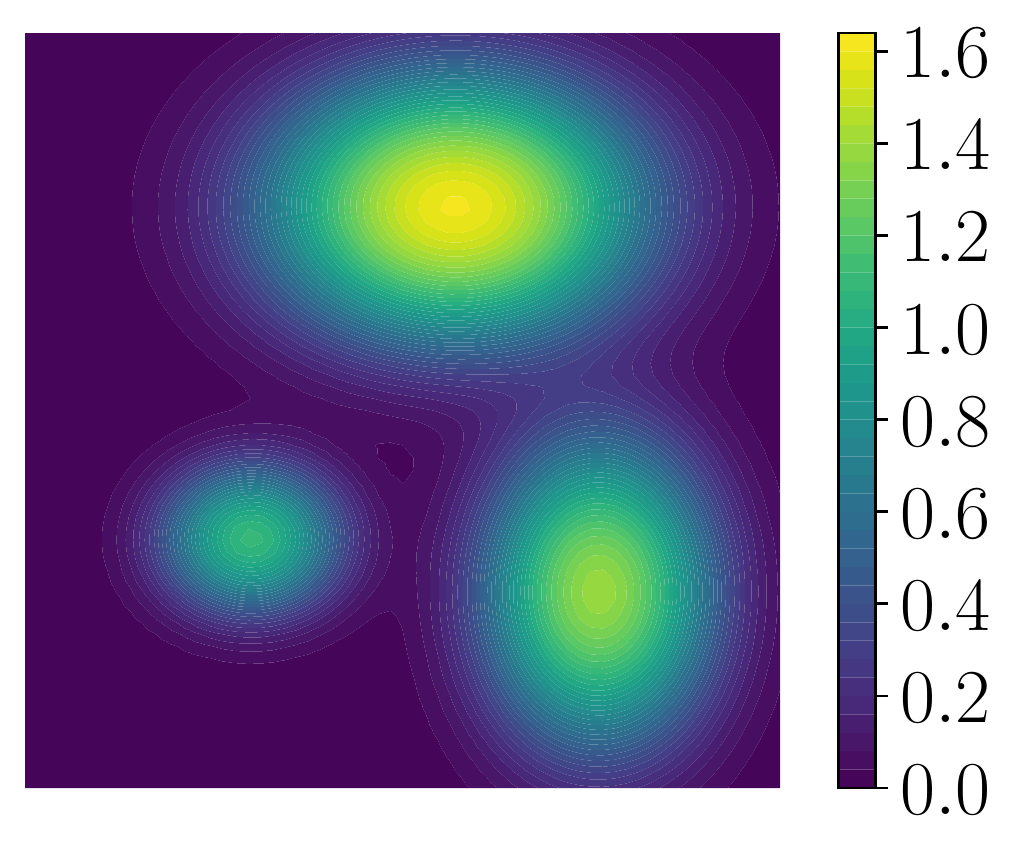} & \includegraphics[width = 0.24\linewidth]{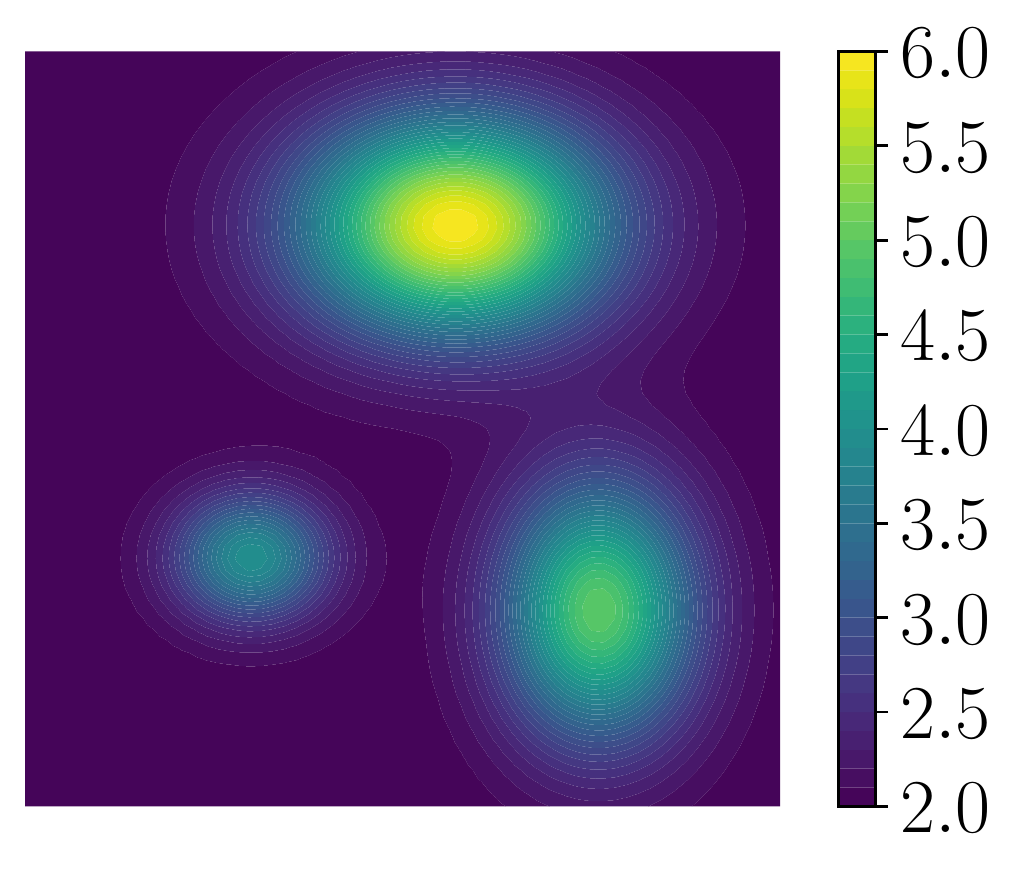} & \includegraphics[width = 0.24\linewidth]{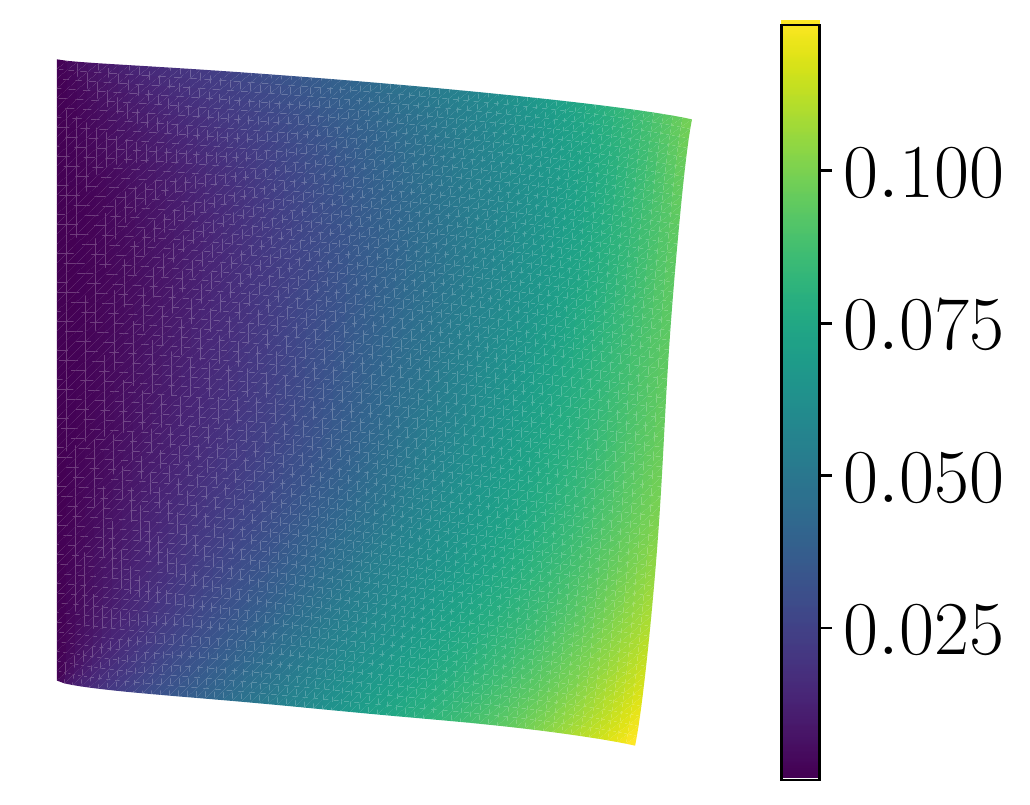} & \includegraphics[width = 0.24\linewidth]{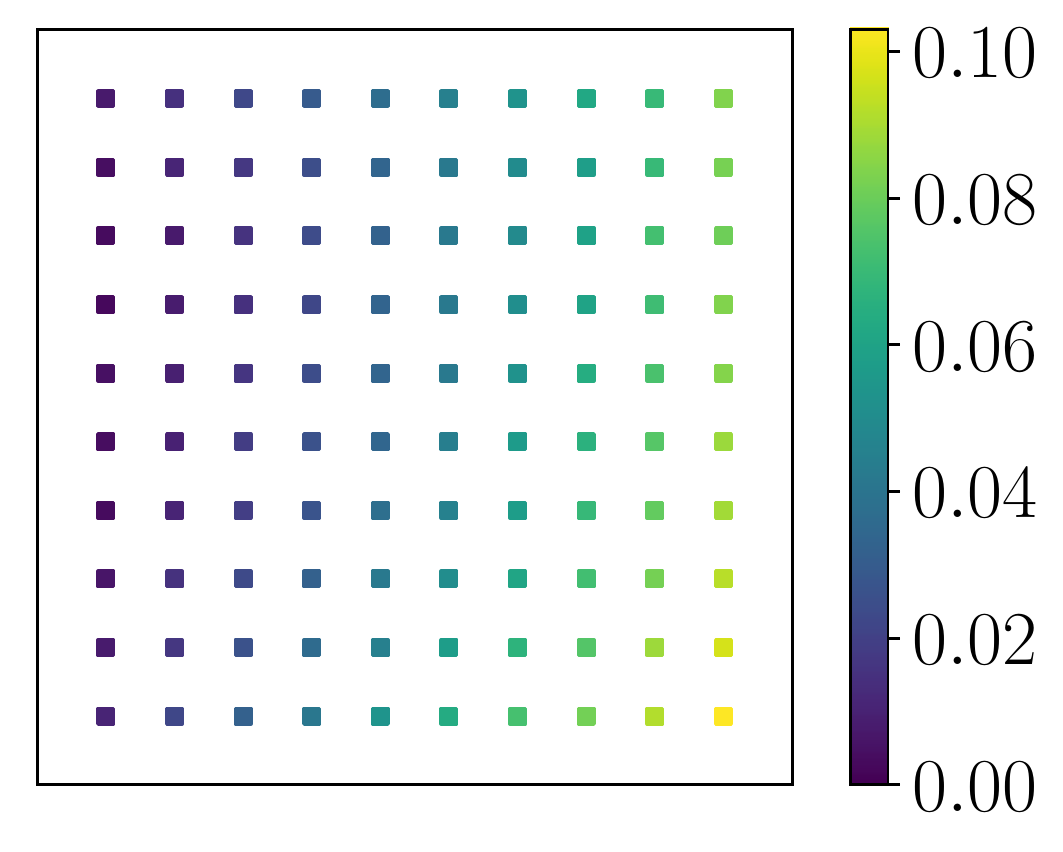}\\    
    \end{tabular}
    \addtolength{\tabcolsep}{6pt} 
    \caption{Visualization of the setting for a synthetic Bayesian inverse problem based on the hyperelastic material deformation introduced in Section~\ref{subsub:hyperelastic_inverse}. From left to right, we have (i) the synthetic parameter field, $m^*$, defined as the sum of three Gaussian bumps with different weights and diagonal covariances, (ii) the corresponding synthetic Young's modulus field, $E^*$, (iii) the finite element solution at $m^*$, $\mathcal{F}^h(m^*)$, (iv) the synthetic observed data, $\by^*$, extracted from locally averaged values of $\mathcal{F}^h(m^*)$ at a $10\times10$ grid of observation points, corrupted by a randomly sampled additive white noise with a standard deviation of $0.0089$.}
    \label{fig:hyperelastictity_setting}
\end{figure}

\subsubsection{Posterior visualization and cost analysis}
To visualize and compare posterior distributions defined with likelihood functions evaluated using the model via the finite element method, neural operators, and error-corrected neural operators, we generate samples from these posterior distributions via the pCN algorithms introduced in Section 2.3, and visualize the posterior predictive mean of Young's modulus by sample average approximation,
\begin{equation}\label{eq:hyper_pred_mean}
    E_{\text{mean}} \approx \frac{1}{n_{\text{post}}}\sum_{j=1}^{n_{\text{post}}} (\exp(m_j) + 1)\,,
\end{equation}
where $\{m_j\in\mathcal{M}^h\}_{j=1}^{n_{\text{post}}}$ are the posterior samples. For each posterior distributions, $8$ MCMC chains are constructed with a mixing parameter of $\beta_{\text{pCN}} = 0.05$ and collect, in total, posterior samples of $n_{\text{post}} = 120,000$. While mixing of the chains is rapid, a conservative burn-in rate of $25\%$ is used. The average sample acceptance ratio for MCMC sampling using the model is around $10\%$.

In Figure~\ref{fig:hyper_no_mean}, we visualize the model generated posterior predictive mean estimate of Young's modulus alongside the ones generated by the three best-performing neural operators. We observe that the estimate by the neural operator with $n_{\text{train}} = 1612$ completely misplaced two of the bumps clearly visible in the model estimate. A doubling of training data and a small increase of accuracy from $n_{\text{train}} = 1612$ to $n_{\text{train}} = 3225$ leads to a visually significant increase in the accuracy of the neural operator estimate. The neural operator with $n_{\text{train}} = 3225$ is able to qualitatively recover most features in the model estimate, with some difficulty in capturing the shape of the top bump and the location of the lower left bump. However, a doubling of training data from $n_{\text{train}} = 3225$ to $n_{\text{train}} = 6912$ in an attempt to further improve this result leads to a slight drop in the accuracy of the neural operator and a significant drop in accuracy of the estimate. This observation reflects the unreliability of approximation error reduction of neural operators and thus the unreliability of the surrogate modeling approach for inverse problems based solely on trained neural operators.

In Figure~\ref{fig:hyper_correction_mean}, we provide the same visualization for the estimates produced by the same neural operators but with error correction. We observe that they generate estimates that are qualitatively much similar to the model generated estimate, where all three bumps are captured with positions and magnitudes matching that of the model. These estimates are relatively consistent across the three neural operators, even though the three neural operators show drastically different accuracy in their estimates, demonstrating the robustness of our approach for error reduction of neural operators in Bayesian inverse problems.

\begin{figure}[H]
    \centering
    \addtolength{\tabcolsep}{-6pt} 
    \begin{tabular}{cccc}
    \textbf{Model $\mathcal{F}^h$} & \multicolumn{3}{c}{\textbf{Neural operator $\operator_{\boldsymbol{w}}$}}\\
     &   $n_{\text{train}} = 1612$ & $n_{\text{train}} = 3225$ & $n_{\text{train}} = 6912$  \\
    \includegraphics[width = 0.24\linewidth]{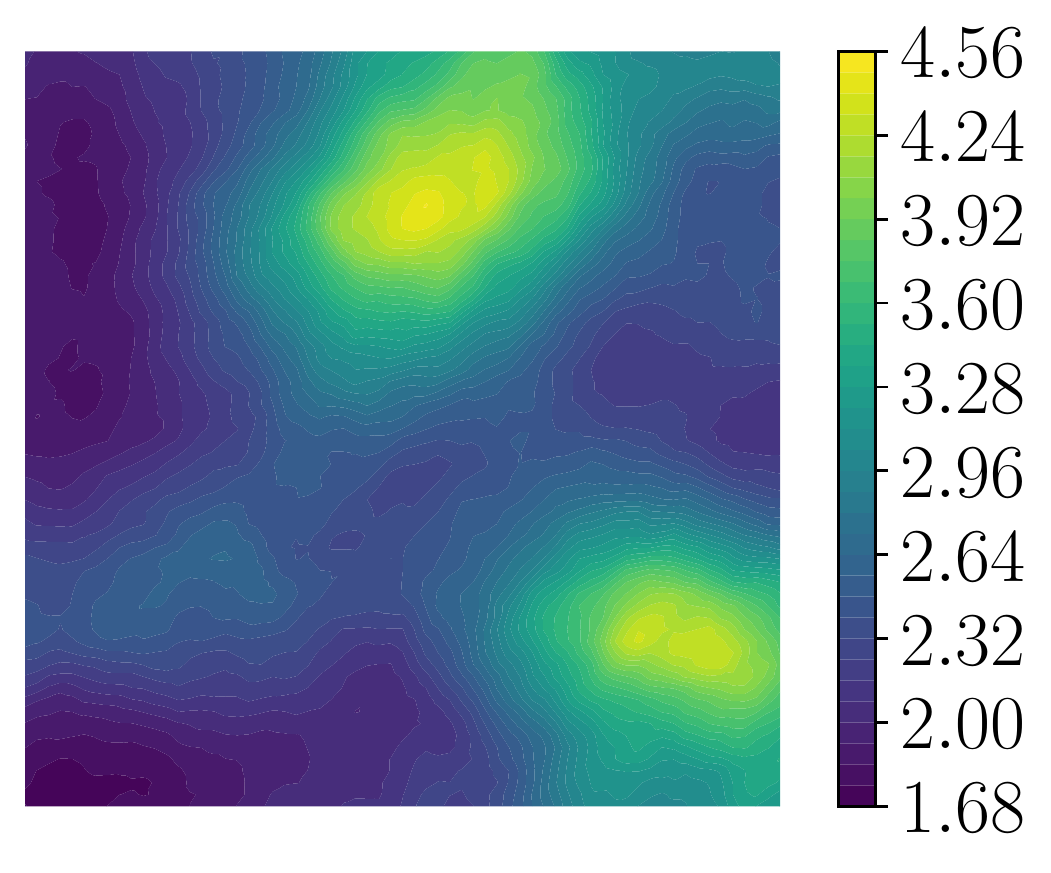} & \includegraphics[width = 0.23\linewidth]{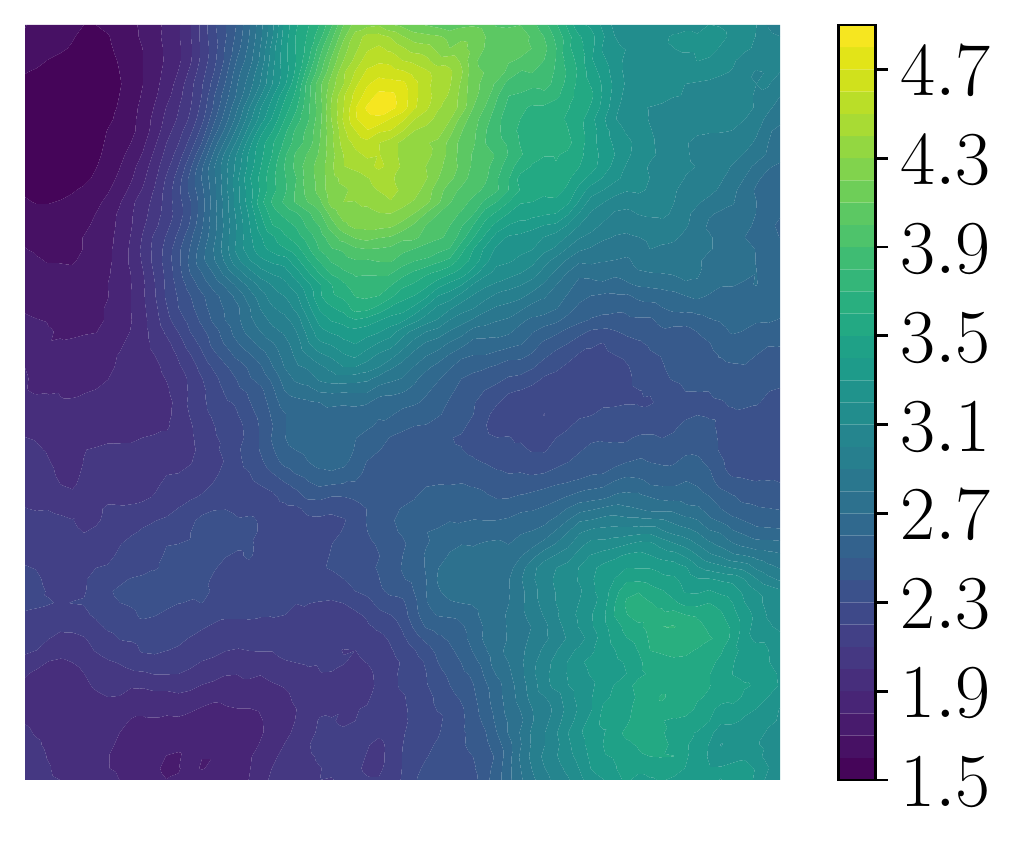} & \includegraphics[width = 0.24\linewidth]{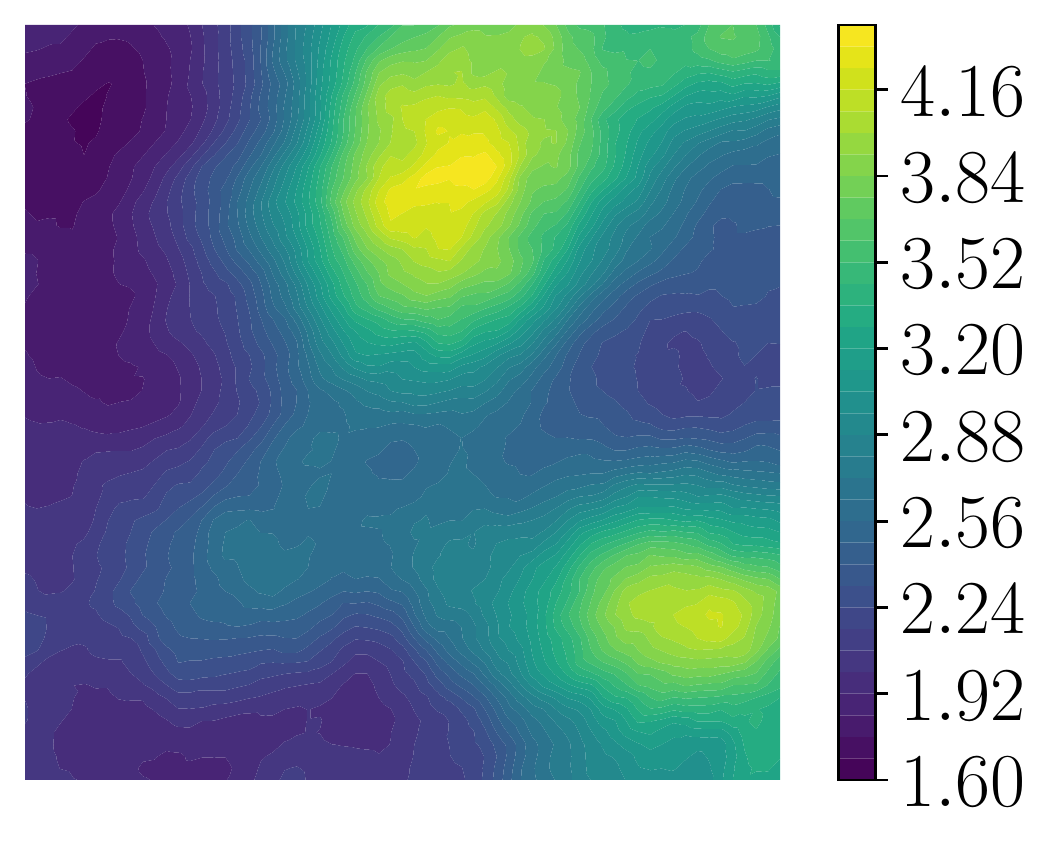} &  \includegraphics[width = 0.24\linewidth]{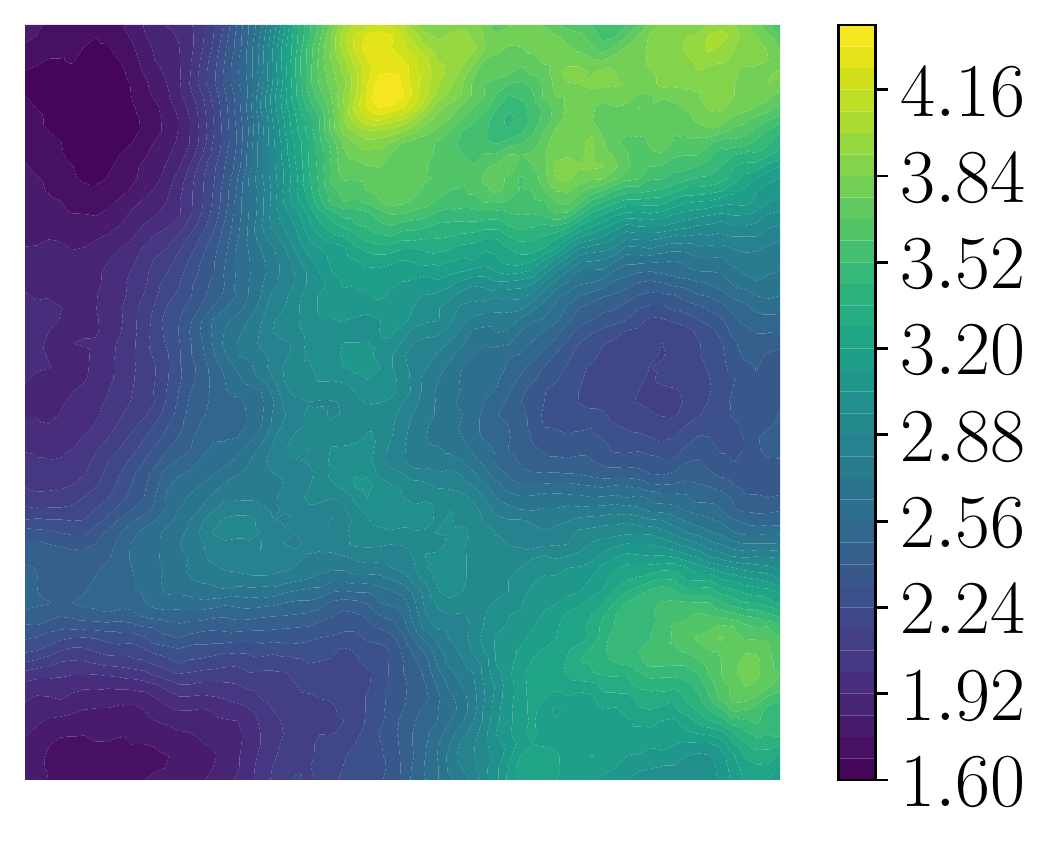}\\
    \end{tabular}
    \addtolength{\tabcolsep}{6pt} 
    \caption{Visualization of the posterior predictive estimates of Young's modulus field in~\eqref{eq:hyper_pred_mean} for a synthetic Bayesian inverse problem introduced in Section~\ref{subsub:hyperelastic_inverse}. From left to right, we have the estimates by (i) the model via the finite element method, (2) the neural operator trained with $n_{\text{train}} = 1612$, (3) the neural operator trained with $n_{\text{train}} = 3225$, (4) and the neural operator trained with $n_{\text{train}} = 6912$.}
    \label{fig:hyper_no_mean}
\end{figure}

\begin{figure}[H]
    \centering
    \addtolength{\tabcolsep}{-6pt} 
    \begin{tabular}{cccc}
    \textbf{Model $\mathcal{F}^h$} & \multicolumn{3}{c}{\textbf{Neural operator with error correction $\operator_{C}$}}\\
     &   $n_{\text{train}} = 1612$ & $n_{\text{train}} = 3225$ & $n_{\text{train}} = 6912$  \\
    \includegraphics[width = 0.24\linewidth]{model_mean_modulus.pdf} & \includegraphics[width = 0.24\linewidth]{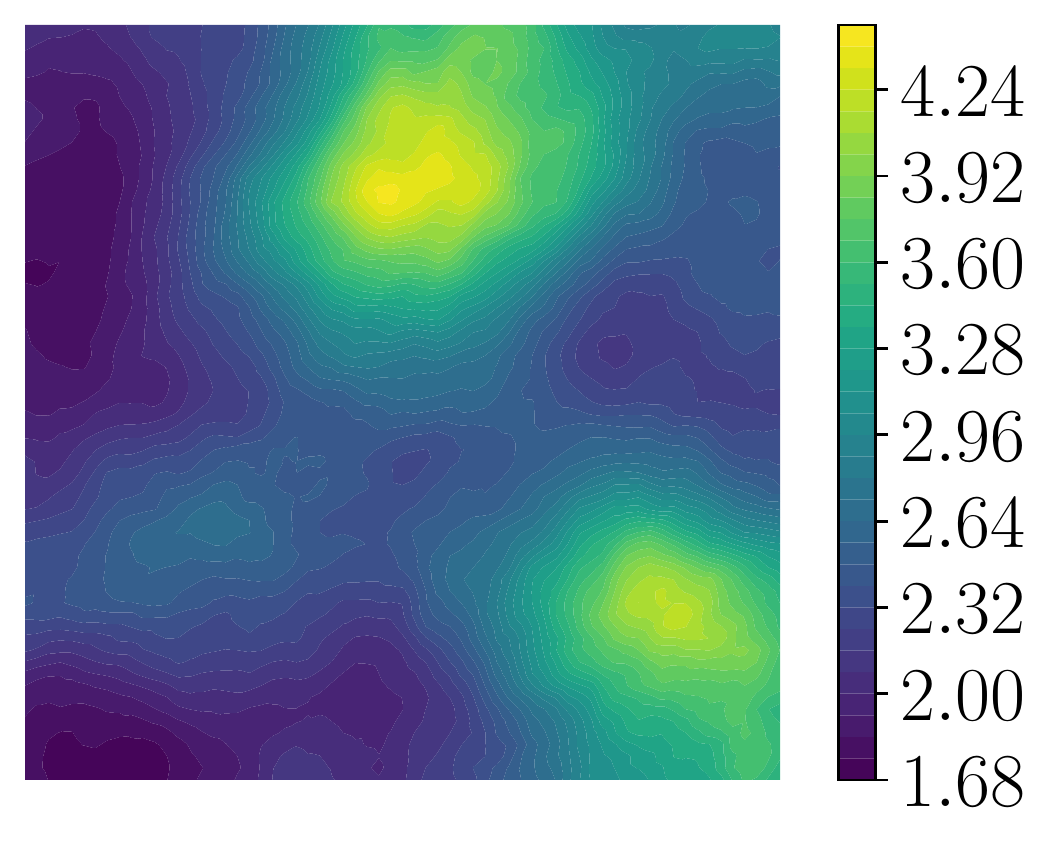} & \includegraphics[width = 0.24\linewidth]{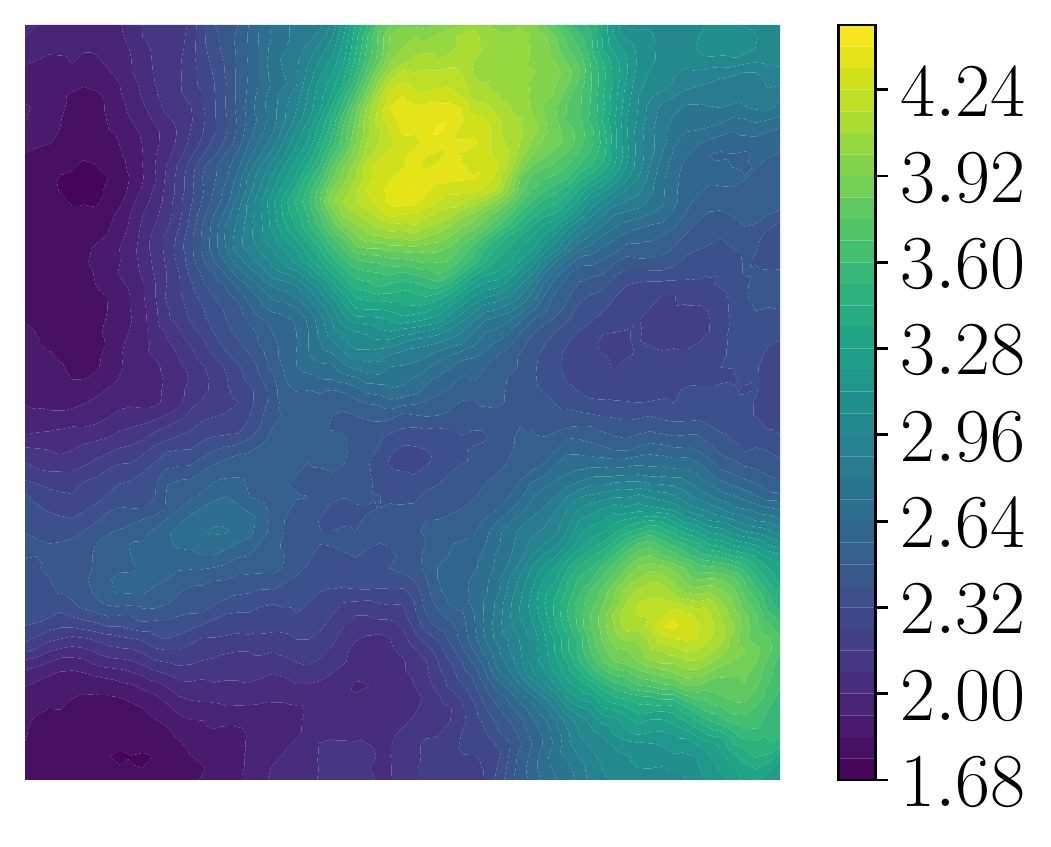} & \includegraphics[width = 0.24\linewidth]{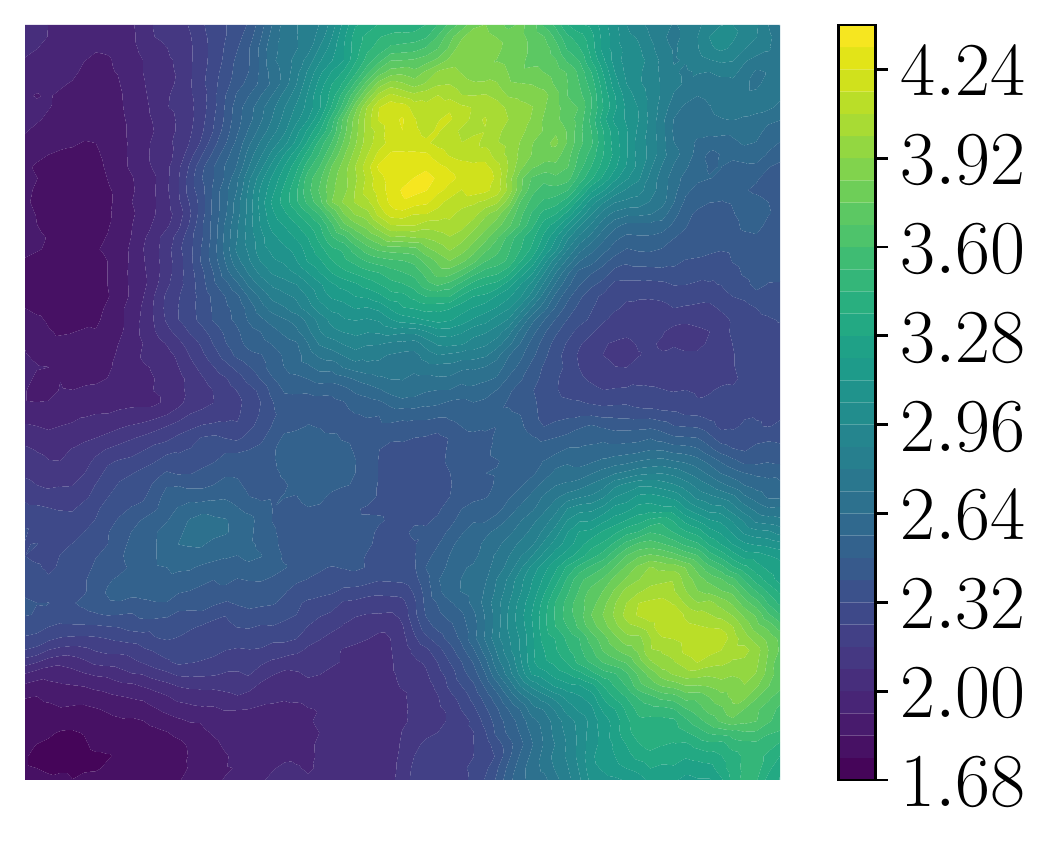} \\
    \end{tabular}
    \addtolength{\tabcolsep}{6pt} 
    \caption{Visualization of posterior the predictive mean estimates of Young's modulus field in~\eqref{eq:hyper_pred_mean} for a synthetic Bayesian inverse problem introduced in Section~\ref{subsub:hyperelastic_inverse}. From left to right, we have the estimates by (i) the model via the finite element method, (2) the neural operator trained with $n_{\text{train}} = 1612$ with error correction, (3) the neural operator trained with $n_{\text{train}} = 3225$ with error correction, (4) and the neural operator trained with $n_{\text{train}} = 6912$ with error correction.}
    \label{fig:hyper_correction_mean}
\end{figure}

In Figure~\ref{fig:hyperelasticity_speedup}, we visualize the observed and asymptotic speedups for the posterior sampling using the $7$ trained neural operators with or without the error correction. Similarly to the computational cost analysis for the reaction--diffusion problem, the asymptotic speedups here assume $n_{\text{chain}}\to\infty$, and the offline cost of the neural operator construction and training are neglected. The asymptotic speedup of the error-corrected neural operators is about one order of magnitude, which is the number of Newton iterations for solving the nonlinear problem, averaged over the posterior distribution. The asymptotic speedup of the neural operators is over two orders of magnitude. The observed speedups additionally account for the offline cost of reduced basis approximation, training data generation, and optimization, and included a finite $n_{\text{chain}}$ used for the posterior sampling presented above. The resulting observed speedups for the neural operators decay substantially as the number of training data increases, and nearly drop an order of magnitude from $n_\text{train} = 100$ to $n_{\text{train}} = 6912$. The observed speedup at the latter is in the same order as the asymptotic speedup of the error-correction neural operators. A similar decay of the observed speedups is seen for the error-corrected neural operators, but the dominant cost remains the cost of solving the linear systems associated with the error-correction steps.

\begin{figure}[H]
\center
\includegraphics[width = 0.9\textwidth]{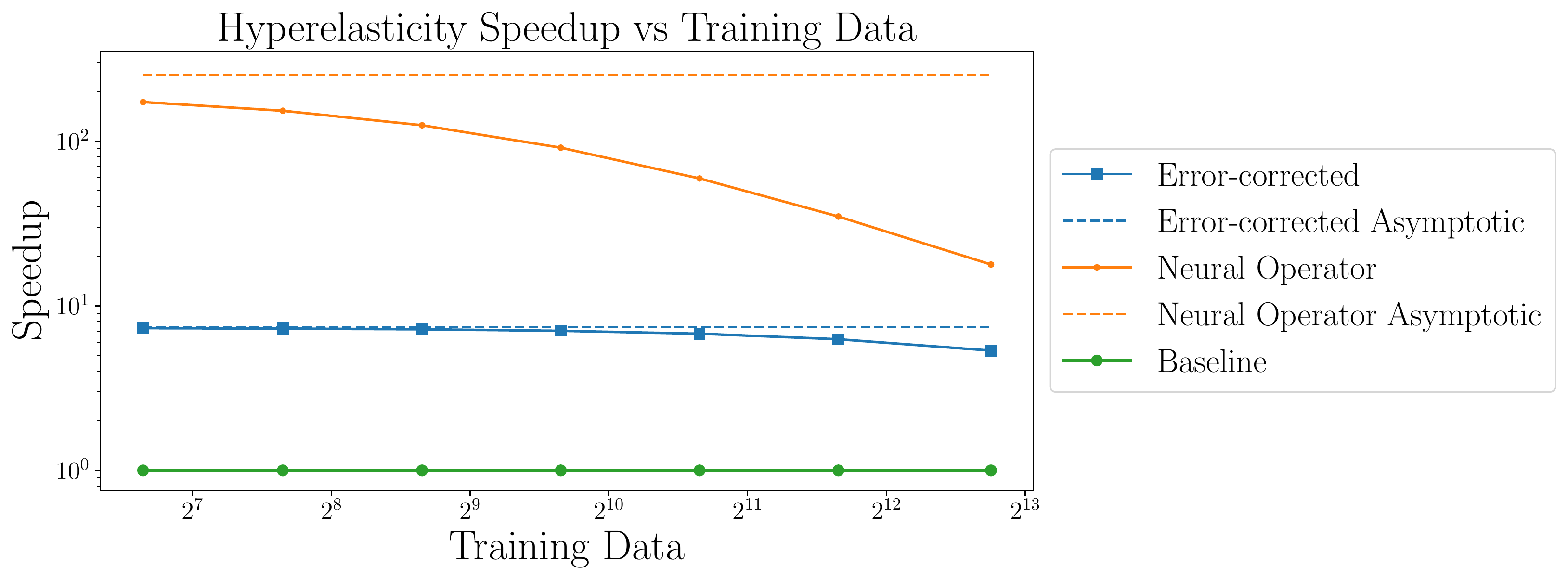}

\caption{The observed and asymptotic speedups, as defined in Section 4.4, for the posterior sampling via pCN for both the neural operators and the error-corrected neural operators. The asymptotic speedups assume $n_{\text{chain}}\to\infty$, thus neglecting the offline cost of neural operator construction and training. The observed speedups additionally account for the offline costs and the total number of iterative solves within generated Markov chains used for the posterior visualization in Figure~\ref{fig:hyper_no_mean} and~\ref{fig:hyper_correction_mean}.}
\label{fig:hyperelasticity_speedup}
\end{figure}

\section{Conclusion and Outlook}

In this work, we presented a residual-based error correction strategy for the reliable deployment of trained neural operators as a surrogate of the parameter-to-state map defined by nonlinear PDEs for accelerating infinite-dimensional Bayesian inverse problems. The strategy is motivated by the \textit{a posterior} error estimation techniques applied to estimating modeling error. For a trained neural operator, we utilized its prediction at a given parameter to formulate and solve a linear variational problem, or an error-correction problem, based on the PDE residual and its derivative. The resulting solution can lead to a quadratic reduction of local error due to the equivalency of solving the error-correction problem and generation one Newton step under some mild conditions. We show that this can be extended to a global reduction, i.e., over the prior distribution, of approximation errors for well-trained neural operators.

The proposed strategy addresses an important issue facing the application of operator learning using neural networks, which is the unreliability of neural operator performance improvement or approximation error reduction via training. Through deriving an \textit{a priori} error bound, we demonstrate that the approximation error of trained neural operators controls the error in the posterior distributions of Bayesian inverse problems when they are used as surrogate parameter-to-state maps in likelihood functions. Furthermore, the resulting error in the posterior distributions may be significantly magnified for Bayesian inverse problems that have uninformative prior distributions, high-dimensional data, small noise corruption, and inadequate models. In situations where the approximation error of a neural operator is persistent and not easily reduced to an acceptable level for the target Bayesian inverse problem via training, our proposed strategy offers an effective alternative of constructing error-corrected neural operators for achieving these accuracy requirements, all while retaining substantial computational speedups by leveraging the predictability of trained neural operators. For models governed by large-scale highly nonlinear PDEs, where the costs of evaluating trained neural operators are relatively negligible, this strategy provides a great computational speedup for posterior characterization, which is approximately the expected number of iterative linear solves within a nonlinear PDE solve at parameters sampled from the posterior distribution,

We demonstrate the advantages of our proposed strategy through two numerical examples: inference of an uncertain coefficient field in an equilibrium nonlinear reaction--diffusion equation and hyperelastic material properties discovery. For both problems, the performance of trained neural operators shows diminishing improvement from $<80\%$ accuracy with a small number of training samples to their empirical accuracy ceilings of $<95\%$ with a much larger number of training samples, while the performance of error-corrected neural operators is consistent with near $100\%$ accuracy. The visualization of the posterior predictive mean estimates suggests that trained neural operators alone as surrogates of the parameter-to-state map cannot reliably recover distinctive features of the physical parameters that are shown to be retrievable via inference with full-scale model solves, while the error-corrected neural operators as surrogates show consistency in its ability to recover these features.

Many other important outer-loop problems in engineering, sciences, and medicine, such as infinite-dimensional optimal control and design problems, can also benefit from fast and accurate full-state predictions using error-corrected neural operators. Moreover, many of these outer-loop problems, including Bayesian inverse problems, are based on physical systems modeled by nonlinear time-evolving PDEs. Although the combination of residual-based error correction and neural operators is generally applicable to these outer-loop problems and models, thorough theoretical and numerical investigations are needed to understand its range of applications and limitations in these settings.

From our theoretical and numerical analysis of the computational speedups of error-corrected neural operators, it is clear that the dominant computational cost of the proposed strategy is the accumulated cost of solving the full linear error correction problem at each neural operator prediction. This burden can also be further mitigated through either incorporating approximation techniques such as model order reduction or developing problem-specific fast solvers for the error correction problem. Other possible applications of the error-corrected neural operators are in multilevel or multifidelity methods for outer-loop problems.

\section*{Acknowledgement}
The work was supported by the U.S. Department of Energy, Office of Science, Office of Advanced Scientific Computing Research, Mathematical Multifaceted Integrated Capability Centers (MMICCs) program under Award DE-SC0019303. The authors would like to thank Peng Chen, Jiaqi Li, and Barbara Wohlmuth for technical suggestions that helped improve this work. The authors would like to thank Michael B. Giles for detailed review of the manuscript and pointing out mistakes in the original pre-print version.

\addcontentsline{toc}{section}{References}
\bibliographystyle{model1-num-names}
\biboptions{sort,numbers,comma,compress}                 
\bibliography{main.bib}

\appendix

\section{The full statement and proof of Theorem 1}
\newtheorem{innercustomthm}{Theorem}
\newenvironment{customthm}[1]
  {\renewcommand\theinnercustomthm{#1}\innercustomthm}
  {\endinnercustomthm}

\begin{customthm}{1}[Operator learning errors in Bayesian inverse problems]
Assume $\mathcal{U}$ and $\mathcal{M}$ are real-valued separable Hilbert spaces equipped with inner product--induced norm $\norm{\cdot}_{\mathcal{U}}$ and $\norm{\cdot}_{\mathcal{M}}$. Let $\nu_M$ be a probability measure on $\mathcal{M}$, and $\mathcal{F},\operator\in L^p(\mathcal{M},\nu_M;\mathcal{U})$, $p\in[2, \infty]$, where the Bochner space is equipped with the norm
\begin{equation*}
    \norm{\mathcal{G}}_{L^p(\mathcal{M},\nu_M;\mathcal{U})} = 
    \begin{cases}
        \left(\mathbb{E}_{M\sim\nu_M} \left[ \norm{\mathcal{G}(M)}^p_{\mathcal{U}}\right]\right)^{1/p}\, & p\in[1,\infty)\,\\
        \esssup_{M\sim\nu_M} \norm{\mathcal{G}(M)}_{\mathcal{U}}\, & p=\infty\,.
    \end{cases}
\end{equation*}
Assume the observation operator $\bdmc{B}:\mathcal{U}\to\R^{n_y}$ satisfies
\begin{equation*}
\begin{gathered}
    \norm{(\bdmc{B}\circ\mathcal{F})(m)}_2\leq c_B\norm{\mathcal{F}(m)}_{\mathcal{U}}\,,\quad \norm{(\bdmc{B}\circ\operator)(m)}_2\leq \widetilde{c_B}\norm{\mathcal{F}(m)}_{\mathcal{U}}\,\quad\nu_M\text{-a.e.}\,,\\
    \norm{(\bdmc{B}\circ\mathcal{F})(m) - (\bdmc{B}\circ\operator)(m)}_2\leq c_L\norm{\mathcal{F}(m) -\operator(m)}_{\mathcal{U}}\,\quad\nu_M\text{-a.e.}\,.
\end{gathered}
\end{equation*}
Assume the additive noise $\R^{n_y}\ni\boldsymbol{N}\sim\nu_{\boldsymbol{N}} = \mathcal{N}(\boldsymbol{0},\boldsymbol{C}_{\boldsymbol{N}})$ is normally distributed with a probability density function $\pi_{\boldsymbol{N}}(\bx) = ((2\pi)^{n_y}\det(\bC_{\bN}))^{-1/2}\exp(-\frac{1}{2}\bx^T\bC_{\bN}\bx)$, $\bx\in\R^{n_y}$.

For a set of observed data $\boldsymbol{y}^*\in\R^{n_y}$, we have

\begin{equation*}
    \mathcal{E}_{\text{post}}\leq c\norm{\mathcal{E}}_{L^p(\mathcal{M},\nu_M;\mathcal{U})}\,,\quad c >0\,,
\end{equation*}
where $\mathcal{E}_{\text{post}}$ is the Kullback--Leibler divergence between the posterior distribution defined by $\mathcal{F}$ and $\operator$.
\begin{equation*}
    \mathcal{E}_{\text{post}} \coloneqq D_{KL}(\widetilde{\nu_{M|\boldsymbol{Y}}}(\cdot|\boldsymbol{y}^*)||\nu_{M|\boldsymbol{Y}}(\cdot|\boldsymbol{y}^*))\,,\quad
    \mathcal{E}(m) \coloneqq \mathcal{F}(m) -\operator(m)\quad \nu_M\text{-a.e.}\,,
\end{equation*}
and the divergence is governed by Bayes' rule
\begin{align*}
    \frac{\dd\nu_{M|\bY}(\cdot|\by^*)}{\dd\nu_{M}}(m) &= \frac{1}{Z(\by^*)}\underbrace{\pi_{\noise}\left(\by^*-(\bdmc{B}\circ\mathcal{F})(m)\right)}_{\eqqcolon\mathcal{L}(m;\boldsymbol{y}^*)}\quad\as\,,\quad Z(\boldsymbol{y}^*) = \mathbb{E}_{M\sim \nu_M}[\mathcal{L}(M;\by^*)]\,,\\
    \frac{\dd\widetilde{\nu_{M|\bY}}(\cdot|\by^*)}{\dd\nu_{M}}(m) &= \frac{1}{\widetilde{Z}(\by^*)}\underbrace{\pi_{\noise}\left(\by^*-(\bdmc{B}\circ\operator)(m)\right)}_{\eqqcolon\widetilde{\mathcal{L}}(m;\boldsymbol{y}^*)}\quad\as\,,\quad \widetilde{Z}(\boldsymbol{y}^*) = \mathbb{E}_{M\sim \nu_M}[\widetilde{\mathcal{L}}(M;\by^*)]\,,
\end{align*}
where we assume $\post{\cdot}\sim\postt{\cdot}\sim\nu_M{\cdot}$, meaning that they are mutually absolutely continuous.
\end{customthm}

\begin{proof}
Let us first dissect $\mathcal{E}_{\text{post}}$ into two parts:
\begin{align*}
    \mathcal{E}_{\text{post}} &= \mathbb{E}_{M\sim\postt{\cdot}}\left[\ln\left(\frac{\dd \postt{\cdot}}{\dd\post{\cdot}}(M)\right)\right] && \text{(Def. of KL divergence)} \notag\\
    & = \mathbb{E}_{M\sim\nu_M}\left[\ln\left(\frac{\dd \postt{\cdot}}{\dd\post{\cdot}}(M)\right)\frac{1}{\widetilde{Z}(\by^*)}\widetilde{\mathcal{L}}(M;\by^*)\right] && \text{(Change of var. formula)} \notag\\
    & = \mathbb{E}_{M\sim\nu_M}\left[\ln\left(\frac{\dd \postt{\cdot}}{\dd\nu_M}(M)\frac{\dd\nu_M}{\dd\post{\cdot}}(M)\right)\frac{1}{\widetilde{Z}(\by^*)}\widetilde{\mathcal{L}}(M;\by^*)\right] && \text{(Mutually abs. continuous)}\notag\\
    & = \mathbb{E}_{M\sim\nu_M}\left[\ln\left(\frac{Z(\by^*)}{\widetilde{Z}(\by^*)}\frac{\widetilde{\mathcal{L}}(M;\by^*)}{\mathcal{L}(M;\by^*)}\right)\frac{1}{\widetilde{Z}(\by^*)}\widetilde{\mathcal{L}}(M;\by^*)\right] && \text{(Bayes' rule)}\notag\\
    & = \underbrace{\ln\left(\frac{Z(\by^*)}{\widetilde{Z}(\by^*)}\right)}_{\tcircle{A}} + \underbrace{\frac{1}{\widetilde{Z}(\by^*)}\mathbb{E}_{M\sim\nu_M}\left[\ln\left(\frac{\widetilde{\mathcal{L}}(M;\by^*)}{\mathcal{L}(M;\by^*)}\right)\widetilde{\mathcal{L}}(M;\by^*)\right]}_{\tcircle{B}}\,. && \text{(Def. of model evidence)}\notag
\end{align*}

The term involving the likelihoods can be bounded by
\begin{align*}
    \tcircle{B}  &= \frac{1}{\widetilde{Z}(\by^*)}\mathbb{E}_{M\sim\nu_M}\left[\left(\Phi(M)-\widetilde{\Phi}(M)\right)\widetilde{\mathcal{L}}(M;\by^*)\right]&& \left(\begin{cases}
        \Phi(m) = \frac{1}{2}\norm{\by^* - (\bdmc{B}\circ\mathcal{F})(m)}^2_{\bC_{\bN}^{-1}}\\
        \widetilde{\Phi}(m) = \frac{1}{2}\norm{\by^* - (\bdmc{B}\circ\operator)(m)}^2_{\bC_{\bN}^{-1}}
    \end{cases}\right)\\
    & \leq \frac{1}{\widetilde{Z}(\by^*)} \mathbb{E}_{M\sim\nu_M}\left[\left|\Phi(M)-\widetilde{\Phi}(M)\right|\widetilde{\mathcal{L}}(M;\by^*)\right]&& \text{(Jensen's ineq.)}\\
    & \leq \frac{\norm{\widetilde{\mathcal{L}}(\cdot;\by^*)}_{L^{q^*}(\mathcal{M,\nu_M})}}{\widetilde{Z}(\by^*)}\norm{\Phi-\widetilde{\Phi}}_{L^{p^*}(\mathcal{M},\nu_M)}\,. && \left(\text{H\"older's ineq., } q^* \text{ is the conj. exp. to } p^*\in[1,\infty]\right)
\end{align*}
We make several remarks on the last inequality. First, The validity of the inequality is conditional upon whether we can bound $\norm{\Phi-\widetilde{\Phi}}_{L^{p^*}(\mathcal{M},\nu_M)}$ for some $p^*\in[1,\infty]$. Second, if such a $p^*$ exists, the fraction term is larger than one if $p^*\in[1,\infty)$ and equal to one for $p^*=\infty$, as the following inequalities hold for any $q\in[1,\infty]$ due to inclusion of Lebesgue spaces defined on domains with finite measures ($\nu_M(\mathcal{M}) = 1$):
\begin{align*}
    \widetilde{Z}(\by^*) = \norm{\widetilde{\mathcal{L}}(\cdot;\by^*)}_{L^1(\mathcal{M},\nu_M)}\leq\norm{\widetilde{\mathcal{L}}(\cdot;\by^*)}_{L^{q}(\mathcal{M},\nu_M)}\leq \left((2\pi)^{n_y}\det(\bC_{\bN})\right)^{-1/2}\,.
\end{align*}
Now we seek to
\begin{enumerate}
    \item bound $\mathcal{E}_P(m)\coloneqq\Phi(m) - \widetilde{\Phi}(m)$, $\nu_M$-a.e., from above for some $p^*$, and
    \item bound $\widetilde{Z}(\by^*)$ from below.
\end{enumerate}
From now on, we assume $p\in[2,\infty)$ as the extensions to $p = \infty$ is straightforward.

First, we examine $\bdmc{B}\circ\mathcal{F}$, $\bdmc{B}\circ\operator$, and $\bdmc{E}_B(m)\coloneqq(\bdmc{B}\circ\mathcal{F})(m)-(\bdmc{B}\circ\operator)(m)$, $\nu_M$-a.e.,
\begin{align*}
    \norm{\bdmc{B}\circ\mathcal{F}}^p_{L^p(\mathcal{M},\nu_M;\R^{n_y})} &= \mathbb{E}_{M\sim \nu_M}\left[\norm{(\bdmc{B}\circ\mathcal{F})(M)}_2^p\right] &&\text{(Def. of Bochner space norm)}\\
    & \leq c_B^p\mathbb{E}_{M\sim \nu_M}\left[\norm{\mathcal{F}(M)}_\mathcal{U}^p\right] &&\text{(Property of $\bdmc{B}$)}\\
    & = c_B^p\norm{\mathcal{F}}^p_{L^p(\mathcal{M},\nu_M;\mathcal{U})}<\infty\,. &&\text{(Def. of Bochner space norm)}
\end{align*}
We thus have $\bdmc{B}\circ\mathcal{F},\bdmc{B}\circ\operator\in L^p(\mathcal{M},\nu_M;\R^{n_y})$. Similarly,
\begin{align*}
    \norm{\bdmc{E}_B}^p_{L^p(\mathcal{M},\nu_M;\R^{n_y})} &= \mathbb{E}_{M\sim \nu_M}\left[\norm{(\bdmc{B}\circ\mathcal{F})(M)-(\bdmc{B}\circ\operator)(M)}_2^p\right] &&\text{(Def. of Bochner space norm)}\\
    & \leq c_L^p\mathbb{E}_{M\sim \nu_M}\left[\norm{\mathcal{F}(M) - \operator(M)}_\mathcal{U}^p\right] &&\text{(Property of $\bdmc{B}$)}\\
    & = c_L^p\norm{\mathcal{E}}^p_{L^p(\mathcal{M},\nu_M;\mathcal{U})} \,. &&\text{(Def. of Bochner space norm)}
\end{align*}
Now we have
\begin{align*}
    \norm{\mathcal{E}_P}_{L^{p/2}(\mathcal{M},\nu_M)}^{p/2} &= \mathbb{E}_{M\sim\nu_M}\left[|\Phi(M) - \widetilde{\Phi}(M)|^{p/2}\right]\\
    & = \mathbb{E}_{M\sim\nu_M}\left[\left|\left(\frac{1}{2}(\bdmc{B}\circ\mathcal{F})(M) + \frac{1}{2}(\bdmc{B}\circ\operator_{\bw})(M) -\by^*\right)^T\boldsymbol{C}_{\boldsymbol{N}}^{-1}\bdmc{E}_B(M)\right|^{p/2}\right]\\
    &\leq \mathbb{E}_{M\sim\nu_M}\left[\norm{\frac{1}{2}\boldsymbol{C}_{\boldsymbol{N}}^{-1}\left((\bdmc{B}\circ\mathcal{F})(M) + (\bdmc{B}\circ\operator_{\bw})(M) -2\by^*\right)}_2^{p/2}\norm{\bdmc{E}_B(M)}_2^{p/2}\right]\\
    &\leq {\underbrace{\mathbb{E}_{M\sim\nu_M}\left[\norm{\frac{1}{2}\boldsymbol{C}_{\boldsymbol{N}}^{-1}((\bdmc{B}\circ\mathcal{F})(M) + (\bdmc{B}\circ\operator_{\bw})(M) -2\by^*)}_2^{p}\right]}_{\displaystyle c_1^p}}^{1/2}\mathbb{E}_{M\sim\nu_M}\left[\norm{\bdmc{E}_B(M)}_2^{p}\right]^{1/2}\,.
\end{align*}
Where we apply the Cauchy–Schwarz inequality for the two inequalities above. The first expectation above is bounded by
\begin{align*}
    c_1^p &\leq \frac{1}{2}\norm{\bC_{\bN}^{-1}}^p_{2}\left(\norm{\bdmc{B}\circ\mathcal{F}}_{L^{p}(\mathcal{M},\nu_M;\R^{n_y})} + \norm{\bdmc{B}\circ\operator}_{L^{p}(\mathcal{M},\nu_M;\R^{n_y})} + 2\norm{\by^*}_2\right)^p&& \text{(Minkowski ineq.)}\\
    &\leq \frac{1}{2}\norm{\bC_{\bN}^{-1}}^p_{2}\left(c_B\norm{\mathcal{F}}_{L^p(\mathcal{M},\nu_M;\mathcal{U})} + \widetilde{c_B}\norm{\operator}_{L^p(\mathcal{M},\nu_M;\mathcal{U})} + 2\norm{\by^*}_2\right)^p< \infty \,.
\end{align*}
Consequently
\begin{align*}
    \norm{\mathcal{E}_P}_{L^{p^*}(\mathcal{M},\nu_M)} &\leq c_1\norm{\bdmc{E}_B}_{L^p(\mathcal{M},\nu_M;\R^{n_y})}\leq c_1c_L\norm{\mathcal{E}}_{L^p(\mathcal{M},\nu_M;\mathcal{U})}<\infty\,,\quad p^*\in [1, p/2]\,.
\end{align*}
Second, we examine the lower bound for $\widetilde{Z}(\by^*)$:
\begin{align}
    c_2^{-1}\coloneqq\left((2\pi)^{n_y}\det(\bC_{\bN})\right)^{1/2}\widetilde{Z}(\by^*) &= \mathbb{E}_{M\sim\nu_M}\left[\exp(-\widetilde{\Phi}(M))\right] \notag\\
    &\geq \exp\left(-\mathbb{E}_{M\sim\nu_M}\left[\widetilde{\Phi}(M)\cdot 1\right]\right) && \text{(Jen. ineq.)} \notag\\
    & \geq \exp\left(-\frac{1}{2}\mathbb{E}_{M\sim\nu_M}\left[\norm{\by^* - (\bdmc{B}\circ\operator)(M)}_{\bC_{\bN}^{-1}}^p\right]^{1/p}\right) && \text{(H\"od. ineq.)} \notag\\
    & \geq \exp\left(- \frac{1}{2}\norm{\bC_{\bN}^{-1}}_{2}\left(\norm{\by^*}_2 + \norm{\bdmc{B}\circ\operator}_{L^{p}(\mathcal{M},\nu_M;\R^{n_y})}\right)\right)&& \text{(Min. ineq.)} \notag\\
    &\geq \exp\left(- \frac{1}{2}\norm{\bC_{\bN}^{-1}}_{2}\left(\norm{\by^*}_2 + \widetilde{c_B}\norm{\operator}_{L^{p}(\mathcal{M},\nu_M;\mathcal{U})}\right)\right)> 0\,. \label{eq:const_c2}
\end{align}
Therefore, for $q\in[p/(p-2), \infty]$ we have
\begin{align*}
    c_3\coloneqq\frac{\norm{\widetilde{\mathcal{L}}(\cdot;\by^*)}_{L^{q}(\mathcal{M,\nu_M})}}{\widetilde{Z}(\by^*)} \leq \frac{\norm{\widetilde{\mathcal{L}}(\cdot;\by^*)}_{L^{\infty}(\mathcal{M,\nu_M})}}{\widetilde{Z}(\by^*)}\leq c_2 < \infty\,.
\end{align*}
Consequently, the term \tcircle{B} is bounded from above as follows
\begin{equation*}
    \tcircle{B}\leq c_1c_3c_L\norm{\mathcal{E}}_{L^p(\mathcal{M},\nu_M;\mathcal{U})}\,.
\end{equation*}

Following, the term \tcircle{A} involving normalization constants can be bounded by
\begin{align*}
    \tcircle{A} & = \ln\left(1 + \frac{Z(\by^*)-\widetilde{Z}(\by^*)}{\widetilde{Z}(\by^*)}\right)\\
    & \leq \ln\left(1 + \frac{|Z(\by^*)-\widetilde{Z}(\by^*)|}{\widetilde{Z}(\by^*)}\right) && \left(\ln(\cdot) \text{ monotonic incres.}\right)\\
    &\leq \frac{1}{\widetilde{Z}(\by^*)}\left|Z(\by^*) - \widetilde{Z}(\by^*)\right|&& \left(\log(1 + x)\leq x\quad \forall x\geq0\right)\\
    &\leq \frac{1}{\left((2\pi)^{n_y}\det(\bC_{\bN})\right)^{1/2}\widetilde{Z}(\by^*)}\underbrace{\left|\mathbb{E}_{M\sim\nu_M}\left[\exp(-\Phi(m))-\exp(-\widetilde{\Phi}(m))\right]\right|}_{\tcircle{D}} \,.
\end{align*}
Next, we estimate $\tcircle{D}$ as follows
\begin{align*}
    \tcircle{D} &\leq \mathbb{E}_{M\sim \nu_M}\left[|\exp(-\Phi(m))-\exp(-\widetilde{\Phi}(m)|\right] && (\text{Jensen's ineq.})\\
    &\leq \mathbb{E}_{M\sim \nu_M}\left[|\mathcal{E}_P\cdot 1|\right] && (|e^{-x_1} - e^{-x_2}|\leq|x_1-x_2|,\, x_1,x_2\geq0)\\
    &\leq \norm{\mathcal{E}_P}_{L^{p/2}(\mathcal{M},\nu_M)} && \text{(H\"older's ineq.)}\\
    &\leq c_1c_L\norm{\mathcal{E}}_{L^p(\mathcal{M},\nu_M;\mathcal{U})}\,.
\end{align*}
Combining the two inequalities above and noting the definition of a constant $c_2$ in \eqref{eq:const_c2}, we have shown
\begin{equation*}
    \tcircle{A} \leq c_1c_2c_L\norm{\mathcal{E}}_{L^p(\mathcal{M},\nu_M;\mathcal{U})}\,.
\end{equation*}
In conclusion, we have
\begin{equation*}
    \mathcal{E}_{\text{post}} \leq c_1(c_2+c_3)c_L\norm{\mathcal{E}}_{L^p(\mathcal{M},\nu_M;\mathcal{U})} \,.
\end{equation*}

\end{proof}

\section{The full statement of the corollary to the Newton--Kantorovich theorem}

We state the corollary to the Newton--Kantorovich theorem assuming the equivalency between solving the error correction problem in the solution set $\mathcal{V}_u$ and generating one Newton step in $\mathcal{U}_0$ with a given ``lifting" element $u_L\in\mathcal{V}_u$ for the nonlinear equation
\begin{equation*}
\text{Given } m\in\mathcal{M}, \text{ find } v\in \mathcal{U}_0 \text{ such that}\quad \widetilde{\mathcal{R}}(v,m) = 0\,,\quad\widetilde{\mathcal{R}}(v,m)\coloneqq\mathcal{R}(v + u_L,m)\quad\in\mathcal{U}_0^* \,.
\end{equation*}
We thus redefine the forward operator and the neural operator to $\mathcal{F}(\cdot) - u_L$ and $\widetilde{\mathcal{F}}_{\boldsymbol{w}}(\cdot)-u_L - \widetilde{u}^{\perp}$, where $\widetilde{u}^{\perp}\in\mathcal{U}_0^{\perp}$ represent $\widetilde{\mathcal{F}}_{\boldsymbol{w}}(\cdot)-u_L$ projected to $\mathcal{U}_0^{\perp}$. Additionally, we define the space of bounded linear operator between two Banach spaces $\mathcal{X}$ and $\mathcal{Y}$ as $B(\mathcal{X}, \mathcal{Y})$ equipped with the operator norm,
\begin{equation*}
    \norm{\mathcal{G}}_{B(\mathcal{X}, \mathcal{Y})} = \sup_{y\not = 0} \frac{\norm{\mathcal{G}(x)}_{\mathcal{Y}}}{\norm{x}_{\mathcal{X}}}\,.
\end{equation*}
The following result is a direct translation of the Newton--Kantorovich theorem in Banach spaces as stated by Ciarlet~\cite{Ciarlet2013} to our setting.

\begin{corollary}[Global error reduction for residual-based error correction]
Let $\mathcal{U}$ and $\mathcal{M}$ be real-valued separable Hilbert spaces equipped with inner products--induced norm $\norm{\cdot}_{\mathcal{U}}$ and $\norm{\cdot}_{\mathcal{M}}$. Let $\nu_M$ be a probability measure on $\mathcal{M}$ and $\operator\in L^{\infty}(\mathcal{M},\nu_M;\mathcal{U})$. Assume there exists $\mathcal{F}\in L^{\infty}(\mathcal{M},\nu_M;\mathcal{U})$ such that
\begin{equation}
    \mathcal{R}(\mathcal{F}(m),m) \equiv 0\,,\quad \nu_M\text{-a.e.}
\end{equation}
for $\mathcal{R}:\mathcal{U}\times\mathcal{M}\to\mathcal{U}^*$.
Assume for any $m\in\mathcal{M}$, $\nu_M$-a.e., there exists an open set $\mathcal{D}_m\subseteq\mathcal{U}$ with $\operator(m)\in \mathcal{D}_m$ such that $\mathcal{R}(\cdot, m):\mathcal{D}_m\to\mathcal{U}^*$ is differentiable and its derivative at $\operator(m)$, $\delta_u\mathcal{R}(\operator(m),m)\in B(\mathcal{U},\mathcal{U}^*)$, is bijective. Assume that there exists three constants $c_1$, $c_2$, $c_3$ such that
\begin{equation*}
    0<c_1c_2c_3\leq\frac{1}{2} \text{ and } B_r(\operator(m))\subset \mathcal{D}_m,\text{ where } r\coloneqq\frac{1}{c_2c_3}\,,
\end{equation*}
and
\begin{equation*}
    \begin{split}
        \norm{\delta_u \mathcal{R}(\operator(\cdot),\cdot)^{-1}\mathcal{R}(\operator(\cdot),\cdot)}_{L^{\infty}(\mathcal{M},\nu;\mathcal{U})}&\leq c_1\,,\\
        \quad\norm{\delta_u \mathcal{R}(\operator(\cdot),\cdot)^{-1}}_{L^{\infty}(\mathcal{M},\nu_M; B(\mathcal{U}^*, \mathcal{U}))}&\leq c_2\,,\\
        \norm{\delta_u \mathcal{R}(u_1,m) - \delta_u \mathcal{R}(u_2, m) }_{B(\mathcal{U}, \mathcal{U}^*)}&\leq c_3\norm{u_1-u_2}_{\mathcal{U}}\,,\quad\forall u_1, u_2\in B_r(\operator(m))\,\quad \nu_M\text{-a.e.}
    \end{split}
\end{equation*}
Then for any $m\in\mathcal{M}$, $\nu_M$-a.e., $\delta_u\mathcal{R}(u, m)$ is bijective at each $u\in B_r(\operator(m))$ and the sequence $\{u_j\}_{j=0}^{\infty}$ with $u_0=\operator(m)$ defined as
\begin{equation*}
    u_{j+1} = u_j - \delta_u\mathcal{R}(u_j,m)^{-1}\mathcal{R}(u_j,m)\,,\quad k\geq 0\,,
\end{equation*}
is contained in the ball $B_{r_-}(\operator(m))$, where
\begin{equation*}
    r_-\coloneqq \frac{1-\sqrt{1-2c_1c_2c_3}}{c_2c_3}\leq r\,,
\end{equation*}
and converges to $\mathcal{F}(m)$. Besides, for each $j\geq 0$,
\begin{equation*}
\begin{dcases}
    \norm{u_j-\mathcal{F}(\cdot)}_{\mathcal{L^\infty}(\mathcal{M},\nu_M;\mathcal{U})} \leq \frac{r}{2^{j}}(\frac{r_-}{r})^{2^{j}}\,, & \quad c_1 < \frac{1}{2c_2c_3}\,,\\
    \norm{u_j-\mathcal{F}(\cdot)}_{\mathcal{L^\infty}(\mathcal{M},\nu_M;\mathcal{U})} \leq \frac{r}{2^{j}}\,, &\quad c_1 = \frac{1}{2c_2c_3}\,.
\end{dcases}
\end{equation*}

\end{corollary}
\end{document}